\documentclass[11pt, oneside]{amsart}
\usepackage{amsmath}
\usepackage{amssymb}
\usepackage{amsthm}
\usepackage{amscd}
\usepackage[all]{xy}

\numberwithin{equation}{section}

\newcounter{AbcT}

\newtheorem {Theorem}    {Theorem}[section]

\newtheorem* {Definition} {Definition} 
\newtheorem* {Problem}    {Problem}
\newtheorem {Problem1}    {Problem}
\newtheorem {Question}    [Theorem]{Question}
\newtheorem {Example}    [Theorem]{Example}
\newtheorem {Lemma}      [Theorem]    {Lemma}
\newtheorem {Corollary}  [Theorem]    {Corollary}
\newtheorem {Proposition}[Theorem]    {Proposition}

\newtheorem* {TheoremGolod}    {Theorem~\ref{thm:GB}}
\newtheorem {Observation}[Theorem]    {Observation}

\newtheorem {Conjecture}[Theorem]    {Conjecture}

\newcounter{DM@bibnum}

\newcommand{\la}{\langle}
\newcommand{\ra}{\rangle}

\def\span{{\rm span}}

\def\log{{\rm log\,}}
\def\deg{{\rm deg\,}}

\def\Gal{{\rm Gal\,}}
\def\Ker{{\rm Ker\,}}
\def\Im{{\rm Im\,}}

\def\NSL_2{{\mathcal N SL_2}}

\def\eps{\varepsilon}

\def\lam{\lambda}            
                
\def\phi{\varphi}

\def\hbar{\bar h}

\def\phat{\widehat p}

\def\dbC{{\mathbb C}}

\def\dbF{{\mathbb F}}

\def\dbH{{\mathbb H}}

\def\dbN{{\mathbb N}}

\def\dbQ{{\mathbb Q}}
\def\dbR{{\mathbb R}}

\def\dbZ{{\mathbb Z}}

\def\Fp{{\dbF_p}}

\newcommand{\lla}{\la\!\la}
\newcommand{\rra}{\ra\!\ra}
\def\skv{{\vskip .12cm}}

\begin{document}

\title{Golod-Shafarevich groups: a survey}
\author{Mikhail Ershov}
\address{University of Virginia}

\thanks{Preprint of an article to appear in International Journal of Algebra and Computation 
\copyright {World Scientific Publishing Company} http://www.worldscinet.com/ijac/.}
\thanks{The author is supported in part by the NSF grant DMS-0901703 and 
the Sloan Research Fellowship grant BR 2011-105.}
\email{ershov@virginia.edu}
\date{June 2nd, 2012}
\subjclass[2000]{Primary 20F05, 20F50, Secondary 20E18, 20E07, 20F69, 17B50}
\keywords{Golod-Shafarevich groups, subgroup growth, class field tower problem,
Burnside problem}

\begin{abstract}
In this paper we survey the main results about Golod-Shafarevich groups
and their applications in algebra, number theory and topology.

\end{abstract}
\maketitle
\tableofcontents
\section{Introduction}

\subsection{The discovery of Golod-Shafarevich groups}
Golod-Shafarevich groups have been introduced (or rather discovered) in connection
with the famous class field tower problem, which asks whether 
the class field tower of any number field is finite. This classical number-theoretic
problem, posed by Furtw\"angler in 1925, remained open for almost 40 years, with no clear indication whether the answer should be
positive or negative. By class field theory, the problem is equivalent to the non-existence
of a number field $K$ whose maximal unramified prosolvable extension has infinite degree (over $K$). A convenient way to construct $K$ with the latter property (and thus settle the class field tower problem in the negative)  would be to show that for some prime $p$ the maximal unramified $p$-extension $K_p$ of $K$ has infinite degree, 
or equivalently, the Galois group $G_{K,p}=Gal(K_p/K)$ is infinite (note that $G_{K,p}$ is a pro-$p$ group,
so if finite, it must be a $p$-group). 

A major evidence for the negative answer to the class field tower problem was
given by the 1963 paper of Shafarevich~\cite{Sh}, where the formula for the minimal
number of generators $d(G_{K,p})$ of $G_{K,p}$ and an upper bound for the 
minimal number of relations $r(G_{K,p})$ were established. These results implied
that for any prime $p$, there exists an infinite sequence of number fields $\{K(n)\}$
such that if $G_n=G_{K(n),p}$, then $d(G_{n})\to\infty$ as $n\to\infty$ and 
$r(G_{n})-d(G_{n})$ remains bounded. Shafarevich conjectured that there cannot
be any sequence of finite $p$-groups with these two properties (which would
imply that in the above sequence $G_n$ must be infinite for sufficiently large $n$).
A year later, in 1964, Golod and Shafarevich~\cite{GS} confirmed this conjecture
by showing that for any finite $p$-group $G$ the minimal numbers of generators $d(G)$
and relators $r(G)$  (where $G$ is considered as  a pro-$p$ group) 
are related by the inequality $r(G)>(d(G)-1)^2/4$ (this was improved to $r(G)>d(G)^2/4$ in the subsequent
works of Vinberg~\cite{Vi} and Roquette~\cite{Ro}). 

\subsection{Golod-Shafarevich inequality}

The algebraic tool used to prove that $r(G)>d(G)^2/4$ for a finite $p$-group 
is the so called {\it Golod-Shafarevich inequality}. It can be formulated in many different categories,
including graded (associative) algebras, complete filtered algebras (algebras
defined as quotients of algebras of power series $K\lla u_1,\ldots, u_d\rra$
in non-commuting variables $u_1,\ldots, u_d$), abstract groups and pro-$p$ groups, 
and relates certain growth function of an object in one of these categories with certain data coming from the presentation
of that object by generators and relators. The main consequence of the Golod-Shafarevich 
inequality is that if the set of relators defining an object is ``small'' (in certain weighted sense)
compared to the number of generators, then the object must be infinite in the case of groups
and infinite-dimensional in the case of algebras. Groups and algebras which admit a presentation
with such a ``small'' set of relators are called {\it Golod-Shafarevich}.

A well-known consequence of the Golod-Shafarevich inequality (which is sufficient for
the solution of the class field tower problem) is that a pro-$p$ group $G$ such that $r(G)<d(G)^2/4$ 
must be infinite -- this is an example of what it means for the set of relators to be ``small''.
However, as we already mentioned, the relators are counted with suitable weights, so even 
an infinite set of relators can be ``small''. In particular, it is easy to see that there
exist Golod-Shafarevich abstract groups which are torsion. This result was established
by Golod~\cite{Go} and yielded the first examples of infinite finitely generated torsion groups,
thereby settling in the negative the general Burnside problem. This is the second major
application of the Golod-Shafarevich inequality.

\subsection{Applications in topology}

The majority of works on Golod-Shafarevich groups in 70s and early 80s dealt
with variations and generalizations of the inequality $r(G)>d(G)^2/4$ both
in group-theoretic and number-theoretic contexts, but no really new applications
of Golod-Shafarevich groups were discovered. In 1983, Lubotzky~\cite{Lu1}
made a very important observation that the fundamental groups of 
(finite-volume orientable) hyperbolic $3$-manifolds (which can be equivalently
thought of as torsion-free lattices in $SL_2(\dbC)$) are Golod-Shafarevich
up to finite index. Using this result, in the same paper Lubotzky
solved a major open problem, known at the time as Serre's conjecture, which asserts that arithmetic lattices in $SL_2(\dbC)$ cannot have the congruence
subgroup property. Lubotzky's proof was highly original, and even though
Golod-Shafarevich techniques constituted a relatively small (and technically
not the hardest) part of the argument, it gave hope that other, possibly
more difficult, problems about $3$-manifolds could be settled with the use of 
Golod-Shafarevich groups. This line of research turned out to be quite successful, and even though
no breakthroughs of the magnitude of the proof of Serre's conjecture had been made,
several important new results about hyperbolic 3-manifold groups
had been discovered, including very strong lower bounds on the subgroup
growth of such groups by Lackenby~\cite{La1, La2}. Equally importantly, the potential
applications in topology served as an extra motivation for developing the general structure
theory of Golod-Shafarevich groups, and many interesting (and useful for other purposes)
results in that area were obtained in the past few years.

\subsection{General structure theory of Golod-Shafarevich groups}

The initial applications in the works of Golod and Shafarevich~\cite{GS, Go}
only required a sufficient condition for a group given by generators and relators
to be infinite. However, the groups satisfying that condition (Golod-Shafarevich groups)
turn out to be not only infinite -- they are in fact big in many different ways.
Already the arguments in the original  paper \cite{GS} show that for any Golod-Shafarevich
group $G$ with respect to a prime $p$, the graded algebra associated to its group
algebra $\Fp[G]$ has exponential growth. Combining this result with Lazard's criterion,
Lubotzky observed that Golod-Shafarevich pro-$p$ groups are not $p$-adic analytic -- this
was a key observation in the proof of Serre's conjecture in \cite{Lu1}. In \cite{Wi1},
Wilson proved that every Golod-Shafarevich groups has an infinite torsion quotient,
using a simple modification of Golod's argument~\cite{Go}. Two deeper results
which required essentially new ideas had been established more recently.
In \cite{Ze1}, Zelmanov proved a remarkable theorem asserting that every Golod-Shafarevich
pro-$p$ group contains a non-abelian free pro-$p$ subgroup. This result, clearly very
interesting from a purely group-theoretic point of view, is also important for number
theory since many Galois groups $G_{K,p}$ discussed above are known to be Golod-Shafarevich,
and the fact that these groups have free pro-$p$ subgroups conjecturally implies that 
they do not have faithful linear representations over pro-$p$ rings.
Very recently, in \cite{EJ2}, it was shown that every Golod-Shafarevich abstract group
has an infinite quotient with Kazhdan's property $(T)$, which implies that Golod-Shafarevich
abstract groups cannot be amenable. The proof in \cite{EJ2} was based, among other things, on an earlier 
work \cite{Er}, which established the existence of Golod-Shafarevich groups with property $(T)$.
The latter result, originally obtained as a counterexample to a conjecture of Zelmanov~\cite{Ze2},
turned out to have many other applications in geometric group theory.

\subsection{Counterexamples in group theory}

As we already mentioned, Golod-Shafarevich groups gave the first counterexamples
to the general Burnside problem, which remained open for 60 years. Just a few years
later, Novikov and Adyan \cite{NA} gave a very long and technical proof of the fact 
that free Burnside groups of sufficiently large odd exponent are infinite, thereby
providing the first examples of infinite finitely generated groups of bounded exponent
(and thus solving THE Burnside problem).
Another construction of infinite finitely generated torsion groups, very different
from \cite{GS} and \cite{NA}, was given by Grigorchuk~\cite{Gr} -- these were
also the first examples of groups of intermediate word growth. In addition,
in the 80's, powerful methods had been developed to produce various kinds
of infinite torsion groups with extremely unusual finiteness properties,
starting with Ol'shanskii examples of Tarski monsters~\cite{Ol} and continuing with even
wilder examples constructed using the theory of hyperbolic and relatively hyperbolic
groups (see, e.g.,\cite{Ol2,Os1}). In view of this,
Golod-Shafarevich groups had been somewhat overshadowed as a potential source
of exotic examples. However, in the last few years Golod-Shafarevich groups reappeared
in this context and were used to solve several interesting problems
where generally more powerful techniques from the area of hyperbolic groups
are not applicable. For instance, the existence of Golod-Shafarevich groups
with property $(T)$ yielded the first examples of torsion non-amenable residually
finite groups~\cite{Er}. In \cite{EJ3}, Golod-Shafarevich groups were
used to produce the first residually finite analogues of Tarski monsters.
In \S~\ref{sec:quot}, we will discuss several other applications of this kind as well as a very general technique for discovering new such results.

\subsection{Generalizations, relatives and variations of Golod-Shafarevich groups}

A lot of attention in this paper will be devoted to {\it generalized Golod-Shafarevich groups}
abbreviated as GGS groups (Golod-Shafarevich groups will be abbreviated as GS groups). 
Generalized Golod-Shafarevich groups are defined in the same way as Golod-Shafarevich groups
except that generators are allowed to be counted with different weights. They have been introduced (without any proper name attached) shortly after Golod-Shafarevich groups 
(for instance, they already appear in Koch's book~\cite{Ko} first published in 1970),
and it is easy to extend all the basic properties of GS groups to GGS groups. However, GGS groups have not been used much until recently, when it became clear that the class of GGS groups is more natural in many ways than that of GS groups. In particular,
GGS groups played a key role in the construction of Kazhdan quotients of GS groups~\cite{EJ2}
and residually finite analogues of Tarski monsters~\cite{EJ3}.

Another interesting class of groups which strongly resemble  Golod-Shafarevich groups
both in their definition and their structure was introduced very recently by Schlage-Puchta~\cite{SP} 
and independently by Osin~\cite{Os}. In the terminology of \cite{SP}, these are groups of positive $p$-deficiency
(we will use the term {\it power $p$-deficiency} instead of $p$-deficiency to avoid 
terminological confusion). Perhaps, the most striking fact about these groups is that
they provide by far the most elementary solution to the general Burnside problem
(discovered almost 50 years after the initial counterexamples by Golod). These
groups will be discussed and compared with Golod-Shafarevich groups in \S~\ref{sec:powerdef}.

Finally, we should make a remark about the term `Golod-Shafarevich groups'. Even though these groups
have been studied for almost 50 years, there seemed to be no consensus on what a `Golod-Shafarevich group'
should mean until the last few years, and it was common for this term to have a more
restricted meaning (say, pro-$p$ groups for which $r(G)<d(G)^2/4$ and $d(G)>1$) or
even the opposite meaning (that is, groups which are not Golod-Shafarevich in our terminology). 
It is also common, especially in older papers, to talk about {\it Golod groups} -- these
usually refer to the class of $p$-torsion groups constructed by Golod in \cite{Go}. These groups are 
defined as certain subgroups of Golod-Shafarevich graded algebras, but it is not clear whether
the groups themselves are Golod-Shafarevich since they are not defined directly by
generators and relators. Thus, a general theorem about Golod-Shafarevich groups does not 
formally apply to Golod groups, but in most cases the corresponding result for Golod groups
can be obtained by using essentially the same argument.

\vskip .15cm
{\bf Acknowledgments: } The author is grateful to Andrei Jaikin for useful
discussions and suggestions and to Ashley Rall for carefully
reading the paper, providing useful feedback and proposing Problem~\ref{Problem5} in \S~\ref{sec:questions}. The author would also like to thank Mark Sapir and the anonymous referee for helpful suggestions.

\section{Golod-Shafarevich inequality}
\subsection{Golod-Shafarevich inequality for graded algebras}
\label{subsec:GSIG}
Let $K$ be a field, $U=\{u_1,\ldots, u_d\}$ a finite set, and denote by
$K\la U\ra=K\la u_1,\ldots, u_d\ra$ the free associative $K$-algebra on $U$,
that is, the algebra of polynomials in non-commuting variables $u_1,\ldots, u_d$
with coefficients in $K$. Let $K\la U\ra_n$ be the degree $n$ homogeneous
component of $K\la U\ra$, so that
$$K\la U\ra=\oplus_{n=0}^{\infty} K\la U\ra_n.$$

Let $R$ be a subset of $K\la U\ra$ consisting of homogeneous elements
of positive degree, and let $A$ be the $K$-algebra given by the presentation
$(U, R)$. This means that 
$$A=K\la U\ra/I,$$
where $I$ is the ideal of $K\la U\ra$ generated by $R$.

Note that $I$ a graded ideal, that is, $I=\oplus I_n$ where $I_n=I\cap K\la U\ra_n$,
and $A$ is a graded algebra: $A= \oplus A_n$ where $A_n=K\la U\ra_n/I_n$. 
Let $a_n=\dim_K A_n$.

For each $n\in\dbN$ let $r_n$ be the number of elements of $R$ which have degree $n$.
Since $K\la U\ra_n$ is a finite-dimensional subspace, we can (and will) assume without 
loss of generality that $r_n<\infty$ for each $n\in\dbN$.

The sequences $\{a_n\}$ and $\{r_n\}$ are conveniently encoded by the corresponding
Hilbert series $Hilb_A(t)=\sum_{n=0}^{\infty} a_n t^n$ and $H_R(t)=\sum_{n=1}^{\infty} r_n t^n$.
The following inequality relating these two series was established by Golod and Shafarevich~\cite{GS}.

Given two power series $f(t)=\sum f_n t^n$ and $g(t)=\sum g_n t^n$ in $\dbR[[t]]$, we shall write
$f(t)\geq g(t)$ if $f_n\geq g_n$ for each $n$.

\begin{Theorem}[Golod-Shafarevich inequality: graded case] 
\label{gsi:graded}
In the above setting
we have 
\begin{equation}
\label{gsineq}
(1-|U|t+H_R(t))\cdot Hilb_A(t)\geq 1.
\end{equation}
\end{Theorem}
\begin{proof} Even though the proof of this result appears in several survey articles and books (see, e.g.,  \cite[Section~5]{Ha}),
we present it here as well due to its elegance and importance.

Let $R_n=\{r\in R: \deg(r)=n\}$, so that $R=\sqcup_{n\geq 1}R_n$.
Recall that $I$ is the ideal of $K\la U\ra$ generated by $R$ and
$I=\oplus_{n=1}^{\infty} I_n$ with $I_n\subseteq K\la U\ra_n$. 

Now fix $n\geq 1$.
Since each $r\in R$ is homogeneous, $I_n$ is spanned over $K$
by elements of the form $vrw$ for some $v\in K\la U\ra_s, w\in  K\la U\ra_t$
and $r\in R_m$, where $s+m+t=n$ and $v$ and $w$ are monic monomials.

If $v\neq 1$, then $v=uv'$ for some $u\in U$, so $vrw=uv'rw\in u I_{n-1}$.
If $v=1$, then $vrw=rw\in R_m K\la U\ra_{n-m}$. Hence
$$I_n=\span_K(U) I_{n-1}+\sum_{m=1}^n \span_K(R_m) K\la U\ra_{n-m}.\eqno (***)$$
For each $i\in\dbZ_{\geq 0}$ choose a $K$-subspace $B_i$ of $K\la U\ra_i$
such that $K\la U\ra_i= I_i\oplus B_i$. Then
$\span_K(R_m) K\la U\ra_{n-m}=\span_K(R_m) B_{n-m}+\span_K(R_m) I_{n-m}$
and $\span_K(R_m) I_{n-m}\subseteq \span_K(U) I_{n-1}$.
Combining this observation with (***), we conclude
\begin{equation}
\label{GSikey}
I_n=\span_K(U) I_{n-1}+\sum_{m=1}^n \span_K(R_m) B_{n-m}.
\end{equation}
Let $d=|U|$.
Since $A_i=K\la U\ra_i/I_i$, we have
$a_i=\dim A_i=K\la U\ra_i- \dim I_i=d^i-\dim I_i$, 
and thus $\dim B_i=a_i$. Hence, computing the dimensions of both sides of \eqref{GSikey}, we get
$$d^n-a_n\leq d(d^{n-1}-a_{n-1})+\sum_{m=1}^n r_m a_{n-m},$$
which simplifies to $a_n-da_{n-1}+\sum_{m=1}^n r_m a_{n-m}\geq 0$.
\vskip .1cm

Finally observe that $a_n-da_{n-1}+\sum_{m=1}^n r_m a_{n-m}$
is the coefficient of $t^n$ in the power series 
$(1-dt+H_R(t))\cdot Hilb_A(t)$. The constant term of this power
series is $a_0=1$. Therefore, we proved that
$(1-dt+H_R(t))\cdot Hilb_A(t)\geq 1+\sum_{m=1}^{\infty}0\cdot t^i=1$
as power series, as desired.
\end{proof}

As an immediate consequence of the Golod-Shafarevich inequality one obtains
a sufficient condition for a graded algebra $A$ given by a presentation 
$(U, R)$
to be infinite-dimensional:

\begin{Corollary}
\label{cor:gs1}
Assume that there exists a real number $\tau>0$ s.t.  $1-d\tau+H_R(\tau)\leq 0$
(in particular, we assume that the series $H_R(\tau)$ converges).  Then
\begin{itemize}
\item[(i)] The series $H_A(\tau)$ diverges.
\item[(ii)] Assume in addition that $\tau\in (0,1)$ and $1-d\tau+H_R(\tau)< 0$. Then
the algebra $A$ has exponential growth, that is,
the sequence $a_n=\dim A_n$ grows exponentially. In particular, $A$ is infinite-dimensional.
\end{itemize}
\end{Corollary}
\begin{proof}
(i) Suppose that the series $H_A(\tau)$ converges. Then, if we substitute $t=\tau$
in \eqref{gsineq}, both factors on the left-hand side of \eqref{gsineq} become convergent,
so we should get a valid numerical inequality $H_A(\tau)(1-d\tau+H_R(\tau))\geq 1$.
This cannot happen since clearly $H_A(\tau)>0$, while by assumption $1-d\tau+H_R(\tau)\leq 0$.

(ii) Since the series $H_A(\tau)$ diverges, we must have $\limsup \sqrt[n]{a_n }\geq \frac{1}{\tau}>1$.
On the other hand, since by construction $A$ is generated in degree $1$, the sequence $\{a_n\}$
is submultiplicative (that is, $a_{n+m}\leq a_n a_m$ for all $n,m$), which implies 
that $\lim \sqrt[n]{a_n }$ exists. Therefore,  $\lim \sqrt[n]{a_n }>1$, so $\{a_n\}$ grows
exponentially.
\end{proof}

\begin{Corollary} 
\label{cor:gs2}
Let $A$ be a finite-dimensional graded $K$-algebra, and let $(U, R)$
be a graded presentation of $A$, which is minimal in the sense that no proper subset of $U$
generates $A$. Then
\begin{equation}
\label{eq:gs2}
|R|> |U|^2/4.
\end{equation}
\end{Corollary}
\begin{proof} We first note that $r_1=0$, that is, $R$ has no relators of degree $1$,
since any such relator would allow us to express one of the generators in $U$ as a linear
combination of the others, contradicting the minimality of the presentation 
$( U, R)$.
Therefore, for any $\tau>0$ we have $1-|U|\tau+H_R(\tau)\leq 1-|U|\tau+|R|\tau^2$.

On the other hand, since $A$ is finite-dimensional, by Corollary~\ref{cor:gs1}, for any 
$\tau>0$ we have $1-|U|\tau+H_R(\tau)>0$. Thus, $1-|U|\tau+|R|\tau^2>0$
for any $\tau> 0$, and setting $\tau=2/|U|$, we obtain $|R|> |U|^2/4$.
\end{proof}
\begin{Remark} Minimality of $U$ is actually equivalent to the assumption $r_1=0$.
Note that it is necessary to make this assumption in Corollary~\ref{cor:gs2} -- without it we 
could start with any finite presentation for $A$ and then add any set of ``artificial'' generators $S$ together
with relations $s=0$ for each $s\in S$, violating \eqref{eq:gs2} for sufficiently large $S$.
\end{Remark}
\vskip .15cm
It is a very interesting question whether inequality \eqref{eq:gs2} is optimal.
More generally, fix a field $K$, and given an integer $d\geq 2$, let $f(d)$ be the smallest positive integer
for which there exists a subset $R$ of $K\la u_1,\ldots, u_d\ra$
consisting of homogeneous elements of degree at least $2$, with $|R|=f(d)$, 
such that the $K$-algebra $\la u_1,\ldots, u_d \mid R\ra$ is finite-dimensional.\footnote{It is not known whether the function $f(d)$ depends on $K$.}

Corollary~\ref{cor:gs2} implies that $f(d)>d^2/4$, and there is an obvious upper bound
$f(d)\leq (d^2+d)/2$ yielded by any commutative algebra in which each $u_i$
is nilpotent. That this upper bound is not optimal was immediately realized by
Kostrikin~\cite{Kos} in 1965 who showed that $f(d)$ is at most $(d^2-1)/3 +d$,
at least when $d$ is a power of $2$. \footnote{The paper \cite{Kos}
gives a construction of finite $p$-groups with the corresponding bound
on the number of relators, but its easy modification yields the analogous
result for associative algebras.} 
In 1990,
Wisliceny~\cite{Wis2} found a better upper bound which asymptotically coincides
with the Golod-Shafarevich lower bound: $f(d)\leq \frac{d^2}{4}+\frac{d}{2}$ if $d$ is even
and $f(d)\leq \frac{d^2+1}{4}+\frac{d}{2}$ if $d$ is odd. The final improvement 
was obtained very recently by Iyudu and Shkarin~\cite{IySh} who proved that
$f(d)\leq ]\frac{d^2+d+1}{4}[\,$ (where $]x[$ is the smallest integer greater
than or equal to $x$). Examples of presentations yielding this bound 
(as well as examples from \cite{Wis2}) are of very simple form, with every relation
of the form $u_i u_j=u_k u_l$ for some indices $i,j,k$ and $l$. Algebras given
by presentations of this form are called {\it quadratic semigroup algebras} in \cite{IySh},
where it is proved that the bound $f(d)\leq ]\frac{d^2+d+1}{4}[$ is optimal for this
class of associative algebras.

Another natural problem is when the Golod-Shafarevich inequality~\eqref{gsineq} becomes
an equality. Anick showed that this happens if and only if the set of relators 
$R$ is {\it strongly free} \cite[Theorem~2.6]{An1}, and if we assume that $R$ is
a minimal set of relators, then $R$ is strongly free if and only if the algebra $A=\la U | R\ra$ has global dimension $\leq 2$ (see \cite[Theorem~2.12]{An1}); see also \cite{Pi} for some
related results. An easily verifiable sufficient condition for $R$ to be strongly free is that
the set of (suitably defined) leading terms of elements of $R$ is combinatorially free (see \cite[Theorem~3.2]{An1}).  Stronger results on this problem have been obtained for the class of quadratic algebras, that is, algebras given by a presentation with all relations homogeneous of degree two
-- see \cite{An2, IySh1} and the books \cite{Uf} and \cite{PP}.

\subsection{First applications of the Golod-Shafarevich inequality}

We now discuss two major applications of the Golod-Shafarevich inequality -- 
the (negative) solutions to the Kurosh-Levitzky and the general Burnside problems.

\begin{Problem}[Kurosh-Levitzky] Let $K$ be a field. Is it true that
a finitely generated nil algebra over $K$ must be finite-dimensional 
(and hence nilpotent)?
\end{Problem}

\begin{Theorem}[Golod, \cite{Go, Go2}] 
\label{thm:KL}
Let $K$ be a field and $d\geq 2$
an integer. Then there exists a $d$-generated associative nil algebra
over $K$ which is infinite-dimensional.
\end{Theorem}
\begin{proof} 
We start with the case when $K$ is countable established in \cite{Go}, where the argument
is very simple.
Let $U=\{u_1,\ldots, u_d\}$ and $K\la U\ra^+$ the subset
of $K\la U\ra$ consisting of all polynomials with zero constant term.
Since $K$ is countable, $K\la U\ra^+$ is also countable, so we can enumerate
its elements: $K\la U\ra^+=\{f_1,f_2,\ldots\}$.

Let $\tau\in (1/d,1)$ and $N\in\dbN$ be such that $1-d\tau+\sum_{n=N}^{\infty}\tau^n<0$.
Choose $N_1\geq N$, and write the element $f_1^{N_1}$ as the sum of its homogeneous
components: $f_1^{N_1}=\sum_{i=1}^{k_1}f_{1,i}$. Note that $\deg(f_{1,i})\geq N_1$
since $f_1$ has zero constant term. Next choose $N_2\geq \max\{N_1+1, \{\deg(f_{1,i})\}\}$,
and let $\{f_{2,i}\}_{i=1}^{k_2}$ be the homogeneous components of $f_2^{N_2}$.
Next choose $N_3\geq \max\{N_2+1,\{\deg(f_{2,i})\}\}$ and proceed indefinitely.

Let $R=\{f_{n,i}: n\in\dbN, 1\leq i\leq k_n\}$, consider the algebra
$A=\la U |R \ra$, and let $A^+$ be the image of $K\la U\ra^+$ in $A$.
By construction, the algebra $A^+$ is $d$-generated and nil and has
codimension $1$ in $A$, so we only need to prove that $A$ is infinite-dimensional.
The choice of $\{N_i\}$ ensures that the set of relations $R$ contains
at most one element of each degree and no elements of degree less than $N$.
Therefore, $1-|U|\tau+H_{R}(\tau)\leq 1-d\tau+\sum_{n=N}^{\infty}\tau^n<0$,
so $A$ is infinite-dimensional by Corollary~\ref{cor:gs1}.
\vskip .1cm

{\it General case:} Let $\Omega$ be the set of all non-empty finite sets of 
monic $U$-monomials of positive degree (this is clearly a countable set). Given $\omega\in \Omega$,
let $S_{\omega}$ be the set of all elements of $K\la U\ra^+$ which are representable as 
a $K$-linear combination of elements of $\omega$ and $\omega$ is the smallest set
with this property. Thus, $K\la U\ra^+=\sqcup_{\omega\in\Omega} S_{\omega}$. 

Now fix $\omega=\{m_1,\ldots, m_n\}\in\Omega$. Every element $f\in S_{\omega}$ can be written as a sum $f=\sum_{i=1}^n c_i m_i$  where $c_i$'s are elements of $K$, which for the moment we
treat as commuting formal variables. Given $N\in\dbN$, we have 
$f^N=\sum_{i=1}^{d_{N,\omega}} \mu_{i,\omega,N} (c_1,\ldots, c_n) p_{i,\omega,N}$ where each $\mu_{i,\omega,N} (c_1,\ldots, 
c_n)$ is a monic  monomial in $c_1,\ldots, c_n$ of degree $N$, $d_{N,\omega}$ is the number of distinct such monomials,
and each $p_{i,\omega,N}$ is a polynomial in $U$ with no terms of degree $<N$.
Here are two key observations:
\begin{itemize}
\item[(i)] Fix $\omega\in\Omega$ and $N\in\dbN$. If a set $R$ contains all homogeneous components of $p_{i,\omega,N}$ for each 
$1\leq i\leq d_{N,\omega}$, then the image of every
element of $S_{\omega}$ in the algebra $A=\la U|R \ra$ will be nilpotent.
\item[(ii)] The sequence $d_{N,\omega}$ grows polynomially in $N$ (as $\omega$ stays fixed) since
$c_i$'s commute with each other.
\end{itemize}

Now fix $\tau\in (1/d,1)$, for each $\omega\in\Omega$ choose $N_{\omega}\in\dbN$, and let
$R$ be the set of all (nonzero) homogeneous components of the polynomials
$p_{i,\omega,N_{\omega}}$ with $\omega\in\Omega$ and $1\leq i\leq d_{N_{\omega},\omega}$. 
Let $A=\la U|R \ra$ and $A^+$ the image of $K\la U\ra^+$ in $A$. 

Property (i) ensures that $A^+$ is nil. Since homogeneous components of
$p_{i,\omega,N_{\omega}}$ have degree $\geq N_{\omega}$, we have
$$H_R(\tau)\leq \sum\limits_{\omega\in\Omega} d_{N_{\omega},\omega}\sum_{j=N_{\omega}}^{\infty}\tau^{j}=
\sum\limits_{\omega\in\Omega} d_{N_{\omega},\omega}\frac{\tau^{N_{\omega}}}{1-\tau}.$$ Property (ii) ensures
that each term of this sum can be made arbitrarily small by choosing sufficiently large $N_{\omega}$
(independently of other terms). In particular, we can make sure that $1-d\tau+H_R (\tau)<0$, so
that $A$ is infinite-dimensional, and we are done as in the case of countable $K$.
\end{proof}

Once Theorem~\ref{thm:KL} is proved, it is very easy to show that the general Burnside problem
also has negative solution.

\begin{Theorem}[Golod, \cite{Go}] 
\label{thm:GB}
For every prime $p$ and integer $d\geq 2$
there exists an infinite $d$-generated $p$-torsion group.
\end{Theorem}
\begin{proof} Let $K=\Fp$, the finite field of order $p$, and
let $A^+$ be the algebra constructed in the proof of Theorem~\ref{thm:KL}.
Since $A^+$ is nil and $K$ has characteristic $p$, the set $1+A^{+}=\{1+a: a\in A^+\}$
is a $p$-torsion group. We claim that the subgroup $\Gamma$ of $1+A^+$ generated
by $1+u_1,\ldots, 1+u_d$ is infinite (and hence satisfies all the required properties).
The inclusion map $\iota:\Gamma\to 1+A^+$ induces a homomorphism $\iota_*:\Fp[\Gamma]\to A$
(where $\Fp[\Gamma]$ is the $\Fp$-group algebra of $\Gamma$). 
The image of $\iota_*$ contains $1$ and $1+u_i$ for each $i$ (hence also $u_i$ for each $i$),
so $\iota_*$ is surjective. Since $A$ is infinite-dimensional, so must be $\Fp[\Gamma]$,
whence $\Gamma$ is infinite.
\end{proof}

Note that while the above construction of infinite finitely generated $p$-torsion groups
uses very elementary tools, it has one disadvantage as we have no control over
presentations of $\Gamma$ by generators and relators. This problem will be addressed
in the next section where we will provide an alternative version of Golod's construction 
(which uses Golod-Shafarevich groups), based on a more general form of the Golod-Shafarevich inequality.

\subsection{Golod-Shafarevich inequality for complete filtered algebras} 
\label{subsec:GSIF}
In order to define and study Golod-Shafarevich groups, one needs a more general
version of the Golod-Shafarevich inequality dealing with complete filtered algebras.
Below we shall essentially repeat the setup of \S~\ref{subsec:GSIG} with two
changes: polynomials are replaced by power series and relators are allowed to be 
non-homogeneous.

As in \S~\ref{subsec:GSIG}, we fix a finite set $U=\{u_1,\ldots,u_d\}$, a field $K$, and let $K\lla U\rra=K\lla u_1,\ldots, u_d\rra$
be the algebra of power series over $K$ in non-commuting variables $u_1,\ldots, u_d$.
As usual, given $f\in K\lla U\rra$, we define $\deg(f)$ to be the smallest length of 
a monomial in $U$ which appears in $f$ with nonzero coefficient. For convenience we also
set $\deg(0)=\infty$. Let $K\lla U\rra_n=\{f\in K\lla U\rra : \deg(f)\geq n\}$.
The sets $\{K\lla U\rra_n\}_{n\in\dbN}$ form a base of neighborhoods of $0$
for the natural degree topology on $K\lla U\rra$.

Let $I$ be a closed ideal of $K\lla U\rra_1$, let $R\subseteq I$ be a subset which generates $I$ as a (closed) 
ideal, and let $r_n=|\{r\in R: \deg(r)=n\}|$. As before, without loss of generality, we can assume that $r_n<\infty$
for each $n$.

Let $A=K\lla U\rra/I$. Note that $A$ is no longer a graded algebra, but it has a natural descending filtration
$\{A_n\}_{n\geq 0}$ where $A_n=\pi(K\lla U\rra_n)$ and $\pi:K\lla U\rra\to A$ is the natural
projection. Since $A$ is also complete with respect to topology determined by the filtration
$\{A_n\}$, we will refer to such algebras as {\it complete filtered algebras}.

Define $a_n=\dim_K A_{n}/A_{n+1}$, and as in the graded case consider
the Hilbert series $Hilb_A(t)=\sum_{n=0}^{\infty} a_n t^n$ and $H_R(t)=\sum_{n=1}^{\infty} r_n t^n$.

\begin{Theorem}[Golod-Shafarevich inequality: general case]
\label{gsi:graded2}
In the above setting
we have 
\begin{equation}
\frac{(1-|U|t+H_R(t))\cdot Hilb_A(t)}{1-t}\geq \frac{1}{1-t}.
\label{gsineq2}
\end{equation}
\end{Theorem}
Inequality \eqref{gsineq2} was first proved by Vinberg~\cite{Vi} (see inequality (12) on p.212).
\footnote{Formally, \cite{Vi} deals with polynomials and not power series, but this makes no difference as explained
in the paragraph just before Theorem~2 in \cite{Vi}. Similarly, the extra assumption $r_1=0$ made
in \cite{Vi} is not essential for the proof.}
Its proof is similar to (but more technical than) the one in the graded case.
This time it is not possible to give a good lower bound for $a_n$, but one can still give a bound for
$\dim A/A_{n+1}=a_0+\ldots+a_n$ which is the coefficient of $t^n$ in the power series
$\frac{Hilb_A(t)}{1-t}$. 

Also note that inequality~\eqref{gsineq2} follows from the one we had in the 
graded case by multiplying both sides by the power series $\frac{1}{1-t}$ with positive coefficients.
While inequality \eqref{gsineq2} is weaker than \eqref{gsineq}, all the key consequences
of \eqref{gsineq} established earlier in this section remain true in this setting:

\begin{Proposition} Let $A$ be a complete filtered algebra given by a presentation $( U, R )$.
\label{gs:filtered}
\begin{itemize}
\item[(a)] Assume that there exists a real number $\tau\in (0,1)$ s.t.  
$1-|U|\tau+H_R(\tau)\leq 0$.  
Then $A$ is infinite-dimensional and $Hilb_A(\tau)$ diverges.
If $1-|U|\tau+H_R(\tau)< 0$, the sequence $\{a_n\}$
(defined above) has exponential growth.
\item[(b)] If $A$ is finite-dimensional and $U$ is minimal, then $|R|>|U|^2/4$.
\end{itemize}
\end{Proposition}
\begin{Definition}\rm  $\empty$
\begin{itemize}
\item[(i)] A presentation $(U,R)$ in the category of complete filtered algebras
will be said to satisfy the {\it Golod-Shafarevich (GS) condition} if $1-|U|\tau+H_R(\tau)< 0$ 
for some $\tau\in (0,1)$. 
\item[(ii)] A complete filtered algebra $A$ will be called {\it Golod-Shafarevich}
if it has a presentation satisfying the GS condition. 
\end{itemize}
\end{Definition}

\section{Golod-Shafarevich groups}
\label{sec:GSG}

\subsection{Definition of Golod-Shafarevich groups}
\label{subsec:GSGdef}

Fix a prime number $p$, and let $G$ be a finitely generated pro-$p$ group. 
Recall that 
$G=\varprojlim_{N\in \Omega_p(G)} G/N,$
where $\Omega_p(G)$ is the set of open normal subgroups
of $G$ (all of which have $p$-power index). 
We shall be interested in the completed group algebra $\Fp[[G]]$ which is defined
as the corresponding inverse limit of $\Fp$-group algebras:
$$\Fp[[G]]=\varprojlim_{N\in \Omega_p(G)} \Fp[G/N].$$
\vskip .2cm
Suppose now that $G$ is given in the form $G=F/\la R\ra^F$ where $F$ is a finitely
generated free pro-$p$ group, $R$ is a subset of $F$ (and $\la R\ra^F$ is the closed normal
subgroup of $F$ generated by $R$). Then it is easy to see that there is a natural isomorphism
$\Fp[[G]]\cong \Fp[[F]]/I_R$ where $I_R$ is the closed ideal of $\Fp[[F]]$ generated by 
the set $\{r-1 : r\in R\}$.

Let $X=\{x_1,\ldots, x_d\}$ be a free generating set of $F$. By a theorem of Lazard~\cite{Laz}, 
the completed group algebra $\Fp[[F]]$ is isomorphic to the algebra of power series
$\Fp\lla u_1,\ldots, u_d\rra$ under the map $x_i\mapsto 1+u_i$. Note that this map
yields an embedding of $F$ into $\Fp\lla u_1,\ldots, u_d\rra^{\times}$, the multiplicative
group of $\Fp\lla u_1,\ldots, u_d\rra$. This embedding is called the {\it Magnus embedding}
(it was initially established by Magnus~\cite{Ma} in the case of free abstract groups).
\skv

The bottom line of the above discussion is that given a presentation 
$(X, R)$ of a pro-$p$ group $G$, there is a corresponding presentation
for the completed group algebra $\Fp[[G]]$ as a quotient of $\Fp\lla U\rra$ (with $|U|=|X|$).
A pro-$p$ group $G$ will be called {\it Golod-Shafarevich} if it has a presentation such that the corresponding presentation of $\Fp[[G]]$ satisfies the GS condition.

\begin{Definition}\rm Let $X=\{x_1,\ldots, x_d\}$ and $U=\{u_1,\ldots, u_d\}$ be finite sets
of the same cardinality. Let $F=F_{\phat}(X)$ the free pro-$p$ group on $X$,
and let $\iota: F\to \Fp\lla U\rra^{\times}$ be the Magnus embedding.
Define the degree function $D: F\to \dbN\cup\{\infty\}$ by
$$D(f)=\deg(\iota(f)-1),$$ 
where $\deg$ is the usual degree of a power series in $u_1,\ldots, u_d$.
\end{Definition}\rm

\begin{Definition}\rm $\empty$
\begin{itemize}
\item[(i)] A (pro-$p$) presentation $(X, R)$ is said to satisfy the {\it Golod-Shafarevich (GS)
condition} if there exists $\tau\in (0,1)$ such that $1-|X|\tau+H_R(\tau)<0$ where
$H_R(t)=\sum_{r\in R}t^{D(r)}$.
\item[(ii)] A pro-$p$ group $G$ is called a {\it Golod-Shafarevich (GS)} group if
it has a presentation satisfying the GS condition.
\item[(iii)] An abstract group $G$ is called a {\it Golod-Shafarevich} group (with respect to
$p$) if its pro-$p$ completion $G_{\phat}$ is Golod-Shafarevich.
\end{itemize}
\end{Definition}

\begin{Remark} It is more common to call an abstract group $G$ Golod-Shafarevich
if it has an {\it abstract} presentation $(X, R)$ s.t. $1-|X|\tau+H_R(\tau)<0$
for some $\tau\in (0,1)$. This condition is certainly sufficient for $G$ to be Golod-Shafarevich
in our sense since if an abstract group $G$ is given by a presentation 
$(X, R)$, then its pro-$p$ completion $G_{\phat}$ is given by the same presentation
$(X, R)$, considered as a pro-$p$ presentation (see, e.g., \cite[Lemma~2.1]{Lu1}). 
To the best of our knowledge, it is an open question whether these two definitions 
of Golod-Shafarevich abstract groups are equivalent. The advantage of our definition 
is that an abstract group $G$ is Golod-Shafarevich if and only if the image of $G$ in its pro-$p$ completion is Golod-Shafarevich.
\end{Remark}

\begin{Theorem}[Golod-Shafarevich] 
\label{GSinf}
Golod-Shafarevich groups are infinite.
\end{Theorem}
\begin{proof} If $G$ is a Golod-Shafarevich pro-$p$ group, then by construction $\Fp[[G]]$ is a Golod-Shafarevich
algebra, hence infinite. This implies that $G$ has infinitely many open subgroups, so
$G$ must be infinite. If $G$ is a Golod-Shafarevich abstract group, its pro-$p$ completion
is infinite, as we just argued, so $G$ itself must be infinite.
\end{proof}

Before discussing applications of Golod-Shafarevich groups, we remark
that the degree function $D$ used above can also be described 
in terms of the Zassenhaus filtration which makes perfect sense
in arbitrary (not just free) groups.

\begin{Definition}\rm Let $G$ be a finitely generated abstract (resp. pro-$p$) group.
Let $M$ be the augmentation ideal of the group algebra $\Fp[G]$ (resp. completed group
algebra $\Fp[[G]]$), that is, $M$ is the ideal generated by the set $\{g-1: g\in G\}$.
For each $n\in\dbN$ let $D_n G=\{g\in G: g-1\in M^n\}$.
The series $\{D_n G\}$ is called the {Zassenhaus filtration} of $G$. 
\end{Definition}

If $F$ is a finitely generated free pro-$p$ group and $f\in F\setminus\{1\}$, it is easy to see
that $D(f)=n$ if and only if $f\in D_n F\setminus D_{n+1}F$. In particular,
this shows that the function $D$ does not depend on the choice of a free generating set $X$ of $F$.

It is well known (see, e.g., \cite[Ch. 11,12]{DDMS}) that the terms of Zassenhaus filtration can also be defined
as verbal subgroups:

\begin{Proposition} 
\label{prop:Zass}
The Zassenhaus filtration $\{D_n G\}$ can be alternatively
defined by $$D_n G=\prod_{i\cdot p^j\geq n}(\gamma_i G)^{p^j}.$$
\end{Proposition}

\subsection{First applications of Golod-Shafarevich groups}

As an immediate consequence of Theorem~\ref{GSinf}, we can construct
infinite finitely generated torsion groups, which are explicitly given
by generators and relators. The argument below is very similar to
(and actually simpler than) the one used in the solution to the Kurosh-Levitzky problem
over countable fields. To the best of our knowledge, this argument was first
used by Wilson in \cite{Wi1}.

\begin{TheoremGolod}[Golod]
\label{GS_first}
For every prime $p$ and integer $d\geq 2$
there exists an infinite $d$-generated $p$-torsion group.
 \end{TheoremGolod}
\begin{proof}[Second proof of Theorem~\ref{thm:GB}] 
Let $X$ be any finite set with $|X|=d$ and $F=F(X)$ the free group on $X$.
Since $F$ is countable, we can enumerate its elements $F=\{f_1,f_2, \ldots\}$. Now take
a sequence of integers $n_1, n_2,\ldots$, and consider a group $G=\la X|R \ra$ where
$R=\{f_1^{p^{n_1}},f_2^{p^{n_2}}, \ldots \}$. By construction, $G$ is $p$-torsion.

It is easy to see that $D(f^{p^k})=D(f)^{p^k}$ for any $f\in F$, so for any $\tau\in (0,1)$ we have
$$1-|X|\tau+H_R(\tau)\leq 1-|X|\tau+\sum_{i=1}^{\infty}\tau^{p^{n_i}}.$$
Now fix $\tau\in (1/|X|,1)$. Then $1-|X|\tau<0$, so we can choose the sequence
$\{n_i\}$ such that $\sum_{i=1}^{\infty}\tau^{p^{n_i}}<-(1-|X|\tau)$. Then $G$
will be Golod-Shafarevich and therefore infinite.
\end{proof}

The following stronger version of Theorem~\ref{GS_first}, also due to Golod,
was announced in \cite{Go} and proved in \cite{Go2}.

\begin{Theorem}[Golod]
\label{GS_second}
For every prime $p$ and integer $d\geq 2$ there exists an infinite $d$-generated $p$-torsion group in which every
$(d-1)$-generated subgroup if finite. 
\end{Theorem}
\begin{Remark} Theorem~\ref{GS_second} was deduced in \cite{Go2}
from the corresponding result for graded algebras (which can also
be found in the book by Kargapolov and Merzlyakov~\cite{KaMe})
in essentially the same way Theorem~\ref{thm:GB} follows from Theorem~\ref{thm:KL}. 
Below we give a ``direct'' group-theoretic proof of this result, generalizing the
argument in the second proof of Theorem~\ref{thm:GB}. 
\end{Remark}

\begin{proof}[Proof of Theorem~\ref{GS_second}] Take any set $X$ with $|X|=d$, and construct a $p$-torsion
Golod-Shafarevich group $G=\la X|R\ra$ as in the second proof of Theorem~\ref{thm:GB}
with the extra requirement that $1-|X|\tau+H_R(\tau)<0$ for some $\tau\in (\frac{1}{d},\frac{1}{d-1})$.
Let $\eps=-(1-|X|\tau+H_R(\tau))$ and $\delta=\tau (d-1)$. Since $\delta<1$, we can find an integer
sequence $\{m_i\}$ such that $\sum_{i=1}^{\infty}\delta^{m_i}<\eps$.

Now let $\Omega=\{\omega^{(1)},\omega^{(2)},\ldots\}$ be the set of all ordered $(d-1)$-tuples of elements of $F(X)$ (the free group on $X$) listed in some order. If $\omega^{(i)}=(f_1^{(i)},\ldots,f_{d-1}^{(i)})$, let $R_i$ be the set of all
left-normed commutators of length $m_i$ involving $f_1^{(i)},\ldots, f_{d-1}^{(i)}$.

We claim that the group $G'=\la X|R\cup\bigcup_{i\geq 1} R_i\ra$ has the required properties.
By construction, every $(d-1)$-generated subgroup of $G'$ is nilpotent and $p$-torsion (since $G'$ is a quotient of $G$)
and hence finite.
If $c=[y_1,\ldots, y_m]$ is a left-normed commutator of length $m$, then $D(c)\geq D(y_1)+\ldots+D(y_m)$.
Therefore, if $S_i=\{(y_1,\ldots, y_{m_i}) : y_j\in \{f_1^{(i)},\ldots,f_{d-1}^{(i)}\}\}$, then
$$H_{R_i}(\tau)\leq \sum_{(y_1,\ldots, y_{m_i})\in S_i}\tau^{D(y_1)+\ldots+D(y_{m_i})}\leq
(\sum_{j=1}^{d-1}\tau^{D(f_j^{(i)})})^{m_i}\leq (\tau(d-1))^{m_i}=\delta^{m_i}.$$
Hence $1-|X|\tau+H_R(\tau)+\sum_{i\geq 1}H_{R_i}(\tau)<0$, so $G'$ is Golod-Shafarevich, hence infinite.
\end{proof}

We now turn to the proof of the famous inequality $|R|>|X|^2/4$ for finite $p$-groups
which was used in the solution of the class field tower problem. Recall that for a pro-$p$ group $G$
we denote by $d(G)$ and $r(G)$ the minimal number of generators and relators of $G$, respectively.

\begin{Theorem} Let $p$ be a prime.
\label{GS_strong}
\begin{itemize}
\item[(a)] Let $G$ be a finitely presented pro-$p$ group such that $r(G)\leq d(G)^2/4$.
Then $G$ is infinite. If in addition $r(G)< d(G)^2/4$ and $d(G)>1$, then
$G$ is Golod-Shafarevich. 
\item[(b)] Let $\Gamma=\la X|R \ra $ be a finitely presented abstract group, 
and let $d_p(\Gamma)=\dim_{\Fp}(\Gamma/[\Gamma,\Gamma]\Gamma^p)$. If
$|R|<d_p(\Gamma)^2/4-d_p(\Gamma)+|X|$ and $d_p(\Gamma)>1$, then $\Gamma$ 
is Golod-Shafarevich (with respect to $p$).
\end{itemize}
\end{Theorem}

\begin{Lemma}
\label{prop_basic}
Let $(X,R)$ be a presentation of a pro-$p$ group $G$, with $X$ finite,
$F=F_{\phat}(X)$ and $\pi:F\to G$ the natural projection.
The following hold:
\begin{itemize}
\item[(i)] $|X|=d(G)$ if and only if $R$ lies in $\Phi(F)=[F,F]F^p$,
the Frattini subgroup of $F$ (which holds if and only if $D(r)\geq 2$ for all $r\in R$).
Moreover, $R$ contains at least $|X|-d(G)$ elements of degree $1$.
\item[(ii)] Assume that $R$ is finite. Then $G$ has a presentation with $d(G)$ 
generators and $|R|-|X|+d(G)$ relators. More precisely, there exists a subset $X'$ of $X$, 
with $|X'|=d(G)$ and a subset $R'$ of $R$ with $|R'|=|R|-|X|+d(G)$ with the following properties:
$\pi(F(X'))=G$ and if $\theta:F_{\phat}(X)\to F_{\phat}(X')$ is the unique homomorphism which 
acts as identity on $X'$ and sends $X\setminus X'$ to $1$, then
$\theta(R')$ generates $\Ker\pi\cap F(X')$ as a closed normal subgroup of $F(X')$,
and thus $(X',\theta(R'))$ is a presentation of $G$.
\end{itemize} 
\end{Lemma}
\begin{proof} Part (i) easily follows from the fact that $d(G)=d(G/[G,G]G^p)$.
Part (ii) follows from the proof of \cite[Prop.~12.1.5]{Wi}; the first
assertion of (ii) is also proved in \cite[Lemma~1.1]{Lu1}.
\end{proof}

\begin{proof}[Proof of Theorem~\ref{GS_strong}] In view of Lemma~\ref{prop_basic}(i),
(a) can be proved by the same argument as Corollary~\ref{cor:gs2}.
For part (b) let $G=\Gamma_{\phat}$ be the pro-$p$ completion of $G$. Then $d(G)=d_p(\Gamma)$,
and the result follows from (a) and Lemma~\ref{prop_basic}(ii).
\end{proof}

As in the case of graded algebras,
given $d\in\dbN$, it is natural to ask what is the minimal number $f(d)$ for which
there exists a finite $p$-group $G$ with $d(G)=d$ and $r(G)=f(d)$. The best
currently known bound is due to Wisliceny~\cite{Wis} who proved that 
$f(d)\leq ]\frac{d^2}{4}+\frac{d}{2}[$ (note that this coincides with the corresponding
bound for graded algebras from \cite{Wis2} obtained several years later).

\subsection{Word growth in Golod-Shafarevich groups}

In view of our informal statement ``Golod-Shafarevich groups are big'',
it is natural to expect that GS abstract groups at least have exponential growth. 
In fact, more is true: Golod-Shafarevich groups always have uniformly exponential growth,
and this fact has a surpsingly simple proof.

\begin{Proposition}[\cite{BaGr}] 
\label{prop:BaGr}
Golod-Shafarevich abstract groups have uniformly exponential growth. 
\end{Proposition}
\begin{proof} Let $\Gamma$ be a GS group with respect to a prime $p$ and
$M$ the augmentation ideal of $\Fp[\Gamma]$. If $G=\Gamma_{\phat}$ is the pro-$p$
completion of $\Gamma$ and $\overline M$ is the augmentation ideal of $\Fp[[G]]$,
it is easy to show for that for every $n\in\dbN$ the natural map $\Fp[\Gamma]/M^n\to \Fp[[G]]/\overline M^n$
is an isomorphism. Thus, by Proposition~\ref{gs:filtered}(a), the sequence $a_n=\dim_{\Fp}\Fp[\Gamma]/M^{n+1}$ grows exponentially in $n$. Now let $X$ be any generating set of $\Gamma$. Then $\Fp[\Gamma]/M^{n+1}$
is spanned by products of the form $(1-x_1)(1-x_2)\ldots (1-x_m)$ with
$0\leq m\leq n$ and $x_i\in X$. Each such product lies in the $\Fp$-span
of $B_X(n)$, the ball of radius $n$ with respect to $X$ in $\Gamma$. Hence
$|B_X(n)|\geq a_n$.  
\end{proof}

It is now known that Golod-Shafarevich groups are uniformly non-amenable~\cite[Appendix~2]{EJ2}
(which strengthens the assertion of Proposition~\ref{prop:BaGr}), but the proof of this result
is much more involved. We will discuss the proof of non-amenability of Golod-Shafarevich groups
in \S~\ref{sec:T}.3.

\subsection{Golod-Shafarevich groups in characteristic zero}
\label{GSzero}
In this subsection we will briefly discuss groups which are defined
in the same way as Golod-Shafarevich groups in \S~\ref{subsec:GSGdef}, 
except that $\Fp$ will be replaced by a field of characteristic zero.
Unlike the positive characteristic case, we will begin the discussion
with abstract groups.

Let $K$ be a field of characteristic zero, $X=\{x_1,\ldots, x_d\}$ and $U=\{u_1,\ldots, u_d\}$
finite sets of the same cardinality. The map $x_i\to 1+u_i$ still extends to an embedding
of the free group $F(X)\to K\lla U\rra^{\times}$. As in \S~\ref{subsec:GSGdef},
we define the degree function $D_0:F(X)\to \dbN\cup\{\infty\}$ by
$$D_0(f)=\deg(f-1)\mbox{ for all }f\in F(X).$$
Again, the function $D_0$ admits two alternative descriptions, one in terms
of the augmentation ideal and another one in terms of the lower central series,
which replaces the Zassenhaus filtration.
\begin{Proposition}
\label{gs0:basic}
Let $F=F(X)$.
Given $f\in F\setminus \{1\}$ and $n\in\dbN$, the following are equivalent:
\begin{itemize}
\item[(i)] $D_0(f)=n$;
\item[(ii)] $f-1\in M^n\setminus M^{n+1}$ where $M$ is the augmentation ideal of $K[F]$;
\item[(iii)] $f\in \gamma_n F\setminus \gamma_{n+1} F$.
\end{itemize}
\end{Proposition}
\begin{Definition}\rm Let $\Gamma$ be a finitely generated abstract group.
We will say that $\Gamma$ is {\it Golod-Shafarevich in characteristic zero}
if there is a presentation $\Gamma=\la X|R \ra$ s.t. $1-|X|\tau+\sum_{r\in R}\tau^{D_0(r)}<0$
for some $\tau\in (0,1)$.
\end{Definition}
\begin{Observation}
\label{0impliesp}
If an abstract group $\Gamma$ is GS in characteristic zero, then $\Gamma$ is GS with
respect to $p$ for any prime $p$. 
\end{Observation}
\begin{proof} Let $F$ be a finitely generated free group and
$D$ the degree function on $F$ coming from the Magnus embedding
in characteristic $p$ (as defined in \S~\ref{subsec:GSGdef}). Then $D_0(f)\leq D(f)$ for any $f\in F$
by Propositions~\ref{prop:Zass} and \ref{gs0:basic}. This immediately implies the result.
\end{proof}

In view of Observation~\ref{0impliesp}, any result about GS (abstract) groups with respect to a prime $p$
automatically applies to GS groups in characteristic zero. Somewhat surprisingly, nothing
much beyond that seems to be known about  GS groups in characteristic zero. It will be
very interesting to prove some results about GS groups in characteristic zero
(which do not apply to or not known for GS groups with respect to a prime $p$) since the former class
includes some important groups, e.g. free-by-cyclic groups with first Betti number at least two. 
\vskip .2cm

The counterparts of GS pro-$p$ groups in characteristic zero are {\it Golod-Shafarevich 
prounipotent groups}.
Let $K$ be a field of characteristic zero. A prounipotent group over $K$ can be defined as an inverse
limit of unipotent groups over $K$. Given a natural number $d$, the free
prounipotent group of rank $d$ over $K$, denoted here by $F_{\hat K}(d)$, can be defined as the
closure (in the degree topology) of the subgroup of $K\lla u_1,\ldots, u_d\rra^{\times}$, generated
by the elements $(1+u_i)^{\lam}$ for all $\lam\in K$ (where by definition $(1+u_i)^{\lam}=\sum_{m=0}^{\infty}
{\lam \choose m}u_i^m$, which makes sense since $K$ has characteristic zero). Thus, the function $D_0$
originally defined on free abstract groups can be naturally extended to free prounipotent groups,
and one can define Golod-Shafarevich prounipotent groups in the same way as Golod-Shafarevich groups
in characteristic zero, replacing abstract presentations by presentations in the category or prounipotent
groups over $K$. If $\Gamma$ is a GS abstract group in characteristic zero, its prounipotent
completion is a GS prounipotent group, but it is not clear whether the converse is true.

The general theory of prounipotent groups as well as the theory of Golod-Shafarevich prounipotent 
groups was developed by Lubotzky and Magid in \cite{LuMa1,LuMa2, LuMa3}.
Another notable result about Golod-Shafarevich prounipotent groups is due
to Kassabov~\cite{Ka} who proved that they always contain non-abelian free
prounipotent subgroups -- this is a characteristic zero analogue of Zelmanov's
theorem~\cite{Ze1} discussed in \S~\ref{sec:Zelmanov}.

\section{Generalized Golod-Shafarevich groups}

In this section we introduce a more general form of the Golod-Shafarevich inequality
and define the notion of generalized Golod-Shafarevich groups. Similarly one can define 
generalized Golod-Shafarevich algebras (graded or complete filtered) -- see the end of  
\S~\ref{sec:GGS}.

\subsection{Golod-Shafarevich inequality with weights}

In this subsection we essentially repeat the setup of \S~\ref{subsec:GSIF} with two differences:
\begin{itemize}
\item[(i)] generators will be counted with (possibly) different weights;
\item[(ii)] the number of generators will be allowed to be countable. 
\end{itemize}

By allowing countably many generators we will avoid some 
unnecessary restrictions in many structural results about generalized Golod-Shafarevich groups. 
However, all the key applications could still be achieved if we considered only the finitely generated
case, so the reader may safely ignore the few minor subtleties arising from
dealing with the countably generated case.

\vskip .12cm
Let $K$ be a field, $U=\{u_1, u_2,\ldots, \}$ a finite or a countable set and $A=\Fp\lla U\rra$.
Define a function $d:A\to \dbR_{\geq 0}\cup\{\infty\}$ as follows:
\begin{itemize}
\item[(i)] Choose an arbitrary function $d:U\to \dbR_{>0}$, and if $U$ is countable,
assume that $d(u_i)\to\infty$ as $i\to\infty$.
\item[(ii)] Extend $d$ to the set of monic $U$-monomials by
$d(u_{i_1}\ldots u_{i_k})=d(u_{i_1})+\ldots +d(u_{i_k})$.
By convention we set $d(1)=0$. 
\item[(iii)] Given an arbitrary nonzero power series $f=\sum c_{\alpha} m_{\alpha}\in A$
(where $\{m_{\alpha}\}$ are pairwise distinct monic monomials in $U$ and 
$c_{\alpha}\in K$), we put 
\begin{equation}
\label{eq1}
d(f)=\min\{d(m_{\alpha}): c_{\alpha}\neq 0\}.
\end{equation}
Finally, we set $d(0)=\infty$.
\end{itemize}
\begin{Definition}\rm $\empty$
\begin{itemize}
\item[(i)] Any function $d$ obtained in this way will be called
a {\it degree function on $\Fp\lla U\rra$ with respect to $U$}.
\item[(ii)] If $U$ is finite, the unique degree function $d$ such that
$d(u)=1$ for all $u\in U$ will be called {\it standard}.
\end{itemize}
\end{Definition}
Given a subset $S\subseteq A$ such that for each $\alpha\in\dbR$, the set 
$\{s\in S: d(s)=\alpha\}$ is finite, we put $$H_{S,d}(t)=\sum_{s\in S}t^{d(s)}.$$
Note that we do not require that $d$ is integer-valued, so $H_{S,d}(t)$ is not a power
series in general. It is easy to see, however, that for any degree function $d$ 
the set $\Im(d)$ of possible values of $d$ is discrete. Therefore, we can
think of $H_{S,d}(t)$ as an element of the ring $K\{\{t\}\}$ whose elements are
formal linear combinations $\sum_{\alpha\geq 0}c_{\alpha}t^{\alpha}$ where $c_{\alpha}\in K$
and the set $\{\alpha: c_{\alpha}\neq 0\}$ is discrete. The latter condition ensures that 
the elements of $K\{\{t\}\}$ can be multiplied in the same way as usual power series.

For each $\alpha\geq 0$, let 
$$K\lla U\rra_{\alpha}=\{f\in K\lla U\rra: d(f)\geq \alpha\} \mbox{ and } 
K\lla U\rra_{>\alpha}=\{f\in K\lla U\rra: d(f)> \alpha\}.$$
\vskip .2cm

As in \S~\ref{subsec:GSIF}, let $I$ be a closed ideal of $K\lla U\rra_{>0}$, let $A=K\lla U\rra/I$
and $R\subseteq I$ a subset which generates $I$ as a (closed) ideal and such that 
$r_{\alpha}=|\{r\in R: d(r)=\alpha\}|$ is finite for all $\alpha$. Let 
$\pi: K\lla U\rra\to A$ be the natural projection. For each $\alpha\geq 0$
let $A_{\alpha}=\pi(K\lla U\rra_{\alpha})$ and $A_{>\alpha}=\pi(K\lla U\rra_{>\alpha})$
and $a_{\alpha}=\dim_K A_{\alpha}/A_{>\alpha}$. Finally, define the Hilbert series
by
$$Hilb_{A,d}(t)=\sum_{\alpha\geq 0} a_{\alpha} t^{\alpha}.$$
In order to get a direct generalization of Theorem~\ref{gsi:graded3} we have
to assume that the degree function $d$ is integer-valued.

\begin{Theorem}[Golod-Shafarevich inequality: weighted case] 
\label{gsi:graded3}
Assume that $d$ is an integer-valued degree function. Then in the above
setting we have
\begin{equation}
\frac{(1-H_{U,d}(t)+H_{R,d}(t))\cdot Hilb_{A,d}(t)}{1-t}\geq \frac{1}{1-t}.
\label{gsineq3}
\end{equation}
\end{Theorem}
We are not aware of any reference where this theorem is proved as stated
above, but the proof of Theorem~\ref{gsi:graded2} extends to the weighted
case almost without changes. Further, inequality \eqref{gsineq3}
is proved in \cite{Ko2} in a more restrictive setting (see formula (2.11) on p.105),
but again the same argument can be used to establish Theorem~\ref{gsi:graded3}.

\subsection{Generalized Golod-Shafarevich groups}
\label{sec:GGS}
Let $X=\{x_1,\ldots, x_m\}$ and $U=\{u_1,\ldots, u_m\}$
be finite sets of the same cardinality, $F=F_{\phat}(X)$ the free pro-$p$ group
on $X$, and embed $F$ into $\Fp\lla U\rra$ via the Magnus embedding $x_i\mapsto 1+u_i$.

\begin{Definition}\rm $\empty$
\begin{itemize}
\item[(a)] A function $D$ is a called a {\it degree function on $F$ with respect to $X$}
if there exists a degree function $d$ on $\Fp[[F]]$ with respect to $U=\{x-1: x\in X\}$
such that $D(f)=d(f-1)$ for all $f\in F$. We will say that $D$ is the {\it standard
degree function} if $d$ is standard (equivalently if $D(x)=1$ for all $x\in X$).
\item[(b)] Given a subset $S$ of $F$ we put $H_{S,D}(t)=\sum_{s\in S}t^{D(s)}$.
\end{itemize}
\end{Definition}

Now let $G$ be a pro-$p$ group, $(X, R)$ a presentation of $G$,
$D$ a degree function on $F=F_{\phat}(X)$ with respect to $X$ and $d$ the corresponding degree
function on $\Fp[[F]]$. If $D$ is integer-valued,  Golod-Shafarevich inequality~\eqref{gsineq3} yields the
following
\begin{equation}
\label{GSineq_weight}
\frac{(1-H_{X,D}(t)+H_{R,D}(t))\cdot Hilb_{\Fp[[G]],d}(t)}{1-t}\geq \frac{1}{1-t}.
\end{equation}

\begin{Definition}\rm $\empty$
\begin{itemize}
\item[(a)] A pro-$p$ group $G$ is called a {\it generalized Golod-Shafarevich (GGS)
group} if there exists a presentation $(X, R)$ of $G$, a real
number $\tau\in (0,1)$ and a degree function $D$ on $F_{\phat}(X)$ with respect to $X$
such that $1-H_{X,D}(\tau)+H_{R,D}(\tau)<0$.
\item[(b)] An abstract group $G$ is called a GGS group (with respect to $p$) if its pro-$p$ completion is a GGS group.
\end{itemize}
\end{Definition}
\begin{Remark}
It is clear that a pro-$p$ group is GS if and only if it satisfies
condition (a) for the standard degree function $D$.
\end{Remark}

The reader may find it strange that we did not require $D$ to be integer-valued in the
definition of GGS groups, since this assumption is necessary for \eqref{GSineq_weight} to hold.
The reason is that this would not make any difference:

\begin{Lemma} \rm(\cite[Lemma~2.4]{EJ2})\it
\label{integral_approx}
If $(X, R)$ is a presentation and $\tau\in (0,1)$ is such that
$1-H_{X,D}(\tau)+H_{R,D}(\tau)<0$ for some degree function $D$ on $F=F_{\phat}(X)$ with respect to $X$,
then there exists an integer-valued degree function $D_1$ on $F$ with respect to $X$
and $\tau_1\in (0,1)$ such that $1-H_{X,D_1}(\tau_1)+H_{R,D_1}(\tau_1)<0$.
Moreover, we can assume that $\tau_1^{D_1(f)}\leq \tau^{D(f)}$ for all $f\in F$.
\end{Lemma}
The proof of this lemma is not difficult, but the result becomes almost obvious
when restated in the language of weight functions, discussed in the next subsection.

Thanks to this lemma, we can state the following consequence of \eqref{GSineq_weight}
without assuming that $D$ is integer-valued:

\begin{Corollary}
\label{GS:divergent}
In the notations of \eqref{GSineq_weight} assume that $1-H_{X,D}(\tau)+H_{X,R}(\tau)<0$ for some $\tau\in (0,1)$. Then the series $Hilb_{\Fp[[G]],d}(\tau)$ is divergent (so in particular, $G$ is infinite).
\end{Corollary}
\vskip .1cm

All properties of Golod-Shafarevich groups established so far trivially extend
to generalized Golod-Shafarevich groups.
The main reason we are concerned with GGS groups in this paper is the following result, 
which does not have a counterpart for GS groups:

\begin{Theorem} 
\label{GGS_open}
Open subgroups of GGS pro-$p$ groups are GGS.
\end{Theorem}

This theorem, whose proof will be sketched in \S~\ref{sec:tech}, plays a crucial role
in the proofs of some structural results about GS groups, so consideration of GGS groups
is necessary even if one is only interested in GS groups. To give the reader
a better feel about GGS groups, we shall provide a simple example of a pro-$p$ 
group, which is GGS, but not GS.

\begin{Proposition} Let $p>3$, let $F=F_{\phat}(2)$ be the free pro-$p$ group of rank $2$,
let $k\geq 2$, and let $\dbZ_p^k$ denote the $k^{\rm th}$ direct power of $\dbZ_p$
(the additive group of $p$-adic integers). Then the group $G=F\times \dbZ_p^k$ is GGS, but not GS.
\end{Proposition}
\begin{proof} The group $G$ has a natural presentation
$$\la z_1, z_2, y_1,\ldots, y_k \mid [z_i, y_j]=1, [y_i,y_j]=1\ra.\eqno (***)$$
Thus, we have $k+2$ generators and ${k+2\choose 2}-1=k(k+3)/2$ relators of degree $2$, and an 
easy computation shows that this presentation does not satisfy the GS condition for $k\geq 2$.
To prove that no other presentation of $G$ satisfies the GS condition,
we can argue as follows. First, by Lemma~\ref{prop_basic} it is sufficient to consider
presentations $(X,R)$ with $|X|=d(G)=k+2$. We claim that in any such presentation
$R$ has at least $k(k+3)/2$ relators of degree $2$ (this will finish the proof).
It is easy to see that the number of relators of degree $2$ in $R$ is at least
$\log_{p} | D_2 F_{\phat}(X)) / D_3 F_{\phat}(X)|-\log_p |D_2 G/ D_3 G|$ (recall that 
$\{D_n H\}$ is the Zassenhaus filtration of a group $H$). If $x_1,\ldots,x_{k+2}$ are the elements
of $X$, it is easy to see that a basis for $D_2 F_{\phat}(X) / D_3 F_{\phat}(X)$ is given by the images
of commutators $[x_i,x_j]$ for $i< j$ (here we use that $p>3$), while $D_2 G / D_3 G$ is cyclic of order $p$
spanned by the image of $[z_1,z_2]$. Thus, the same computation as above finishes the proof.

To prove that $G$ is a GGS group, let $(X,R)$ be the presentation of $G$ given by (***),
and consider the degree function $D$ on $F_{\phat}(X)$ with respect to $X$ given by $D(z_1)=D(z_2)=1$
and $D(y_i)=N$ for $1\leq i\leq k$, where $N$ is a large integer. Then
$D([z_i,y_j])=N+1$ and $D([y_i,y_j])=2N$, so 
$1-H_{X,D}(\tau)+H_{R,D}(\tau)=1-2\tau-k\tau^N+2k \tau^{N+1}+{k\choose 2}\tau^{2N}$.
This expression can be clearly made negative by first choosing $\tau\in (1/2,1)$ and then
taking a large enough $N$. Thus, $G$ is indeed a GGS group.
\end{proof}
\begin{Remark} Essentially the same argument shows that the direct product of any
GGS pro-$p$ group $G$ with any finitely generated pro-$p$ group $H$ will be GGS.
\end{Remark} 

\vskip .2cm
In complete similarity to the group case, we will call a graded or a complete filtered
$K$-algebra $A$ {\it generalized Golod-Shafarevich} if there exists a presentation (in the corresponding category) $(U,R)$ of $A$, a real number $\tau\in (0,1)$ and a degree function $d$ on 
$K\lla U\rra$ such that $1-H_{U,d}(\tau)+H_{R,d}(\tau)<0$.
 However, it is not clear whether key properties of generalized Golod-Shafarevich groups
(e.g. Theorem~\ref{GGS_open}) would remain true for algebras. In fact, the arguments
of Voden~\cite{Vo} strongly suggest that even if $A$ is a non-abelian free graded algebra, finite
codimension graded subalgebras of $A$ may not be generalized Golod-Shafarevich algebras.
At the same time Voden~\cite{Vo} proves that if a graded algebra $A=\oplus_{n=0}^{\infty}A_n$
has a minimal presentation $(U,R)$, with  $|R|<\frac{1}{4}(\frac{|U|}{2}-1)^2$ and $|U|>1$, then the Veronese power $A^{(k)}$ is Golod-Shafarevich for infinitely many values of $k$,
where by definition $A^{(k)}=\oplus_{n=0}^{\infty}A_{kn}$, with $k\in\dbN$. 
At the moment it is not clear what should be the ``right'' substitute for the notion of
generalized Golod-Shafarevich algebra (if any) which would lead to an interesting theory.

\subsection{Weight functions}

In this subsection we introduce multiplicative counterparts of degree functions,
called weight functions. Even though weight functions are obtained from degree functions
merely by exponentiation, they provide a very convenient language for working
with generalized Golod-Shafarevich groups.

\begin{Definition}\rm Let $F$ be a free pro-$p$ group, $X$ a free generating set of $F$
and $U=\{x-1: x\in X\}$ so that $\Fp[[F]]\cong \Fp\lla U\rra$.
\begin{itemize}
\item[(i)] A function $w:\Fp[[F]]\to [0,1)$ is called
a {\it weight function on $\Fp[[F]]$ with respect to $U$} if there exists
$\tau\in (0,1)$ and a degree function $d$ on $\Fp[[F]]$ with respect to $U$ such that
$$w(f)=\tau^{d(f)}.$$ 
\item[(ii)]
A function $W:F\to [0,1)$ is called
a {\it weight function on $F$ with respect to $X$} if there exists
$\tau\in (0,1)$ and a degree function $D$ on $F$ with respect to $X$ such that
$$W(f)=\tau^{D(f)}.$$
Equivalently, $W$ is a weight function on $F$ with respect to $X$ if
there is a weight function $w$ on $\Fp[[F]]$ with respect to $U$ such that
$W(f)=w(f-1)$ for all $f\in F$.
\end{itemize}
\end{Definition}

If $S$ is a subset of $F$ and $W$ is a weight function on $F$, we define
$W(S)=\sum_{s\in S} W(s)$. Thus, in our previous notations, if $W=\tau^D$
for a degree function $D$, then $W(S)=H_{S,D}(\tau)$. The definition
of generalized Golod-Shafarevich groups can now be expressed as follows:

\begin{Definition}\rm A pro-$p$ group $G$ is a generalized Golod-Shafarevich group
if there exists a presentation $(X, R)$ of $G$ and a weight function $W$
on $F_{\phat}(X)$ with respect to $X$ such that $$1-W(X)+W(R)<0.$$
\end{Definition}

A weight function $W$ will be called {\it uniform} if $W=\tau^D$
for some $\tau$ and the standard degree function $D$ (that is, $D(x)=1$ for all $x\in X$).

\section{Properties and applications of weight functions and valuations}
\label{sec:tech}

\subsection{Dependence on the generating set}

So far we defined weight functions with respect to a fixed
generating set $X$. 
For many purposes it is convenient to have a ``coordinate-free''
characterization of weight functions, where the set $X$ need not be specified
in advance.

\begin{Definition}\rm Let $F$ be a free pro-$p$ group.
A function $W: F\to [0,1)$ is called a {\it weight function on $F$}
if $W$ is a weight function on $F$ with respect to $X$ for some free
generating set $X$ of $F$. Any set $X$ with this property will be called
{\it $W$-free}.
\end{Definition}

The first basic question is which sets are $W$-free for a given weight function $W$
on $F$? If $W$ is a uniform  weight function (that is, $W=\tau^D$ for the standard
degree function $D$), then any free generating set $X$ will be $W$-free,
since the standard degree function can be defined without reference
to a specific free generating set. However, if $W$ is not uniform,
there always exists a free generating set $X$ which is not $W$-free.
This follows from Theorem~\ref{free_crit} below and is illustrated by
the following example.

\begin{Lemma} 
\label{lem:weightdep}
Let $F$ be a finitely generated free pro-$p$ group,
$X=\{x_1,\ldots, x_m\}$ a free generating set of $F$ and 
$W$ a weight function on $F$ with respect to $X$.
Let $f=x_{i_1}^{n_1}\ldots x_{i_k}^{n_k}$ where
the indices $i_1,\ldots, i_k$ are distinct and each $n_i$
is not divisible by $p$. Then $W(f)=\max\{W(x_{i_j})\}_{j=1}^k$. 
\end{Lemma}
\begin{proof} Let $u_i=x_i-1\in\Fp[[F]]$. Then by definition $W(f)=w(f-1)$,
where $w$ is the unique weight function on $\Fp[[F]]$ with respect to $U=\{u_1,\ldots, u_m\}$
such that $w(u_i)=W(x_i)$. Note that $f-1=\sum n_j u_{i_j} + h$ where $h$ is a sum
of monomials, each of which involves at least two $u_{i_j}$'s. Since each $n_j\neq 0$ in $\Fp$
by assumption, we get $W(f)=w(f-1)=\max\{w(u_{i_j})\}_{j=1}^k=\max\{W(x_{i_j})\}_{j=1}^k$.
\end{proof}

\begin{Example} Let $X=\{x_1, x_2\}$, $F=F_{\phat}(X)$ and $\alpha,\beta\in (0,1)$.
Let $W$ be the unique weight function on $F$ with respect to $X$ such that
$W(x_1)=\alpha$ and $W(x_2)=\beta$. Let $X'=\{x_1, x_1 x_2\}$. We claim
that if $\alpha>\beta$, then $W$ is NOT a weight function with respect to $X'$.
Indeed, by Lemma~\ref{lem:weightdep}, $W(x_1 x_2)=\max\{\alpha, \beta\}=\alpha$.
If $W$ was also a weight function with respect to $X'$, Lemma~\ref{lem:weightdep}
would have implied that $W(x_2)=W(x_1^{-1}\cdot x_1x_2)$ is equal to 
$\max\{W(x_1),W(x_1x_2)\}=\alpha$, which is false.
\end{Example}

If we assume that $\alpha\leq \beta$ in the above example, then $W$ will
be a weight function with respect to $\{x_1, x_1 x_2\}$, although this
is harder to show by a direct computation. In general, we have the following
criterion for $W$-freeness.

\begin{Theorem}
\label{free_crit}
Let $W$ be a weight function on a free pro-$p$ group $F$, let $X$ be a free generating set of $F$,
and assume that $W(X)<\infty$. The following are
equivalent.
\begin{itemize}
\item[(i)] $X$ is $W$-free.
\item[(ii)] If $X'$ is any free generating set of $F$, then $W(X)\leq W(X')$.
\item[(iii)] If $X'$ is any free generating set of $F$, then there is a bijection
$\sigma: X\to X'$ such that $W(x)\leq W(\sigma(x))$ for all $x\in X$.
\end{itemize}
\end{Theorem}
\begin{proof} This is an easy consequence of results in \cite[\S~3]{EJ3}.
More specifically, let us say that $X$ is a $W$-optimal generating set if
it satisfies condition (ii). The definition of a $W$-optimal generating set
in \cite{EJ3} is different, but the two definitions are equivalent by 
Proposition~3.6 and Corollary~3.7 of \cite{EJ3}. Proposition~3.9 of \cite{EJ3}
then shows the equivalence of (i) and (ii) and Proposition~3.6 easily
implies the equivalence of (ii) and (iii).
\end{proof}

One of the most important properties of weight functions is the following theorem:

\begin{Theorem}\cite{EJ2, EJ3} 
\label{weight_inherit}
Let $F$ be a free pro-$p$ group and $W$
a weight function on $F$. If $K$ is a closed subgroup of $F$,
then $W$ restricted to $K$ is also a weight function.
\end{Theorem}

This theorem appears as \cite[Cor.~3.6]{EJ2} in the case when $K$ is open in $F$
and $F$ is finitely generated and as \cite[Cor.~3.4]{EJ3} in the general case. The second proof
is more conceptual (see a brief sketch in \S~\ref{sec:tech}.2), while the first one has the advantage of 
producing an algorithm for finding a $W$-free generating set for an open subgroup $K$ 
(see a sketch  in \S~\ref{sec:tech}.3).

\subsection{Valuations}

One inconvenience in working with weight functions is that they are defined
on free pro-$p$ groups and not on the groups given by generators and relators
we are trying to investigate. However, as we will explain below, given
a pro-$p$ group $G$ and a {\it free presentation} $\pi:F\to G$ (that is, an epimorphism
from a free pro-$p$ group $F$ to $G$), every weight function on $F$ will
induce a function on $G$ satisfying certain properties. Such functions
will be called valuations. 

\begin{Definition}\rm Let $G$ be a pro-$p$ group.
A continuous  function $W:G\to [0,1)$ is called a \emph{valuation} if
\begin{itemize}
\item[(i)] $W(g)=0$ if and only if $g=1$;
\item[(ii)] $W(fg)\leq \max\{W(f),W(g)\}$ for any $f,g\in G$;
\item[(iii)] $W([f,g])\leq W(f) W(g)$ for any $f,g\in G$;
\item [(iv)] $W(g^{p})\leq W(g)^p$ for any $g\in G$.
\end{itemize}
\end{Definition}

It is easy to check that any weight function on a free pro-$p$ group is a valuation, 
but the converse is not true -- for instance, if $W$ is a weight function on a free pro-$p$ group $F$ 
such that $W(f)<1/2$  for all $f\in F$, the function $W'(f)=2W(f)$ will be a valuation, but not a weight function.
It is also not hard to show that given a free generating set $X$ of $F$ and a weight function $W$
on $F$ with respect to $X$ the following is true:
\vskip .3cm
If $W'$ is any valuation on $F$ such that $W'(x)=W(x)$ for all $x\in X$,
then $W'(f)\leq W(f)$ for all $f\in F$.
\vskip .3cm

There are two simple ways to get new valuations from old. 
\begin{itemize}
\item[(i)] If $W$ is a valuation
on $G$ and $H$ is a closed subgroup of $G$, then $W$ restricted to $H$ is clearly a valuation
on $H$ (we will denote the restricted valuation also by $W$).

\item[(ii)]If $\pi:H\to G$ is an epimorphism of pro-$p$ groups, then every valuation $W$ on $H$
induces the corresponding valuation $W'$ on $G$ given by
$$W'(g)=\inf\{W(h): h\in H, \pi(h)=g\}.$$
In the special case when $G$ is defined as a quotient of $H$ and $\pi:G\to H$ is the natural
projection, we will usually denote the induced valuation on $G$ also by $W$.
\end{itemize}

An important special case of (ii) is that if $G$ is a pro-$p$ group and $\pi:F\to G$ is a free presentation 
of $G$, then every weight function on $F$ will induce a valuation on $G$. By \cite[Prop.~4.7]{EJ3},
the converse is also true: every valuation on $G$ is induced from some weight function in such a way;
however, if $G$ is finitely generated, one cannot guarantee that $F$ can be chosen finitely generated.

Given a valuation $W$ on $G$ and $\alpha\in (0,1)$, define the subgroups $G_{\alpha,W}$ and 
$G_{<\alpha,W}$ of $G$
by $$G_{\alpha,W}=\{g\in G : W(g)\leq \alpha\}\quad \mbox{ and }\quad 
G_{<\alpha,W}=\{g\in G : W(g)< \alpha\}.$$
The associated graded restricted Lie algebra $L_W(G)$ is defined as follows:
as a graded abelian group $L_W(G)=\oplus_{\alpha\in\Im(W)}G_{\alpha,W}/G_{<\alpha,W}$,
the Lie bracket is defined by $$[g G_{<\alpha,W}, h G_{<\beta,W}]=[g,h]G_{<\alpha\beta,W}
\mbox{ for all }g\in G_{\alpha,W}\mbox{ and }h\in G_{\beta,W}$$
(where $[g,h]=g^{-1}h^{-1}gh$) and the $p$-power operation is defined by
$$(g G_{<\alpha,W})^{[p]}=g^p G_{<\alpha^p,W}
\mbox{ for all }g\in G_{\alpha,W}\mbox{ and }h\in G_{\beta,W}.$$
\vskip .2cm
If $H$ is a closed subgroup of $G$, it is easy to
see that $L_W(H)$ (the Lie algebra of $H$ with respect to the induced valuation)
is naturally isomorphic to a subalgebra of $L_W(G)$. The following characterization
of weight functions among all valuations is obtained in \cite{EJ3}.

\begin{Theorem}\rm (\cite[Corollary~3.3]{EJ3})
\label{free_rest1}
Let $F$ be a free pro-$p$ group.
A valuation $W$ on $F$ is a weight function if and only if $L_W(F)$
is a free restricted Lie algebra.
\end{Theorem}

Since subalgebras of free restricted Lie algebras are free restricted,
Theorem~\ref{weight_inherit} follows from Theorem~\ref{free_rest1} and
a paragraph preceding it.
\vskip .2cm
Finally, similarly to weight functions, given a valuation $W$ on a pro-$p$ group $G$
and a subset $S$ of $G$, we put $W(S)=\sum_{s\in S}W(s)$.

\subsection{Weighted rank, index and deficiency}

Given a pro-$p$ group $G$ and a valuation $W$ of $G$, there are
three important numerical invariants --
\begin{itemize}
\item[(i)] the $W$-rank of $G$, denoted by $rk_W(G)$,
\item[(ii)] the $W$-deficiency of $G$, denoted by $def_W(G)$, and
\item[(iii)] for every closed subgroup $H$ of $G$, the $W$-index of $H$ in $G$, denoted by $[G:H]_W$
\end{itemize}
-- which behave very similarly to their usual (non-weighted) counterparts.

The definition of the $W$-rank is the obvious one:

\begin{Definition}\rm The $W$-rank of a pro-$p$ group $G$, denoted by $rk_W(G)$,
is the infimum of the set $\{W(X)\}$ where $X$ ranges over all generating sets of $G$. 
In fact, a standard compactness argument shows that if this infimum is finite, it must be attained on some set $X$.
\end{Definition}

Before defining $W$-deficiency, we need some additional terminology.

\begin{Definition}\rm A valuation $W$ on a pro-$p$ group $G$ is called {\it finite}
if there is a free presentation $\pi:F\to G$ and a weight function $\widetilde W$
on $F$ which induces $W$ and such that $rk_{\widetilde W}(F)<\infty$.
\end{Definition}

Note that in many cases finiteness of a valuation $W$ holds automatically:
this is the case if $W$ is a weight function on a finitely generated free pro-$p$ group $F$
or if $W$ is quotient-induced from such a weight function. In applications of 
Golod-Shafarevich groups all valuations will be obtained in such way, so the problem
of verifying finiteness of a valuation never arises in practice. In more theoretical
contexts, one can use the following criterion (\cite[Prop.~4.7]{EJ3}): a valuation $W$ on $G$ is 
finite if and only if there is a subset $Y$ of $G$, with $W(Y)<\infty$, 
s.t. the elements $\{y G_{<W(y),W} : y\in Y\}$ generate the Lie algebra $L_W(G)$. 

\begin{Definition}\rm $\empty$
\begin{itemize}
\item[(a)] A weighted presentation is a triple $(X,R,W)$ where $(X,R)$ is a (pro-$p$) presentation
and $W$ is a weight function on $F_{\phat}(X)$ with respect to $X$.
\item[(b)] Let $(X,R,W)$ be a weighted presentation, where $W(X)<\infty$.
We set $def_W(X,R)=W(X)-W(R)-1$.
\item[(c)] Let $G$ be a pro-$p$ group and $W$ a finite valuation on $G$. The $W$-deficiency of $G$,
denoted by $def_W(G)$, is defined to be the supremum of the set $\{def_{\widetilde W}(X,R)\}$ where
$(X,R,\widetilde W)$ ranges over all weighted presentations such that $G=\la X|R\ra$,
$\widetilde W$ induces $W$ and $\widetilde W(X)<\infty$.
\end{itemize} 
\end{Definition}

Note that we can now rephrase the definition of GGS groups as follows.

\begin{Definition}\rm
A pro-$p$ group $G$ is GGS if and only if $def_W(G)>0$ for some finite valuation $W$ of $G$.
\end{Definition}

Thus, GGS groups can be thought of as groups of positive weighted deficiency (explaining the title
of \cite{EJ3}). We will not use the latter terminology in this paper, but we will work with the quantity
$def_W(G)$, which is very convenient.

Recall two classical inequalities relating the usual (non-weighted) notions of rank, deficiency
and index.
\begin{Theorem}
\label{thm:Schreier}
Let $G$ be a finitely generated abstract or a pro-$p$ group and $H$ a 
finite index subgroup of $G$. Then
\begin{itemize}
\item[(a)] $d(H)-1\leq (d(G)-1)[G:H]$. Moreover, if $G$ is free, then $H$ is free and equality holds.
\item[(b)] $def(H)-1\geq (def(G)-1) [G:H]$.
\end{itemize}
\end{Theorem}
Part (a) is just the Schreier formula, and (b) is an easy consequence of (a) and the Reidemeister-Schreier
rewriting process (see, e.g., \cite[Lemma~2.1]{Os}).

It turns out that one can define $W$-index $[G:H]_W$ in such a way that the weighted analogues of (a) and (b)
will hold. 

\begin{Definition}\rm
Let $W$ be a valuation on a pro-$p$ group $G$ and $H$ a closed subgroup of $G$.
For each $\alpha\in\Im(W)$ let $c_{\alpha,W}(G/H)=\log_p | G_{\alpha,W}H/G_{<\alpha,W}H |$.
The quantity $$[G:H]_W=\prod_{\alpha\in\Im(W)}\left(\frac{1-\alpha^p}{1-\alpha}\right)^{c_{\alpha,W}(G/H)}$$
is called the {\it $W$-index of $H$ in $G$}.
\end{Definition}

In the next subsection we will reveal where the above formula comes from. At this point we just
observe that the usual index $[G:H]$ is given by the formula  $[G:H]=p^{\sum_{\alpha\in\Im(W)} c_{\alpha,W}(G/H)}$. 
Hence, if we fix $H$ and consider a sequence $\{W_n\}$ of valuations on $G$
which converges pointwise to the constant function $1$ on $G\setminus\{1\}$, then
the sequence $[G:H]_{W_n}$ will converge to $[G:H]$.

Here is the weighted counterpart of Theorem~\ref{thm:Schreier}. We will discuss
the idea of its proof in the next subsection.

\begin{Theorem}
\label{thm:weightedSchreier}
Let $W$ be a valuation on a pro-$p$ group $G$ and let $H$ be a closed subgroup of $G$
with $[G:H]_W<\infty$. Then
\begin{itemize}
\item[(a)] $rk_W(H)-1\leq (rk_W(G)-1)[G:H]_W$. Moreover, if $G$ is free and $W$ is a weight function, 
then equality hodls.
\item[(b)] $def_W(H)\geq def_W(G) [G:H]_W$.
\end{itemize}
\end{Theorem}

Note that part (b) immediately implies that an open subgroup of a GGS pro-$p$ group
is also a GGS pro-$p$ group, the result stated earlier as Theorem~\ref{GGS_open}.

Below are some properties of $W$-index which we shall need later:

\begin{Proposition}\rm
\label{index_basicrpop}
Let $W$ be a valuation on a pro-$p$ group $G$.
Then $W$-index is multiplicative, that is, if $K\subseteq H$ are closed subgroups of $G$,
then $[G:K]_W=[G:H]_W\cdot [H:K]_W$.
\end{Proposition}

\begin{Proposition}[Continuity lemma]
\label{lemma:continuity}
Let $W$ be a valuation on a pro-$p$ group $G$, let $H$ be a closed subgroup of $G$,
and let $\{U_n\}$ be a descending chain of open subgroups of $G$ such that $H=\cap U_n$.
Then
\begin{itemize}
\item[(i)] $[G:H]_W=\lim_{n\to\infty} [G:U_n]_W.$
\item[(ii)] If $[G:H]_W<\infty$, then $rk_W(H)=\lim_{n\to\infty} rk_W(U_n).$
\item[(iii)] If $[G:H]_W<\infty$, then $def_W(H)=\lim_{n\to\infty} def_W(U_n)$
\end{itemize}
\end{Proposition}

Proposition~\ref{index_basicrpop} is straightforward, and Continuity Lemma is established in \cite{EJ3}:
(i) is \cite[Lemma~3.15]{EJ3}, (ii) is \cite[Lemma~3.17]{EJ3} and (iii) follows from 
the proof of \cite[Prop.~4.3]{EJ3} although it is not explicitly stated there.

\subsection{Proof of Theorem~\ref{thm:weightedSchreier} (sketch)}

Proposition~\ref{lemma:continuity} reduces both (a) and (b) to the case
when $H$ is an open subgroup. By Proposition~\ref{index_basicrpop},
it suffices to consider the case when $[G:H]=p$.

(a) We first treat the case when $G$ is free and $W$ is a weight function.
Let $X$ be any $W$-free generating set of $F$. First, replacing
$X$ by another $W$-free generating set, we can always assume that
$$\mbox{There is  just one element $x\in X$ which lies outside of $H$}. \eqno (***)$$
To achieve this, we let $x$ be the element of the original set $X$
which lies outside of $H$ and has the smallest $W$-weight among all such elements
and then, for each $z\in X\setminus \{x\}\setminus H$, replace $z$ by $zx^m$
for suitable $m\in\dbZ$ so that $zx^m\in H$ (such $m$ exists since $G/H$ is cyclic
of prime order and $x\not\in H$). Condition (ii) in the definition
of a valuation and Theorem~\ref{free_crit} ensure that the new generating
set of $F$ is still $W$-free. 

From now on we shall assume that (***) holds. A standard application
of the Schreier method shows that $H$ is freely generated by the
set $X'=\{y^{x^i} : y\in X\setminus\{x\}, 0\leq i<p \}\cup \{x^p\}$. This set,
however, is not $W$-free by Theorem~\ref{free_crit} since it
clearly does not have the smallest possible $W$-weight -- for instance,
one can replace $y^{x}$ by $y^{-1}y^x=[y,x]$, thereby decreasing the
total weight as $W([y,x])\leq W(y)W(x)<W(y)=W(y^x)$. It is not difficult
to show that the $W$-weight will be minimized on the generating set
\begin{equation}
\label{eq:tilde}
\widetilde X=\cup_{y\in X\setminus\{x\}}\{y, [y,x],[y,x,x], \ldots, [y,\underbrace{\! x,\ldots, x]}_{p-1\mbox
{ times }}\}\cup\{x^p\}.
\end{equation}
Informally, this happens because if we let $U=\{x-1: x\in X\}$ (so that $\Fp[[F]]\cong \Fp\lla U\rra$)
and expand elements of the set $\widetilde U=\{\widetilde x-1 : 
\widetilde x \in\widetilde X\}$
as power series in $U$, then the monomials of maximal $W$-weight in those expansions
will all be distinct (this is a Gr\"obner basis type of argument). Thus, by Theorem~\ref{free_crit},
$\widetilde X$ is $W$-free. 
\skv

Note that if $\tau=W(x)$, then in the above formula we have 
$$W(\widetilde X)-1=(W(X)-\tau)(1+\tau+\ldots+\tau^{p-1})+\tau^p-1=
(W(X)-1)\cdot \frac{1-\tau^p}{1-\tau}.$$ Again by Theorem~\ref{free_crit} we have
$W(X)=rk_W(F)$ and $W(\widetilde X)=rk_W(H)$. Moreover, it is not difficult
to show that $c_{\alpha,W}(G/H)$ is equal to $1$ for $\alpha=\tau$ and $0$ for $\alpha\neq \tau$.
Hence $[G:H]_W=\frac{1-\tau^p}{1-\tau}$, and we are done.
\skv
In the case when $G$ is an arbitrary group, we can essentially repeat the above argument,
assuming at the beginning that $X$ is a generating set of $G$ with $W(X)=rk_W(G)$.
The set $\widetilde X$ given by \eqref{eq:tilde} will still generate $H$, but we no longer
know whether $W(\widetilde X)$ equals $rk_W(H)$; this is why we can only claim inequality
in the formula.

(b) follows easily from (a) and the Schreier method. First, by Propositions~\ref{index_basicrpop}
and ~\ref{lemma:continuity} we can again assume that $[G:H]=p$. 
Consider any weighted presentation $(X,R,\widetilde W)$ of $G$
where $\widetilde W$ induces $W$ and $\widetilde W(X)<\infty$. 
As before, we can assume that (***) from (a) holds. Then $H$ is given by the presentation
$(\widetilde X, R')$ where $\widetilde X$ is as before and 
$R'=\{r^{x^i} : r\in R,\, 0\leq i<p \}$. Again, we can replace the set of relators $R'$ by the set 
$\widetilde R=\cup_{r\in R}\{r, [r,x],[r,x,x], \ldots, [r,\underbrace{\! x,\ldots, x]}_{p-1\mbox
{ times }}\},$ and by direct computation $\widetilde W(\widetilde R)\leq \widetilde W(R)[G:H]_W$.
Since $\widetilde W(\widetilde X)-1=( \widetilde W(X)-1)[G:H]_W$ by the proof of (a), we conclude that
$def_{\widetilde W}(\widetilde X,\widetilde R)\geq def_{\widetilde W}(X,R)[G:H]_W$, which yields (b) by taking the supremum
of both sides over all triples $(X,R,\widetilde W)$.

\subsection{$W$-index and Quillen's theorem}
\label{subsec:Quillen}

Although we already gave some indication why the notion of $W$-index
is a useful tool, its definition may still appear mysterious.
Below we state another formula generalizing (a version of) Quillen's theorem,
where $W$-index naturally appears. In order to state it, we have to go back to degree functions and also introduce
some additional notations.

Let $G$ be a pro-$p$ group, $\pi:F\to G$ a free presentation of $G$,
$d$ a degree function on $\Fp[[F]]$ and $D$ the corresponding degree function
on $F$, that is, $D(a)=d(a-1)$. Then $D$ induces a function on $G$ 
(by abuse of notation also denoted by $D$) given by
$$D(g)=\inf\{D(f): f\in F, \pi(f)=g\}.$$
For each $\lam\in\dbR_{>0}$ let $G^{\lam,D}=\{g\in G: D(g)\geq \lambda\}$,
$G^{>\lam,D}=\{g\in G: D(g)> \lambda\}$ and $c^{\lam,D}(G)=[G^{\lam,D}:G^{>\lam,D}]$

\begin{Theorem} Let $G,F,\pi, d,D$ be as above.
\label{thm:Quillen}
\begin{itemize}
\item[(a)] (Quillen's theorem) The following equality of generalized power series holds:
\begin{equation}
\label{eq:Quillen}
Hilb_{\Fp[[G]],d}(t)=\prod_{\lam\in \Im(D)}\left(\frac{1-t^{p\lam}}{1-t^{\lam}} \right)^{c^{\lam,D}(G)}.
\end{equation}
\item[(b)] Let $\tau\in (0,1)$, and define the function $W:G\to [0,1)$ by $W(g)=\tau^{D(g)}$.
Then $W$ is a valuation on $G$ and $Hilb_{\Fp[[G]],d}(\tau)=[G:\{1\}]_W$, the $W$-index of 
the trivial subgroup.
\end{itemize}
\end{Theorem}
\begin{proof}[Sketch of proof] The idea of the proof of (a) is very simple. Consider the graded restricted Lie algebra
$L^{D}(G)=\oplus_{\lam\in\Im(D)}G^{\lam,D}/G^{>\lam,D}$ associated to the degree function $D$
and the graded associative algebra $gr_{d}\Fp[[G]]=\oplus_{\lam\in\Im(d)}\Fp[[G]]^{\lam,d}/\Fp[[G]]^{>\lam,d}$ 
associated to the degree function $d$
(in fact, $L^D(G)$ coincides with the Lie algebra $L_W(G)$ defined in \S~\ref{sec:tech}.2,
corresponding to the valuation $W=\tau^D$ for any $\tau\in (0,1)$). It turns out that 
the restricted universal enveloping algebra $\mathcal U(L^{D}(G))$ is isomorphic to $gr_{d}\Fp[[G]]$ 
as a graded associative algebra. Hence
$Hilb_{\Fp[[G]],d}(t)$, which by definition is the Hilbert series of $gr_{d}\Fp[[G]]$, 
must equal the Hilbert series of $\mathcal U(L^{D}(G))$, which is equal to the right-hand side of
\eqref{eq:Quillen} by the Poincare-Birkhoff-Witt theorem for restricted Lie algebras.

In the case when $D$ is the standard degree function, the isomorphism 
$gr_{d}\Fp[[G]]\cong \mathcal U(L^{D}(G))$ is known as Quillen's theorem and its 
detailed proof can be found, for instance, in \cite[Ch.11,12]{DDMS}.
In \cite[Prop.~2.3]{EJ2}, (a) is proved for the case of integer-valued degree functions $D$,
but the same argument works for arbitrary $D$.
\vskip .1cm
(b) It is clear that $W$ is a valuation on $G$. Note that $G^{\lam,D}=G_{\tau^{\lam},W}$
and $G^{>\lam,D}=G_{<\tau^{\lam},W}$, so $c^{\lam,D}(G)=c_{\tau^{\lam},W}(G/\{1\})$. Therefore,
if we set $t=\tau$ in \eqref{eq:Quillen}, the right-hand side becomes equal to $[G:\{1\}]_W$ by definition.
\end{proof}

\begin{Corollary} 
\label{Worder}
Let $G$ be a GGS pro-$p$ group, and let $W$ be a valuation on $G$
such that $def_W(G)>0$. Then $[G:\{1\}]_W=\infty$, and therefore 
by Proposition~\ref{lemma:continuity}~(Continuity lemma),
the set $\{[G:U]_W\}$, where $U$ runs over open subgroups of $G$,
is unbonded from above.
\end{Corollary}
\begin{proof} Let $(X,R,\widetilde W)$ be a weighted presentation of $G$
such that $def_{\widetilde W}(X,R)>0$ and $\widetilde W$ induces $W$.
By definition $\widetilde W=\tau^D$ for some $\tau\in (0,1)$
and degree function $D$ on $F=F_{\phat}(X)$ with respect to $X$.
Let $d$ be the degree function on $\Fp[[F]]$ corresponding to $D$.
By Corollary~\ref{GS:divergent}, the series $Hilb_{\Fp[[G]],d}(\tau)$
diverges, so the result follows from Theorem~\ref{thm:Quillen}(b).
\end{proof}

Using Quillen's theorem, we can also interpret inequality in Theorem~\ref{thm:weightedSchreier}(b)
(or rather a slightly stronger version of it) as yet another generalization of GS inequality.

First we restate Theorem~\ref{thm:weightedSchreier}(b) in terms of weight functions
(formally the statement below is stronger, but it follows immediately from the proof):

\begin{Theorem}
\label{Schreier_pres}
Let $G$ be a pro-$p$ group given by a presentation $(X, R)$
and let $W$ be a weight function on $F_{\phat}(X)$ with respect to $X$.
Let $K$ be an open subgroup of $G$. Then $K$ has a presentation $(X', R')$,
with $X'\subset F_{\phat}(X)$,
such that
$$(1-W(X')+W(R'))\leq (1-W(X)+W(R))\cdot [G:K]_W.$$
 \end{Theorem}

Suppose now that $W=\tau^D$ for a degree function $D$ and $\tau\in (0,1)$. 
As in the proof of Theorem~\ref{thm:Quillen}(b) we have
$c^{\lam,D}(G/K)=c_{\tau^{\lam},W}(G/K)$, so
Theorem~\ref{thm:weightedSchreier} can now be restated as a numerical inequality
\begin{equation}
\label{GSineq_finite}
1-H_{X',D}(\tau)+H_{R',D}(\tau)\leq \left( 1-H_{X,D}(\tau)+H_{R,D}(\tau)\right)\cdot 
\prod_{\lam\in \Im(D)} \left(\frac{1-\tau^{\lam p}}{1-\tau^{\lam}} \right)^{c^{\lam,D}(G/K)}.
\end{equation}

One can show (see \cite[Theorem~3.11(a)]{EJ2}) that if $D$ is integer-valued (this time
it is an essential assumption), then by dividing both sides of \eqref{GSineq_finite}
by $1-\tau$ and replacing a real number $\tau$ by the formal variable $t$,
we get a valid inequality of power series (the proof
of this result follows the same scheme as that of Theorem~\ref{thm:weightedSchreier}(b)):

\begin{Theorem}\rm (\cite[Theorem~3.11(a)]{EJ2})\it
\label{GS_finitary}
Let $G$ be a pro-$p$ group, $( X, R)$ a presentation of $G$ and 
$D$ an integer-valued degree function on $F=F_{\phat}(X)$ with respect to $X$. Let $K$ be an open subgroup of $G$.
Then there exists a presentation $( X', R')$ of $K$, with $X'\subset F_{\phat}(X)$,
such that the following inequality of power series holds:
\begin{equation}
\label{GSineq_finite2}
\frac{1-H_{X',D}(t)+H_{R',D}(t)}{1-t}\leq \frac{1-H_{X,D}(t)+H_{R,D}(t)}{1-t}\cdot 
\prod_{n\in \Im(D)} \left(\frac{1-t^{np}}{1-t^n} \right)^{c^{n,D}(G/K)}.
\end{equation}
\end{Theorem}

Finally, assume that $K$ is normal in $G$. Then applying Theorem~\ref{thm:Quillen}(a)
to the quotient group $G/K$ and letting $d$ be the degree function corresponding to $D$, 
we can rewrite \eqref{GSineq_finite2} as follows:
\begin{equation}
\label{GSineq_finite3}
\frac{1-H_{X',D}(t)+H_{R',D}(t)}{1-t}\leq \frac{1-H_{X,D}(t)+H_{R,D}(t)}{1-t}\cdot Hilb_{\Fp[[G/K]],d}(t).
\end{equation}
This inequality can be thought of as a finitary version of the generalized Golod-Shafarevich 
inequality~\eqref{GSineq_weight}; in fact, \eqref{GSineq_weight} can be deduced from \eqref{GSineq_finite3}, 
as explained in \cite{EJ2} (see a remark after Theorem~3.11). 
Moreover, \eqref{GSineq_finite3} remains true even without the assumption that $K$ is normal in $G$, but 
$Hilb_{\Fp[[G/K]],d}(t)$ will need to be defined differently.

\subsection{A proof without Hilbert series}

In conclusion of this long section we shall give a short alternative
proof of the fact that GGS pro-$p$ groups are infinite, which does not
use Hilbert series. The proof is based on the following lemma,
which we shall also need later for other purposes.

\begin{Lemma}
\label{lem:lowerbound}
Let $G$ be a pro-$p$ group and $(X,R,W)$ a weighted presentation of $G$,
where $W$ is finite. Then
\begin{itemize}
\item[(a)] $d(G)\geq W(X)-W(R)$.
\item[(b)] Let $\alpha>0$, and given a subset $S$ of $F(X)$, let $S_{\geq \alpha}=\{s\in S : W(s)\geq \alpha\}$.
Then $d(G)\geq W(X_{\geq \alpha})-W(R_{\geq \alpha})$.
\end{itemize}
\end{Lemma}
\begin{proof} Note that (a) follows from (b) by letting $\alpha\to 0$, so 
we shall only prove (b).
Since $G$ is pro-$p$, there is a subset $Y\subseteq X$ such that $Y$ generates $G$ and $|Y|=d(G)$.
Then the presentation $(X, R\cup Y)$ defines the trivial group. Hence $R\cup Y$ generates
$F_{\phat}(X)$ as a  normal subgroup of itself, and since $F_{\phat}(X)$ is pro-$p$, $R\cup Y$ generates $F_{\phat}(X)$ as a pro-$p$ group.
Since $X$ is $W$-free, by  Theorem~\ref{free_crit}(i)(iii), we have $W(X_{\geq\alpha})\leq W((R\cup Y)_{\geq\alpha})$,
whence $W(X_{\geq\alpha})\leq W(R_{\geq\alpha})+W(Y_{\geq\alpha})\leq W(R_{\geq\alpha})+|Y_{\geq\alpha}|
\leq W(R_{\geq\alpha})+d(G)$.
\end{proof}

Using Lemma~\ref{lem:lowerbound} and Theorem~\ref{thm:weightedSchreier}(b),
it is now very easy to show that GGS pro-$p$ groups are infinite. Indeed, suppose
that $G$ is a GGS pro-$p$ group, so that $def_W(G)>0$ for some $W$. 
Then $\widetilde W(X)-\widetilde W(R)>0$ for some weighted presentation 
$(X,R,\widetilde W)$ of $G$, so by Lemma~\ref{lem:lowerbound}(a),
$G$ is non-trivial and in particular has an open subgroup of index $p$, call it $H$.
By Theorem~\ref{thm:weightedSchreier}, $H$ is also GGS. We can then apply the same argument to $H$
and repeat this process indefinitely, thus showing that $G$ is infinite. 

\section{Quotients of generalized Golod-Shafarevich groups}
\label{sec:quot}

One of the reasons (generalized) Golod-Shafarevich groups are so useful is that they
possess infinite quotients with many prescribed group-theoretic properties.
Some results of this type are deep and require original arguments,
but in many cases all one needs is the following obvious lemma:

\begin{Lemma} 
\label{GGS_quotbase}
Let $G$ be a pro-$p$ group and $W$ a valuation on $G$.
If $S$ is any subset of $G$, then $def_W(G/\la S\ra^G)\geq def_W(G)-W(S)$.
\end{Lemma}

As the first application of this lemma, we shall prove a simple but extremely useful result 
due to J. Wilson~\cite{Wi}.

\begin{Theorem}
\label{GS_Wilson}
Every GS (resp. GGS) abstract group has a torsion quotient which is
also GS (resp. GGS).
\end{Theorem}
\begin{proof} The argument is just a minor variation of the second proof of Theorem~\ref{thm:GB},
but we shall state it using our newly developed language. Let $\Gamma$ be a GGS abstract group
(which can be assumed to be residually-$p$), $G=\Gamma_{\phat}$ and $W$ a valuation on $G$ such that 
$def_W(G)>0$. For each $g\in \Gamma$ we can choose an integer $k(g)\in\dbN$ such that
if $R=\{g^{p^{k(g)}}: g\in \Gamma\}$, then $W(R)<def_W(G)$.

Let $\Gamma'=\Gamma/\la R\ra^{\Gamma}$. Then $\Gamma'$ is torsion; on the other hand,
the pro-$p$ completion of $\Gamma'$ is isomorphic to $G'=G/\la R\ra^G$ which is GGS
by Lemma~\ref{GGS_quotbase}.

If $\Gamma$ is GS, then we can assume that the initial $W$ is induced by a uniform weight
function, whence $\Gamma'$ is also GS.
\end{proof}

The argument used to prove Theorem~\ref{GS_Wilson} has the following obvious generalization:
\begin{Observation}
\label{obs_basic}
Let (P) be some group-theoretic property such that
\begin{itemize}
\item[(i)] (P) is inherited by quotients,
\item[(ii)] given an abstract group $\Gamma$, a valuation $W$ on its
pro-$p$ completion $G=\Gamma_{\phat}$ and $\eps>0$, there exists a subset
$R_{\eps}$ of $\Gamma_{\phat}$ such that $W(R_{\eps})<\eps$ and
the image of $\Gamma$ in $G/\la R_{\eps}\ra^G$ has (P).
\end{itemize}
Then any GGS group (resp. GS group) has a GGS quotient (resp. GS quotient) with (P).
Moreover, this quotient can be made residually finite.
\end{Observation}

Of course, in the proof of Theorem~\ref{GS_Wilson} (P) was the property of being
a $p$-torsion group. Below we state several other results which can be proved using
Observation~\ref{obs_basic} or its variation.

\begin{Theorem}\rm (\cite[Theorem~1.2]{EJ3})\it 
\label{thm:LERF0}
Every GGS abstract group has a GGS quotient with
property LERF.
\end{Theorem}

\begin{Theorem}[\cite{Er2}] 
\label{thm:FC}
Every GGS abstract group has a residually finite quotient whose
FC-radical (the set of elements centralizing a finite index subgroup)
is not virtually abelian. (Of course, such a quotient must be infinite).
\end{Theorem}

\begin{Theorem}[\cite{MyaOs}] 
\label{thm:algfin}
Every recursively presented GS abstract group has
a GS quotient $Q$ which is algorithmically finite (this means that no
algorithm can produce an infinite set of pairwise distinct elements in $Q$).
\end{Theorem}

For the motivation and proofs of these results the reader is referred to the respective papers 
(the first two theorems will be mentioned again in \S~\ref{sec:T}). Here we remark that verification of condition (ii) in the proofs of these three theorems
is not as straightforward as it  was in Theorem~\ref{GS_Wilson}. In particular, the set of additional relators $R_{\eps}$
cannot be described ``right away''; instead it is constructed via certain iterated
process.
\vskip .1cm

We finish this section with two useful technical results, which are also based on 
Lemma~\ref{GGS_quotbase}.

\begin{Lemma}[Tails Lemma]
\label{cuttails} Let $G$ be a GGS pro-$p$ group. Let $\Lambda$ and $\Gamma$
be countable subgroups of $G$ with $\Lambda\subseteq \Gamma$ and $\Lambda$ dense
in $\Gamma$. Then $G$ has a GGS quotient $G'$ such that $\Lambda$ and $\Gamma$
have the same image in $G'$. 
\end{Lemma}
\begin{proof} Let $W$ be a valuation on $G$ such that $def_W(G)>0$. 
Since $\Lambda$ is countable and dense in $\Gamma$ and $W$ is continuous, 
for each $g\in \Gamma$, we can choose $l_g\in \Lambda$ such that if
$R=\{l_g^{-1} g : g\in \Gamma\}$, then $W(R)<def_W(G)$. It is clear that
the group $G/\la R\ra^G$ has the required property.
\end{proof}
\begin{Remark} The terminology `tails lemma' is based on the following ``visualization''
of the above procedure: we represent each element $g\in \Gamma$ as $l_g\cdot (l_g^{-1} g)$
where $l_g$ is a good approximation of $g$ by an element of $\Lambda$ and $l_g^{-1} g$
is a tail of $g$ (which is analogous to a tail of a power series). The desired quotient
$G'$ of $G$ is constructed by cutting all the tails.
\end{Remark}

\begin{Lemma}
\label{infgen} Let $G$ be a GGS pro-$p$ group.
Then some quotient $Q$ of $G$ has a weighted presentation $(X,R,W)$ 
such that $def_W(X,R)>0$ and $X$ is finite (in particular, $Q$
is a finitely generated GGS pro-$p$ group).
\end{Lemma}
\begin{proof} By definition, $G$ has a weighted presentation $(X_0,R_0,W)$
with $def_{W}(X_0,R_0)>0$. Choose a finite subset $X\subseteq X_0$ such that
$W(X\setminus X_0)<def_{W}(X_0,R_0)$. Then it is easy to check that the group
$Q=G/\la X\setminus X_0\ra^G$ has the required property.
\end{proof}

\section{Free subgroups in generalized Golod-Shafarevich pro-$p$ groups}
\label{sec:Zelmanov}

As we saw in \S~\ref{sec:GSG}, Golod-Shafarevich abstract groups
may be torsion and therefore need not contain free subgroups. 
In this section we shall discuss a remarkable theorem of Zelmanov~\cite{Ze1} which
asserts that Golod-Shafarevich {\it pro-$p$} groups always contain
non-abelian free pro-$p$ subgroups. In fact, we will show that the proof
of this result easily extends to generalized Golod-Shafarevich pro-$p$ groups.

\begin{Theorem}[Zelmanov]
\label{thm:Zelmanov}
 Every generalized Golod-Shafarevich pro-$p$ group contains  
a non-abelian free pro-$p$ subgroup.
\end{Theorem}

It is not a big surprise that Golod-Shafarevich pro-$p$ groups contain non-abelian
free abstract groups since the latter property seems to hold for all known examples
of non-solvable pro-$p$ groups. However, containing a non-abelian free pro-$p$ subgroup
is a really strong property for a pro-$p$ group. For instance, pro-$p$ groups linear
over $\dbZ_p$ or $\dbF_p[[t]]$ cannot contain non-abelian free pro-$p$ subgroups \cite{BL},
and it is conjectured that the same is true for pro-$p$ groups linear over any pro-$p$ ring.

We start with some general observations. Let $G$ be a finitely generated pro-$p$ group.
There is a well-known technique for proving that $G$ contains a non-abelian free abstract
subgroup. Let $F(2)$ be the free abstract group of rank $2$, and suppose that $F(2)$ does not
embed into $G$. Then for any $g,h\in G$ there exists a non-identity word $w\in F(2)$
such that $w(g,h)=1$. Thus 
$$G\times G=\cup_{w\in F(2)\setminus\{1\}} (G\times G)_w \mbox{ where }
(G\times G)_w=\{(g,h)\in G\times G :  w(g,h)=1\}. \eqno (***)$$
It is easy to see that each subset $(G\times G)_w$ is closed in $G\times G$,
and since $F(2)$ is countable, while $G\times G$ is complete (as a metric space),
Baire category theorem and (***) imply that for some $w\in F(2)\setminus\{1\}$, the set 
$(G\times G)_w$ is open in $G\times G$, so in particular, it contains a coset of some open subgroup.
The latter has various strong consequences (e.g. it implies that the Lie algebra
$L(G)$ satisfies an identity), which in many cases contradicts some known property of $G$.

The following technical result is a (routine) generalization of \cite[Lemma~1]{Ze1} 
from GS to GGS algebras and is proved similarly to Theorem~\ref{GS_second}.
\begin{Lemma}\rm
\label{GolodZelmanov} Let $K$ be a countable field, and let
$A$ be a generalized Golod-Shafarevich complete filtered algebra,
that is, $A$ has a presentation $\la U| R\ra$ s.t. $1-H_{U,d}(\tau)+H_{R,d}(\tau)<0$
for some $\tau\in (0,1)$ and some degree function $d$ on $K\lla U\rra$.
Let $A_{abs}$ be the (abstract) subalgebra of $A$ (without 1) generated by $U$. 
Then there exist an epimorphism $\pi:A\to A'$, with $A'$ also
GGS, and a function $\nu:\dbN\to\dbN$ such that for any $n\in\dbN$, any $n$ elements 
of $\pi(A_{abs}^{\nu(n)})$ generate a nilpotent subalgebra.
\end{Lemma}

\subsection{Sketch of proof of Theorem~\ref{thm:Zelmanov}}

The proof of Theorem~\ref{thm:Zelmanov} roughly consists of two parts --
reducing the problem to certain question about associative algebras
(Proposition on p.227 in \cite{Ze1}) and then proving the proposition.
This proposition is actually the deeper part of Zelmanov's theorem,  but 
since it is not directly related to GS groups or algebras
and its proof is somewhat technical, we have chosen to skip this part
in our survey and concentrate on the first part of the proof. This
will be sufficient to make it clear that the proof applies to GGS groups
and not just GS groups as stated in \cite{Ze1}.
\vskip .12cm

Let $G$ be a GGS pro-$p$ group and assume that it does not contain a
free pro-$p$ group of rank $2$, denoted by $F_2$. As above, denote by $F(2)$
the free abstract group of rank $2$. Let $D$ denote the standard degree function 
defined in \S~\ref{subsec:GSGdef} (not the degree function which makes $G$ a GGS group!)
In this proof we shall use the definition of $D$ in terms of the Zassenhaus filtration
(see Proposition~\ref{prop:Zass}).

{\it Step 1:} Since $F_2$ is uncountable, the above approach for proving the existence
of a non-abelian free abstract subgroup cannot be applied directly.
However we can still say that for any $g,h\in G$ there exists a non-identity element $w\in F_2$ (this time $w$
may be an infinite word) such that $w(g,h)=1$. Equivalently, for any $g,h\in G$ there is a 
non-identity word $w\in F(2)$ such that $w(g,h)=w'(g,h)$ for some $w'\in F_2$ with $D(w')>D(w)$.

{\it Step 2:} The same application of the Baire category theorem as above implies that there is
$w\in F(2)$, elements $g_0,h_0\in G$ and an open subgroup $K$ of $G$ such that 
for any $g,h\in K$ we have $w(g_0 g,h_0 h)=w'(g_0 g,h_0 h)$ for some $w'\in F_2$ 
(depending on $g$ and $h$) with $D(w')>D(w)$. 

{\it Step 3:} Since Steps 1 and 2 can be applied to any open subgroup of $G$, we can assume
in Step~2 that $W(g_0)$ and $W(h_0)$ are as small as we want, where $W$ is a valuation on $G$
s.t. $def_W(G)>0$. In particular, by Lemma~\ref{GGS_quotbase}, we can ensure
that the group $G'=G/\la g_0,h_0\ra^G$ is also GGS.  The image of $K$ in $G'$, call it $K'$,
is also GGS by Theorem~\ref{GGS_open}. Thus, replacing $G$ by $K'$ (and changing the notations), we can assume that
\vskip .15cm
\centerline{for any $g,h\in G$ there is $w'\in F_2$, with $D(w')>D(w)$, s.t.
$w(g,h)=w'(g,h)\,\,\,\,$ (***)}
\vskip .15cm
{\it Step 4:} If $k=D(w)$, then we can multiply $w$ by any element of $D_{k+1}F_2$
without affecting (***). In this way we can assume that $w$ is a product of elements 
$c^{p^e}$ where $c$ is a left-normed commutator of degree $k/p^e$. Next note that if some 
$w$ satisfies (***) and $v\in F_2$ is such that $D([w,v])=D(v)+D(w)$, then (***) still
holds with $w$ replaced by $[w,v]$. After applying this operation several times, we can assume that
$w=c_1\ldots c_t$ where each $c_i$ is a left-normed commutator of length $D(w)$.

{\it Step 5:} Now let $L_2$ be the free $\dbF_p$-Lie algebra of rank $2$
and $Lie(w)=Lie(c_1)+\ldots +Lie(c_t)\in L_2$ where $Lie(c_i)$ is the Lie commutator
corresponding to $c_i$. It is  not difficult to see that condition (***)
can now be restated as the following equality in $\Fp[[G]]$: for any $g,h\in G$ we have 
$$Lie(w)(g-1, h-1)=O_{k+1}(g-1, h-1)$$ where $Lie(w)(g-1,h-1)$ is simply
the element $Lie(w)$ evaluated at the pair $(g-1,h-1)$ and, for a finite set
of elements $a_1,\ldots, a_s$, $O_{k+1}(a_1,\ldots,a_s)$ is a (possibly
infinite) but converging sum of products of $a_1,\ldots, a_s$, with each product of length
$\geq k+1$ (recall that $k=D(w)$).

Thinking of $Lie(w)$ as an element of the free associative $\Fp$-algebra of rank $2$,
we can consider the full linearization of $Lie(w)$, call it $f$. Then $f$ is a polynomial of degree $k$
in $k$ variables, and it is easy to check that for any $g_1,\ldots, g_k\in G$
we have $f(g_1-1,\ldots, g_k-1)=O_{k+1}(g_1-1,\ldots, g_k-1)$.

{\it Step 6:} Let $(X, R)$ be a presentation for $G$ satisfying the GGS condition.
Then the algebra $A=\Fp[[G]]$ is GGS (with presentation $(U, R_{alg})$ where
$U=\{x-1: x\in X\}$ and $R_{alg}=\{r-1: r\in R\}$). Let us apply Lemma~\ref{GolodZelmanov} to $A$,
and let $\pi: A\to A'$ be as in the conclusion of that lemma. Note that we do not know
whether the group $G'=\pi(G)$ is GGS, but the fact that $A'$ is a GGS algebra will be sufficient.
Let $B=\pi(A_{abs})$  (in the notations of Lemma~\ref{GolodZelmanov}), and let $\Gamma$
be the abstract subgroup of $G'$ generated by (the image of) $X$. Note that 
$\Gamma\subset 1+B=\{1+b: b\in B\}$, and moreover $D_m\Gamma\subset 1+B^m$ for all $m\in\dbN$.
Therefore, we have the following (with part (ii) being a consequence of Step~5). 
\begin{itemize}
\item[(i)] There exists a function $\nu:\dbN\to\dbN$ such that for any $n\in\dbN$,
any $n$ elements of $B^{\nu(n)}$ generate a nilpotent subalgebra.
\item[(ii)] There exists a multilinear polynomial $f$ of degree $k$ such that for any 
$g_1,\ldots, g_k\in D_{\nu(k)}\Gamma$, the element $f(g_1-1,\ldots, g_k-1)$
is equal to a finite sum of products of $g_1-1,\ldots, g_k-1$, with each product
of length at least $k+1$ (the sum must be finite by (i)).
\end{itemize}

As proved in \cite[Proposition, p.227]{Ze1}, if $B$ is a finitely generated $\Fp$-algebra and
$\Gamma$ is a subgroup of $1+B$ such that (i) and (ii) above hold, then $B$ is nilpotent
(and hence finite-dimensional). This yields the desired contradiction since
in our setting $\Fp+B$ is a dense subalgebra of $A'=\pi(A)$, which is GGS and
therefore infinite-dimensional. This concludes our sketch of proof of Theorem~\ref{thm:Zelmanov}.
\vskip .15cm
As we already mentioned, the characteristic zero counterpart of Theorem~\ref{thm:Zelmanov} was established by Kassabov in \cite{Ka}, who showed that every  Golod-Shafarevich prounipotent group contains a non-abelian free prounipotent subgroup.

\section{Subgroup growth of generalized Golod-Shafarevich groups}

In this section we shall discuss various results about subgroup growth of
GGS groups. We shall pay particular attention to this topic in this paper
not only because of its intrinsic importance, but also because there are two
classes which contain many Golod-Shafarevich groups -- Galois groups $G_{K,p,S}$ 
and fundamental groups of hyperbolic $3$-manifolds -- where subgroup growth 
has a direct number-theoretic (resp. topological) interpretation.

We shall restrict our discussion to GGS {\it pro-$p$} groups.
Since the majority of the results we state deal with lower bounds on subgroup growth
and the subgroup growth of an abstract group $\Gamma$ is bounded below by the subgroup
growth of its pro-$p$ completion, these results yield the corresponding lower bounds
for the subgroup growth of GGS abstract groups. 

\subsection{Some generalities on subgroup growth}
If $G$ is a finitely generated pro-$p$ group, denote by $a_m(G)$ the number of open subgroup of $G$
of index $m$ (note that $a_m(G)=0$  unless $m$ is a power of $p$). The asympotic behaviour
of the sequence $\{a_{p^k}(G)\}_{k\geq 1}$ is closely related to that of the sequence
$\{r_k(G)\}$ defined below, the latter being much easier to control.

\begin{Lemma} 
\label{sg_basicbound}
Let $G$ be a finitely generated pro-$p$ group. For each $k\in\dbN$ let 
$$r_k(G)=\max\{d(U): U \mbox{ is an open subgroup of } G\mbox{ of index }p^k\}.$$
Then
\begin{itemize}
\item[(i)] $a_{p^k}(G)\geq p^{r_{k-1}(G)}-1$;
\item[(ii)] $a_{p^k}(G)\leq a_{p^{k-1}}(G)\cdot (p^{r_{k-1}(G)}-1)$, and therefore
 $a_{p^k}(G)\leq p^{\sum_{i=0}^{k-1} r_{i}(G)}$.
\end{itemize}
\end{Lemma}
\begin{proof} (i) Let $U$ be an open subgroup of index $p^{k-1}$ with $d(U)=r_{k-1}(G)$.
The quotient $U/[U,U]U^p$ is a vector space over $\Fp$ of dimension $d(U)$
and therefore has $p^{d(U)}-1$ subspaces of codimension $1$. These subspaces correspond to
subgroups of index $p$ in $U$, and each of those subgroups has index $p^k$ in $G$.

(ii) follows from the same argument and the fact that each subgroup of index $p^k$
in a pro-$p$ group is contained in a subgroup of index $p^{k-1}$.
\end{proof}

By the Schreier index formula, the sequence $\{r_k(G)\}$ grows at most linearly in $p^k$. Hence,
Lemma~\ref{sg_basicbound} shows that the subgroup growth of any finitely generated
pro-$p$ group $G$ is at most exponential, and the subgroup growth is exponential
if and only if $\inf\limits_k r_k(G)/p^k >0$. In fact, there is an even more
elegant characterization of exponential subgroup growth due to Lackenby~\cite[Theorem~8.1]{La3}.

\begin{Definition}\rm $\empty$
\begin{itemize}
\item[(a)] Let $G$ be a finitely generated pro-$p$ group and $\{G_n\}$ a strictly
descending chain of open normal subgroups of $G$. We will say that
$\{G_n\}$ is an {\it LRG chain} (where LRG stands for linear rank growth)
if $\inf\limits_n (d(G_n)-1)/[G:G_n]>0$.
\item[(b)] Let $\Gamma$ be a finitely generated abstract group and $\{\Gamma_n\}$ a strictly
descending chain of normal subgroups of $\Gamma$ of $p$-power index. We will say that
$\{\Gamma_n\}$ is an {\it LRG $p$-chain} if $\inf\limits_n (d_p(\Gamma_n)-1)/[\Gamma:\Gamma_n]>0$.
(Recall that $d_p(\Lambda)=d(\Lambda/[\Lambda,\Lambda]\Lambda^p)=d(\Lambda_{\phat})$
for an abstract group $\Lambda$.)
\end{itemize}
\end{Definition}

\begin{Theorem}[Lackenby] Let $G$ be a finitely generated pro-$p$ group. The following are equivalent.
\begin{itemize}
\item[(i)] $G$ has exponential subgroup growth
\item[(ii)] There is $c>0$ such that $r_k(G)>cp^k$ for all $k$.
\item[(iii)] There is $c>0$ such that $r_k(G)>cp^k$ for infinitely many $k$.
\item[(iv)] $G$ has an LRG chain.
\end{itemize} 
\end{Theorem}
\begin{proof} The equivalence of (i) and (ii) has already been discussed.
The implication ``(iv)$\Rightarrow$ (iii)'' is clear. The implication ``(iii)$\Rightarrow$ (ii)''
follows from the fact that the quantity $(d(U)-1)/[G:U]$ does not increase if
$U$ is replaced by its open subgroup, and hence the sequence $\frac{r_k(G)-1}{p^k}$ is non-increasing. Finally, the implication ``(ii)$\Rightarrow$ (iv)'' is established by a Cantor diagonal argument (see \cite[Theorem~8.1]{La3}
for details).
\end{proof}

\subsection{Subgroup growth of GGS groups}
As we will see later in the paper, many naturally occurring GS groups have LRG chains
and therefore have exponential subgroup growth. It is very likely that there exist
GS groups with subexponental subgroup growth, but to the best of our knowledge,
this problem is still open. The best currently known lower bound on the subgroup growth of 
GS groups (which also applies to GGS groups) is due to Jaikin-Zapirain~\cite[Appendix~B]{EJ2} 
and is best stated in terms of the sequence $\{r_k(G)\}$.

\begin{Theorem}[Jaikin-Zapirain]
\label{subgpgrowth_Jaikin} Let $G$ be a finitely generated generalized Golod-Shafarevich
pro-$p$ group. Then there exists a constant $\beta=\beta(G)>0$ such that
$r_k(G)>p^{k^{\beta}}$ for infinitely many $k$.
\end{Theorem}
\begin{proof}
Let $(X,R)$ be a presentation of $G$ and $D$
an integer-valued degree function on $F=F_{\phat}(X)$ with respect to $X$ such that
$H_{X,D}(\tau)-H_{R,D}(\tau)-1>0$ for some $\tau\in (0,1)$.
Let $\pi:F\to G$ be the natural projection, and for each $n\in\dbN$
let $$G_n=\{g\in G: g=\pi(f) \mbox{ for some }f\in F\mbox{ with }D(f)\geq n\}
\mbox{ and } c_n=\log_p[G_n:G_{n+1}]$$ (Thus, $G_n=G^{n,D}$ and $c_n=c^{n,D}(G)$ in the notations
of \S~\ref{subsec:Quillen}).

By Theorem~\ref{GS_finitary}, there exists a presentation $(X_n,R_n)$ of $G_{n}$, with $X_n\subset F_{\phat}(X)$,
such that $H_{X_n,D}(\tau)-H_{R_n,D}(\tau)-1\geq (H_{X,D}(\tau)-H_{R,D}(\tau)-1)\prod_{i=0}^{n-1}
\left(\frac{1-\tau^{pi}}{1-\tau^i}\right)^{c_{i}}$, whence
\begin{multline*}\log_p(H_{X_n,D}(\tau)-H_{R_n,D}(\tau)-1)\geq
\log_p(H_{X,D}(\tau)-H_{R,D}(\tau)-1)+c_{n-1}\log_p(1+\tau^{n-1})
\end{multline*}
By Lemma~\ref{lem:lowerbound}(a), $d(G_n)\geq H_{X_n,D}(\tau)-H_{R_n,D}(\tau)$,
whence
$$\log_p d(G_n)\geq c_{n-1}\log_p(1+\tau^{n-1})+\log_p(H_{X,D}(\tau)-H_{R,D}(\tau)-1)\geq
\frac{1}{2\,\log p} c_{n-1}\tau^{n-1}-E,$$ 
where $E$ is a constant independent of $n$.

Now take $0<\tau_1<\tau$ such that $H_{X,D}(\tau_1)-H_{R,D}(\tau_1)-1>0$.
By Theorem~\ref{thm:Quillen}, the infinite product 
$\prod_{i=0}^{\infty}
\left(\frac{1-\tau_1^{pi}}{1-\tau_1}\right)^{c_i}$ diverges, whence
the series $\sum_{i=0}^{\infty} c_i\tau_{1}^i$ also diverges.
Thus if $c=\limsup \sqrt[i]{c_i}$, then $c\tau_1\geq 1$,
so $c\tau>1$. Now choose any $\alpha\in (1, c\tau)$. Then
$\limsup \frac{ c_{n-1}\tau^{n-1}} {\alpha^{n}}=\infty$,
whence $$\log_p\, d(G_n)\geq \alpha^n \mbox{ for infinitely many }n. \eqno (***)$$

On the other hand, the trivial upper bound $c_n\leq |X|^n$ implies that
$\log_p[G:G_n]= \sum_{i=0}^{n-1}c_i<|X|^n$. Thus, if we let $k=\log_p [G:G_n]$
where $n$ satisfies (***), then 
$$\log_p {\,r_k(G)}\geq \log_p\, d(G_n)\geq |X|^{n\, \frac{\log\alpha}{\log|X|}}>k^{\beta}
\mbox{ for }\beta=\frac{\log\alpha}{\log|X|}.$$
\end{proof}

\subsection{Subgroup growth of groups of non-negative deficiency}
In the proof of Theorem~\ref{subgpgrowth_Jaikin} we used Theorem~\ref{GS_finitary}
as a numerical inequality. The fact that it holds
as inequality of power series also has a very interesting consequence.

{\bf Notation.} Given positive integers $n,m$ and $p$, define ${n \choose m}_p$ to be the
coefficient of $t^m$ in the polynomial $(1+t+\ldots+t^{p-1})^n$. Thus,
${n \choose m}_2={n \choose m}$ is the usual binomial coefficient.

\begin{Proposition}
\label{Lackenby_key}
Let $G$ be a finitely presented pro-$p$ group and $K$ an open subgroup  
of $G$ containing $\Phi(G)=[G,G]G^p$, so that $G/K\cong (\dbZ/p\dbZ)^n$ for some $n$. 
Then for any integer $0\leq l<(p-1)n$
the following inequality holds: 
$$d(K)\geq d(G)\sum_{i=0}^l {n \choose i}_p-r(G)\sum_{i=0}^{l-1} {n \choose i}_p-
\sum_{i=0}^{l+1} {n \choose i}_p.$$
\end{Proposition}
\begin{proof} Let $(X,R)$ be a minimal presentation of $G$ (so that 
$|X|=d(G)$ and $|R|=r(G)$). We shall apply Theorem~\ref{GS_finitary}
to this presentation and the standard degree function $D$ on $F_{\phat}(X)$. 
Since (in the notations of \S~\ref{subsec:Quillen}) $G^{2,D}=[G,G]G^p$,
we have $G^{2,D}\subseteq K$, so $c^{1,D}(G/K)=n$ and $c^{i,D}(G/K)=0$ for $i>1$.
Thus, by Theorem~\ref{GS_finitary}, $K$ has a presentation $(X',R')$ such that
$$\frac{H_{X',D}(t)-H_{R',D}(t)-1}{1-t}\geq \frac{d(G)t-H_{R,D}(t)-1}{1-t}\left(1+t+\ldots+t^{p-1}\right)^n.\eqno (***)$$
Let us write $H_{X',D}(t)=\sum d_i(K) t^i$, $H_{R',D}(t)=\sum r_i(K) t^i$ and $H_{R,D}(t)=\sum r_i t^i$.
Note that $r_1=0$ by Lemma~\ref{prop_basic}(i) since the presentation $(X,R)$ is minimal and $\sum_{i\geq 2}r_i=r(G)$.

Computing the coefficient of $t^{l+1}$ on both sides of (***), we obtain
$$\sum_{i\leq l+1}d_i(K)-\sum_{i\leq l+1}r_i(K)-1\geq d(G)\sum_{i=0}^l {n \choose i}_p-r(G)\sum_{i=0}^{l-1} {n \choose i}_p- \sum_{i=0}^{l+1} {n \choose i}_p.$$ To finish the proof it suffices to show that 
$d(K)\geq \sum_{i\leq l+1}d_i(K)-\sum_{i\leq l+1}r_i(K)$. To prove the latter we take
$\tau\in (0,1)$ and apply Lemma~\ref{lem:lowerbound}(b) to the weighted presentation $(X,R,W)$ 
where $W$ is the uniform weight function on $F_{\phat}(X)$ with respect to $X$ such that 
$W(x)=\tau$ for all $x\in X$. The desired inequality follows by letting $\tau$ tend to $1$.
\end{proof}

The inequality in Proposition~\ref{Lackenby_key} was proved by Lackenby~\cite[Theorem~1.6]{La2} for $p=2$,
and in a slightly weaker form for arbitrary $p$. Lackenby's proof was
based on clever topological arguments, and it is remarkable that the
finitary Golod-Shafarevich inequality~\eqref{GSineq_finite2} yields the same result
for $p=2$. \footnote{The above proof of Proposition~\ref{Lackenby_key} was outlined by Kassabov
during an informal discussion at the workshop ``Lie Groups, Representations and Discrete Mathematics''
at IAS, Princeton in February 2006, before Theorem~\ref{GS_finitary} was formally proved in \cite{EJ2}.}

There are many different ways in which Proposition~\ref{Lackenby_key} may be used.
A very important application, discovered by Lackenby for $p=2$, deals with the case
when $r(G)\leq d(G)$ and $K=\Phi(G)$.

\begin{Corollary} 
\label{Cor_Lacksg}
Let $G$ be a finitely presented pro-$p$ group with $r(G)\leq d(G)$
and $d(G)\geq 36p^2$. Then $d(\Phi(G))\geq \frac{1}{6}\sqrt{d(G)} p^{d(G)-1}$.
\end{Corollary}
\begin{proof} Applying Proposition~\ref{Lackenby_key} with $K=\Phi(G)$ (so that $n=d(G)$)
and $l=[(p-1)n/2]$, we get $d(\Phi(G))\geq n\cdot {n \choose l}_p-\sum_{i=0}^{l+1} 
{n \choose i}_p$.
Note that $\sum_{i=0}^{l+1} {n \choose i}_p\leq \sum_{i=0}^{n(p-1)} {n \choose i}_p=p^n$.
We claim that
\begin{equation}
\label{centrallimit}
{n \choose l}_p\geq \frac{1}{3\sqrt{n}} p^{n-1}.
\end{equation}
If $p=2$, this is an easy consequence of Stirling formula, and for $p>2$
this could be proved, for instance, as follows. Consider independent identically
distributed random variables $X_1,\ldots, X_n$ which take on integer values
$0,1,\ldots, p-1$ with equal probabilities $1/p$, and let $S_n=X_1+\ldots+X_n$.
Then ${n \choose l}_p=p^n \cdot Prob (S_n=l)$, and note that
$l=(p-1)n/2$ is the expected value of $S_n$.

Each $X_i$ has variance $\sigma=\sqrt{(p-1)(2p-1)/6}$, so by the central limit
theorem we have
$$Prob\left(|S_n-l|<\frac{1}{2}\sigma\sqrt{n}\right)\geq \frac{1}{\sqrt{2\pi}}\int_{-1/2}^{1/2}e^{-x^2/2}\geq \frac{e^{-1/8}}{\sqrt{2\pi}}.$$
On the other hand, an easy induction on $n$ shows that ${n \choose i}_p$ is 
increasing as a function of $i$ for $0\leq i\leq l$ and (by symmetry) decreasing for 
$l\leq i\leq 2l=(p-1)n$. Hence
$$Prob(S_n=l)\geq \frac{e^{-1/8}}{\sqrt{2\pi}(1+\sigma\sqrt{n})}\geq 
\frac{1}{3+2p\sqrt{n}}\geq \frac{1}{3p\sqrt{n}},$$
which yields \eqref{centrallimit}.
Hence $$d(\Phi(G))\geq \frac{(\sqrt{n}-3p)p^{n-1}}{3}\geq \frac{1}{6} \sqrt{n}p^{n-1}.$$ 
\end{proof}
\begin{Remark} The inequality $d(G)\geq 36p^2$ can be significantly weakened using a more careful estimate.
In particular, if $p=2$, it is enough to assume that $d(G)\geq 4$, as proved in \cite{La2}. 
\end{Remark}

Note that if $G$ is a free pro-$p$ group, then $d(\Phi(G))-1=(d(G)-1)p^{d(G)}$, so
the ratio $d(\Phi(G))/d(G)$ guaranteed by Corollary~\ref{Cor_Lacksg} is not far from the
best possible. In particular, it yields a very good bound on the subgroup growth
of groups $G$ for which $d(U)\geq r(U)$ for every open subgroup $U$ and $d(U)>p^3$
for some open subgroup $U$, and this class includes pro-$p$ completions of all hyperbolic 
3-manifold groups (see \S~\ref{sec:geom} for details).

\section{Groups of positive power $p$-deficiency}
\label{sec:powerdef}
  
In this short section we will briefly discuss groups of positive power $p$-deficiency
which are close relatives of Golod-Shafarevich groups. These groups provide very
simple counterexamples to the general Burnside problem, and it is quite amazing that
they had been discovered just two years ago by Schlage-Puchta~\cite{SP} and in a slightly different
form by Osin~\cite{Os}.

\begin{Definition}\rm Let $p$ be a fixed prime number.
\begin{itemize}
\item[(i)] Let $F$ be a free abstract group. Given $f\in F$, we let $\nu_p(f)$ be the largest
non-negative integer such that $f=h^{p^{\nu_p(f)}}$ for some $h\in F$.
\item[(ii)] If $(X,R)$ is an abstract presentation, with $|X|<\infty$, 
we define its {\it power $p$-deficiency},
denoted by $def_p(X,R)$ by $$def_p(X,R)=|X|-1-\sum_{r\in R}p^{-\nu_p(r)}.$$
\item[(iii)] If $G$ is an abstract group, its power $p$-deficiency $def_p(G)$
is defined to be the supremum of the set $\{def_p(X,R)\}$ where $(X,R)$ runs over
all presentations of $G$.
\end{itemize}
\end{Definition}

The key property of power $p$-deficiency is the inequality in part (b) of the following theorem
which is analogous to the corresponding inequalities for usual deficiency (Theorem~\ref{thm:Schreier}(b))
and weighted deficiency (Theorem~\ref{thm:weightedSchreier}(b)).

Recall that for a finitely generated abstract group $G$ 
we set $d_p(G)=d(G/[G,G]G^p)$ and that $d_p(G)=d(G_{\phat})$ where $G_{\phat}$
is the pro-$p$ completion of $G$.
\begin{Theorem}
\label{powerdef_index} 
Let $G$ be a finitely generated abstract group. The following hold:
\begin{itemize}
\item[(i)] $d_p(G)\geq def_p(G)+1$
\item[(ii)] Let $H$ be a subnormal subgroup of $G$ of $p$-power index. Then $$def_p(H)\geq def_p(G) [G:H]$$
and therefore $$\frac{d_p(H)-1}{[G:H]}\geq def_p(G).\eqno (***)$$
\end{itemize}
\end{Theorem}
\begin{proof} Relations which are $p$-powers do not affect $d_p(G)$, so (i) follows
from the fact that if $G=\la X|R \ra$, then $d_p(G)\geq |X|-|R|$.
To establish (ii), by multiplicativity of index it is enough to consider the case when
$H$ is a normal subgroup of index $p$. This case is covered by \cite[Theorem~2]{SP},
and the proof of this result is quite short unlike Theorem~\ref{thm:weightedSchreier}(b).
\end{proof}

An immediate consequence of Theorem~\ref{powerdef_index} is that groups of positive $p$-deficiency
are infinite; in fact they must have infinite pro-$p$ completion.
Indeed, suppose that $def_p(G)>0$, but the pro-$p$ completion of $G$ is finite. Then there exists a 
minimal subnormal subgroup of $p$-power index, call it $H$; then $d_p(H)=0$ which contradicts (***).
On the other hand, it is clear that there exist torsion groups of positive power $p$-deficiency, 
so in this way one obtains a short elementary self-contained proof of the existence of infinite
finitely generated torsion groups.

As suggested by the titles of both \cite{SP} and \cite{Os}, the original motivation for
introducing groups of positive power $p$-deficiency was to find examples of torsion finitely
generated groups with positive rank gradient.

\begin{Definition}\rm $\empty$
\begin{itemize}
\item[(i)] Let $G$ be a finitely generated abstract or pro-$p$ group. The {\it rank gradient} 
of $G$ is defined as $RG(G)=\inf\limits_H\frac{d(H)-1}{[G:H]}$ where $H$ runs over
all finite index subgroups of $G$.
\item[(ii)] Let $G$ be a finitely generated abstract group. The {\it $p$-gradient}
(also known as mod $p$ homology gradient) $RG_p(G)$ is defined as
$RG_p(G)=\inf\limits_H\frac{d_p(H)-1}{[G:H]}$ where $H$ runs over
all finite index subnormal subgroups of $p$-power index in $G$.
\end{itemize}
\end{Definition}
\begin{Remark}
Since subnormal subgroups of $p$-power index are precisely the subgroups
open in the pro-$p$ topology, it is easy to show that if $G$ is an abstract group,
then $RG_p(G)=RG(G_{\phat})$, that is, the $p$-gradient of $G$ is equal to the
rank gradient of its pro-$p$ completion. This also implies that $RG_p(G)=RG_p(G')$
if $G'$ is the image of $G$ in its pro-$p$ completion.
\end{Remark}
\vskip .2cm

Theorem~\ref{powerdef_index}(ii) asserts that groups of positive power $p$-deficiency
have positive $p$-gradient. Combining this result with theorems from
\cite{La3} and \cite{AJN}, one obtains the following corollary:

\begin{Corollary}[\cite{SP}] 
\label{Puchta1}
Let $G$ be an abstract group of positive power $p$-deficiency.
The following hold:
\begin{itemize}
\item[(a)] $G$ is non-amenable. Moreover, the image of $G$ in its pro-$p$ completion
is non-amenable.
\item[(b)] If $G$ is finitely presented, then $G$ is large, that is,
some finite index subgroup of $G$ maps onto a non-abelian free group.
\end{itemize}
\end{Corollary}
\begin{proof} (a) Note that $G$ has a $p$-torsion quotient $Q$ with $def_p(Q)>0$
(this is proved in the same way as the analogous result for Golod-Shafarevich groups --
see Theorem~\ref{GS_Wilson}). 
Let $Q'$ be the image of $Q$ in its pro-$p$ completion. We claim
that$$RG(Q')\geq RG_p(Q')=RQ_p(Q)\geq def_p(Q)>0.$$
Indeed, $RG(Q')\geq RG_p(Q')$ since $Q'$ is $p$-torsion, so
every finite index normal subgroup of $Q'$ is of $p$-power index,
and therefore every finite index subgroup of $Q'$ is subnormal of $p$-power index.
The equality $RG_p(Q')=RQ_p(Q)$ holds by the remark following the definition
of $p$-gradient and $RQ_p(Q)\geq def_p(Q)$ by Theorem~\ref{powerdef_index}(ii).

Thus, $Q'$ is a residually finite group with positive rank gradient and 
therefore cannot be amenable as proved in \cite{AJN}. If $G'$ is the image of
$G$ in its pro-$p$ completion, then $Q'$ is a quotient of $G'$, so $G'$ is also
non-amenable.

(b) follows from a theorem of Lackenby~\cite[Theorem~1.18]{La4}, which asserts that a finitely
presented group with positive $p$-gradient is large.
\end{proof}

\begin{Corollary}
\label{nonamrf} There exist residually finite torsion non-amenable groups.
\end{Corollary}
\begin{proof} If $G$ is any torsion group of positive power $p$-deficiency,
the image of $G$ in its pro-$p$ completion has the desired property by Corollary~\ref{Puchta1}.
\end{proof}
Another construction of residually finite torsion non-amenable groups will be given in \S~\ref{sec:T}.
Recall that the first examples of torsion non-amenable groups (which were not residually
finite) were Tarski monsters constructed by Ol'shanskii~\cite{Ol} (with their
non-amenability proved in \cite{Ol3}).

\vskip .12cm
We finish this section with a brief comparison of GS groups and groups of positive power $p$-deficiency. 
As suggested by the definitions, the latter class should be smaller than that of GS groups since in the 
definition of power $p$-deficiency only relators of the form $f^{p^k}$, with $k$ large, are counted with small
weight, while in the definition of GS groups the set of relators counted with small weight also includes
long commutators (in addition to relators of the form $f^{p^k}$, with $k$ large). This heuristics suggests that
groups of positive power $p$-deficiency may always be Golod-Shafarevich, and this turns out
to be almost true:

\begin{Theorem}\cite{BT} Let $G$ be an abstract or a pro-$p$ group of positive power $p$-deficiency.
Then $G$ has a finite index Golod-Shafarevich subgroup. Moreover, if $p\geq 7$, then $G$ itself
must be Golod-Shafarevich.
\end{Theorem}

While being a smaller class than GS groups, groups of positive power $p$-deficiency
satisfy much stronger ``largeness'' properties, as we saw in this section. It would
be interesting to find some intermediate condition on groups which is significantly
weaker than having positive power $p$-deficiency, but has stronger consequences than being
Golod-Shafarevich.

\section{Applications in number theory}
\label{sec:nt}
\subsection{Class field tower problem}

Let us begin with the following natural number-theoretic question:

\begin{Question} 
\label{PID}
Let $K$ be a number field. Does there exist a finite
extension $L/K$ such that the ring of integers $O_L$ of $L$ is a PID?
\end{Question}

If $K$ is  a number field, the extent to which $O_K$ fails
to be a PID is measured by the ideal class group $Cl(K)$.
In particular, $O_K$ is a PID if and only if $Cl(K)$ is trivial.
The group $Cl(K)$ is always finite and by class field theory, $Cl(K)$ 
is isomorphic to the Galois  group $Gal(\dbH(K)/K)$ where $\dbH(K)$ 
is the maximal abelian unramified extension of $K$, called the
{\it Hilbert class field of $K$}.

\begin{Definition}\rm Let $K$ be a number field. The {\it class field tower}
$$K=\dbH^0(K)\subseteq \dbH^1(K)\subseteq \dbH^2(K)\subseteq \ldots$$ of $K$ is defined
by $\dbH^i(K)=\dbH(\dbH^{i-1}(K))$ for $i\geq 1$.
\end{Definition}

The class field tower of $K$ is called finite if
it stabilizes at some step and infinite otherwise.

\begin{Lemma} Let $K$ be a number field. Then the class field
tower of $K$ is finite if and only if there is a finite extension
$L/K$ with $Cl(L)=\{1\}$.
\end{Lemma}
\begin{proof} Let $\{K_i=\dbH^i(K)\}$ be the class field tower of $K$.

``$\Rightarrow$'' By assumption $K_n=K_{n+1}$ for some $n$, so 
$Cl(K_n)=\{1\}$  whence we can take $L=K_n$. 

``$\Leftarrow$'' Consider the tower of fields $L=L K_0\subseteq L K_1 \subseteq \ldots$.
Since for each $i$ the extension $K_{i+1}/K_i$ is abelian and unramified, the same is true
for the extension $L K_{i+1}/L K_i$. In particular, $LK_1$ is an abelian unramified extension
of $L$. But $Cl(L)=\{1\}$, so $L$ does not have such non-trivial extensions, which implies
that $LK_1=L$. Repeating this argument inductively, we conclude that $L K_i =L$ for each $i$,
so each $K_i$ is contained in $L$, whence the tower $\{K_i\}$ must be finite.
\end{proof}

Thus, Question~\ref{PID} is equivalent to the so called {\it class field tower problem}.

\begin{Problem}[Class field tower problem] Is it true that for any number field $K$ the class field tower of $K$ is finite?
\end{Problem}

Computing the class field of a given number field is a rather difficult task. It is a 
little bit easier to control the {\it $p$-class field}, where $p$ is a fixed prime.

\begin{Definition}\rm Let $p$ be a prime and $K$ a number field.
\begin{itemize}
\item[(a)] The {\it $p$-class field of $K$}, denoted by $\dbH_p(K)$, is the maximal
unramified Galois extension of $K$ such that the Galois group $Gal(\dbH_p(K)/K)$ is an elementary abelian
$p$-group.
\item[(b)] The {\it $p$-class field tower of $K$} is the ascending chain $\{\dbH^i_p(K)\}_{i\geq 0}$
defined by $\dbH^0_p(K)=K$ and $\dbH^i_p(K)=\dbH_p (\dbH^{i-1}_p(K))$ for $i\geq 1$.
\end{itemize}
\end{Definition}

It is easy to see that $\dbH^i_p(K)\subseteq \dbH^i(K)$ for any $K$ and $p$, so if the 
$p$-class field tower of $K$ is infinite for some $p$, then its class field tower must also
be infinite. Let $\dbH^{\infty}_p(K)=\cup_{i\geq 0} \dbH^i_p(K)$ be the union of all fields
in the $p$-class field tower of $K$ -- this is easily shown to be the maximal unramified pro-$p$ 
extension~\footnote{We say that a Galois extension $L/K$ is pro-$p$ if the Galois group
$Gal(L/K)$ is pro-$p$.} of $K$. Let $G_{K,p}=Gal(\dbH^{\infty}_p(K)/K)$. Thus, to solve the class field tower problem
in the negative it suffices to find an example where the group $G_{K,p}$ is infinite.
The latter problem can be solved using Golod-Shafarevich inequality since quite a lot
is known about the minimal number of generators and relators for the groups $G_{K,p}$.
\vskip .1cm

By definition of $\dbH_p(K)$, the Frattini quotient $G_{K,p}=G_{K,p}/[G_{K,p},G_{K,p}]G_{K,p}^p$ is isomorphic to $\Gal(\dbH_p(K)/K)$ which, by the earlier discussion, is isomorphic to $Cl(K)[p]=\{x\in Cl(K) : px=0\}$,
the elementary $p$-subgroup of $Cl(K)$. Let $\rho_p(K)=\dim Cl(K)[p]$.
Then
$$d(G_{K,p})=d(G_{K,p}/[G_{K,p},G_{K,p}]G_{K,p}^p)=\rho_p(K).$$
The following relation between the minimal number of generators and the minimal number of relators of $G_{K,p}$
was proved by Shafarevich~\cite[Theorem~6]{Sh}.

\begin{Theorem}[Shafarevich] 
\label{thm:shaf}
Let $K$ be a number field and $\nu(K)$ the number of infinite primes of $K$.
Then for any prime $p$ we have
$$0\leq r(G_{K,p})-d(G_{K,p})\leq \nu(K)-1.$$ 
\end{Theorem}

Combining Theorem~\ref{thm:shaf} with Theorem~\ref{GS_strong}(a),
we obtain the following criterion for the group $G_{K,p}$ to be infinite:

\begin{Corollary}[Golod-Shafarevich] In the above notations, assume that
\begin{equation}
\label{eq:gsnt}
\rho_p(K)> 2+2\sqrt{\nu(K)+1}.
\end{equation}
Then the group $G_{K,p}$ is Golod-Shafarevich and therefore
infinite. 
\end{Corollary}

To complete the negaitve solution to the class field tower problem it suffices to exhibit
examples of number fields satisfying the number-theoretic inequality \eqref{eq:gsnt}.
As we explain below, for any prime $p$ and $n\in\dbN$ there exists a number field $K=K(p,n)$
such that $[K:\dbQ]=p$ and $\rho_p(K)\geq n$. Since $\nu(K)\leq [K:\dbQ]$ by the Dirichlet unit theorem, 
such $K$ satisfies \eqref{eq:gsnt} whenever $n>2+2\sqrt{p+1}$.

For $p=2$ one can simply take any set of $n+1$ distinct odd primes $q_1,\ldots, q_{n+1}$
and let $K=\dbQ(\sqrt{\eps q_1\ldots q_{n+1}})$ where $\eps=\pm 1$ so that
$\eps q_1\ldots q_{n+1}\equiv 1 \mod 4$. If we choose $\eps_i=\pm 1$ for $1\leq i\leq n+1$
so that $\eps_i q_i\equiv 1\mod 4$, then the extension $\dbQ(\sqrt{\eps_1 q_1}, \ldots, \sqrt{\eps_{n+1} q_{n+1}})/K$
is unramified (which can be seen, for instance, by directly computing its discriminant). Since
this extension is abelian with Galois group $(\dbZ/2\dbZ)^n$, we have $\rho_2(K)\geq n$ 
(it is not hard to show that in fact $\rho_2(K)= n$).

For an arbitary $p$, we take $n+1$ distinct primes $q_1,\ldots, q_{n+1}$ congruent to $1$ mod $p$,
let $L_i=\dbQ(\zeta_{q_i})$ be the $q_i^{\rm th}$ cyclotomic field and $K_i$ the unique subfield of
degree $p$ over $\dbQ$ inside $L_i$. Let $L=L_1\ldots L_{n+1}$
and $M=K_1\ldots K_{n+1}$. Then $\Gal(L/\dbQ)\cong \oplus \Gal(L_i/\dbQ)$, so
$\Gal(M/\dbQ)\cong \oplus \Gal(K_i/\dbQ)\cong (\dbZ/p\dbZ)^{n+1}$. 
Clearly, $\Gal(M/\dbQ)$ has a subgroup of index $p$ which does not contain $\Gal(K_i/\dbQ)$
for any $i$. Equivalently, there exists a subfield 
$\dbQ\subset K\subset M$ such that $[K:\dbQ]=p$ and $K$ is not contained in the compositum of any proper 
subset of $\{K_1,\ldots, K_{n+1}\}$. We claim that each extension $KK_{i}/K$
is unramified. Indeed, $KK_{i}/K$ may only be ramified at $q_i$ since $L_i/\dbQ$ (and hence $K_i/\dbQ$)
is only ramified at $q_i$. If $KK_{i}/K$ is ramified at $q_i$, then $M/K$ is ramified at $q_i$,
which is impossible since $M/K=K\prod_{j\neq i} K_j/K$ and $K_j/\dbQ$ is unramified at $q_i$ for $j\neq i$.
Thus, each $KK_{i}$ is unramified over $K$, so their compositum $M$ is also unramified over $K$.
Since $M/K$ is abelian with $Gal(M/K)\cong (\dbZ/p\dbZ)^n$, we conclude that $\rho_p(K)\geq n$
(again, one can show that equality holds).

\subsection{Galois groups $G_{K,p,S}$}

The groups $G_{K,p}$ which arose in the solution to the class field tower problem
are very interesting in their own right and have been studied for almost a century,
yet their structure remains poorly understood. In fact, it is more natural to consider
a larger class of groups: given a number field $K$, a prime $p$ and a finite set
of primes $S$  of $K$, denote by $\dbH^{\infty}_{p,S}(K)$ the maximal pro-$p$ extension
of $K$ which is unramified outside of $S$, and let $G_{K,p,S}=Gal(\dbH^{\infty}_{p,S}(K)/K)$
(so $G_{K,p}=G_{K,p,\emptyset}$).

The structure of the group $G_{K,p,S}$ depends dramatically on whether $S$ contains
a prime above $p$ or not. In the sequel we shall only discuss the so called tame case
when $S$ contains NO primes above $p$ -- this is equivalent to saying that $p\nmid N(s)$
for any $s\in S$ where $N:K\to\dbQ$ is the norm function. In this case, without loss
of generality we can actually assume that $N(s)\equiv 1\mod p$ for any $s\in S$
since primes $s$ with $N(s)\not\equiv 0,1\mod p$ cannot ramify in $p$-extensions.

In this setting, Theorem~\ref{thm:shaf} is a special case of the following result
\cite[Theorems~1,6]{Sh}:

\begin{Theorem}
\label{thm:shaf2}
Let $K$ be a number field, $p$ a prime, $S$ a finite set of primes of $K$,
and assume that $N(s)\equiv 1\mod p$ for every $s\in S$. Let $G=G_{K,S,p}$.
Then
\begin{itemize}
\item[(i)] $d(G)\geq  |S|+1 -\nu(K)-\delta(K)$
\item[(ii)] $0\leq r(G)-d(G)\leq \nu(K)-1$
\end{itemize}
where $\nu(K)$ is the number of infinite primes of $K$ and $\delta(K)=1$ or $0$
depending on whether $K$ contains a primitive $p^{\rm th}$ root of unity or not. 
\end{Theorem}
In particular, if we fix $K$ and $p$, the group $G_{K,p,S}$ can be made
Golod-Shafarevich if $|S|$ is chosen to be large enough:

\begin{Corollary}\rm(\cite[Theorem~10.10.1]{NSW})\it 
\label{cor:NSW}
In the above notations assume that 
$$|S|>1+\nu(K)+2\sqrt{\nu(K)}+\delta(K).$$
Then the group $G_{K,p,S}$ is Golod-Shafarevich.
\end{Corollary}

In general, if we fix $K$ and $p$, the structure of the group $G_{K,p,S}$
becomes more transparent as $S$ gets larger. In particular, by choosing
$S$ sufficiently large, one can ensure that certain number-theoretically defined 
group $V_{K,S}$ is trivial, in which case one can write down a (fairly) explicit presentation for $G_{K,p,S}$
by generators and relations (see, e.g., \cite[Chapters~11,12]{Ko}). Further
very interesting results in this direction have been recently obtained by Schmidt
\cite{Sch1, Sch2,Sch3}.

Thus, in attempting to study the structure of the groups $G_{K,p,S}$
one might want to concentrate on the case when $S$ is sufficiently large, 
and in view of Corollary~\ref{cor:NSW} one might hope to make use of Golod-Shafarevich
techniques in this case. One important result of this flavor was obtained
by Hajir~\cite{Haj} who proved that the group $G_{K,p,S}$ has exponential
subgroup growth for sufficiently large $S$. In the next subsection we present a 
group-theoretic version of Hajir's argument in the simplest case $K=\dbQ$.

\subsection{Examples with exponential subgroup growth}

The following presentation for the groups $G_{\dbQ,p,S}$ is a special case 
of a more general theorem of Koch (see \cite[\S~11.4]{Ko} and \cite[\S~6]{Ko2}).

\begin{Theorem}
\label{thm:shafQpres}
Let $p>2$ be a prime and $S=\{q_1, \ldots, q_d\}$ 
a finite set of primes congruent to $1$ mod $p$. Then the group $G=G_{\dbQ,p,S}$
has a presentation
\begin{equation}
\label{eq:shafQ}
\la x_1, \ldots, x_d \mid x_i^{a_i}=x_i^{q_i}\ra
\end{equation}
for some elements $\{a_i\}_{i=1}^d$ in the free pro-$p$ group on $x_1,\ldots, x_d$
(where as usual $g^h=h^{-1}gh$ for group elements $g$ and $h$). 
\end{Theorem}

Note that presentation \eqref{eq:shafQ} can be rewritten
as $\la x_1, \ldots, x_d \mid [x_i,a_i]=x_i^{q_i-1}\ra$ and each $q_i-1$
is divisible by $p$. This implies that all relators in the presentation \eqref{eq:shafQ} of $G$ 
lie in the Frattini subgroup, and so $d(G)=d$. We shall now show that any group with such presentation has an LRG chain 
(and therefore exponential subgroup growth) whenever $d\geq 10$. 

\begin{Proposition} 
\label{LRG1}
Let $G$ be a group given by a presentation of the form
$$\la x_1, \ldots, x_d \mid [x_i,a_i]=x_i^{p\lam_i}\ra$$ where $a_i\in F_{\phat}(x_1,\ldots, x_d)$
and $\lam_i\in\dbZ_p$ (note that $d(G)=d$ and $r(G)\leq d$). If $d\geq 10$, then $G$ has 
an LRG chain.
\end{Proposition}
\begin{proof} Consider the group $Q=G/\la x_1,x_2,a_1,a_2\ra^G$, the quotient
of $G$ by the normal subgroup generated by $x_1,x_2,a_1$ and $a_2$. We claim that
$Q$ is Golod-Shafarevich (hence infinite). Indeed, by construction $d(Q)\geq d(G)-4=d-4\geq 6$.
Also note that $Q$ has a presentation with $d$ generators and $d+2$ relations, 
namely $$Q=\la x_1,\ldots, x_d \mid x_1=x_2=a_1=a_2=1, [x_i,a_i]=x_i^{p\lam_i}\mbox{ for }i\geq 3\ra,$$
so $r(Q)-d(Q)\leq 2$ by Lemma~\ref{prop_basic}(ii). Therefore, $r(Q)-d(Q)^2/4\leq 2+d(Q)-d(Q)^2/4<0$ since $d(Q)\geq 6$.

Now choose any infinite descending chain $\{Q_i\}$ of open subgroups of $Q$,
and let $G_i$ be the full preimage of $Q_i$ under the projection $\pi:G\to Q$.
Let $F=F_{\phat}(x_1,\ldots, x_d)$, let $N=\la \{[x_i,a_i](x_i^{p\lam_i})^{-1}, 1\leq i\leq d\}\ra^F$ (so that
$G=F/N$), and let $F_i$ be the full preimage of $G_i$ in $F$. Note that $G_i=F_i/N$.
Let $n_i=[G:G_i]=[F:F_i]$. Then $d(F_i)=(d-1)n_i+1$ by the Schreier formula, and 
$N$ is generated as a normal subgroup of $F_i$ by the set $R_i$ which consists
of conjugates of elements of $R$ by some transversal of $F_i$ in $F$.
By construction, each $F_i$ contains the elements $x_1,x_2,a_1,a_2$. Hence
the relators $[x_i,a_i](x_i^{p\lam_i})^{-1}$ for $i=1,2$ and all their conjugates
lie in $\Phi(F_i)$, the Frattini subgroup of $F_i$, and therefore do not affect
$d(G_i)$. The number of remaining relations in $R_i$ is $(d-2)n_i$. Hence
$d(G_i)\geq (d-1)n_i+1-(d-2)n_i=n_i+1=[G:G_i]+1$, so $\{G_i\}$ is an infinite chain with linear
rank growth.
\end{proof}

Hajir~\cite{Haj} actually proved something more interesting than
exponential subgroup growth for $G_{K,p,S}$ for sufficiently large $S$ --
he showed that exponential subgroup growth can be achieved
in the group $G_{K,p}=G_{K,p,\emptyset}$ for a suitably chosen number field
$K$ depending on $p$. We now briefly outline how to construct such examples
using the above presentations of the groups $G_{\dbQ,p,S}$.

Let $S$ and $G$ be as in the statement of Theorem~\ref{thm:shafQpres}. 
Let $H$ be any index $p$ subgroup of $G=G_{\dbQ,p,S}$ which does not 
contain any of the generators $x_i$ (such
$H$ exists since all relators in the presentation \eqref{eq:shafQ} lie
in the Frattini subgroup),
and let $K\subset \dbH^{\infty}_{p,S}(\dbQ)$ be the fixed field of $H$. Then it is not hard to show that 
each prime from $S$ ramifies in $K$. Note that $\dbH^{\infty}_{p}(K)$, the maximal
unramified pro-$p$ extension of $K$, is contained in $\dbH^{\infty}_{p,S}(\dbQ)$
by construction, and therefore, the Galois group $Gal(\dbH^{\infty}_{p}(K)/\dbQ)$
is a quotient of $G_{\dbQ,p,S}$.

According to \cite[Theorem~12.1]{Ko}, the assumption that each prime from $S$
ramifies in $K$ ensures that $G'=Gal(\dbH^{\infty}_{p}(K)/\dbQ)$ is isomoprhic to
$G/\la x_1^p,\ldots, x_d^p\ra^G$ where $x_1,\ldots, x_d$ are as in \eqref{eq:shafQ},
so 
$$G'=\la x_1,\ldots, x_d \mid x_i^p=1, [x_i,a_i]=1\mbox{ for }1\leq i\leq d\ra.$$
One can show directly from this presentation that the group $G'$
has an LRG chain whenever $d\geq 12$ and $p\geq 11$ or $d\geq 65$ and $p$ is arbitrary
-- this is achieved by combining the idea
of the proof of Proposition~\ref{LRG1} with the notion of power $p$-deficiency
discussed in the previous section. Since $G_{K,p}=Gal(\dbH^{\infty}_{p}(K)/K)$
is a subgroup of index $p$ in $G'=Gal(\dbH^{\infty}_{p}(K)/\dbQ)$, we conclude
that $G_{K,p}$ also has an LRG chain.

\skv
Finally, we remark that the existence of an LRG chain in the group $G_{K,p}$
has a very natural number-theoretic interpretation: it is equivalent to the existence 
of an infinite ascending chain of finite unramified $p$-extensions $K\subset K_1\subset K_2\subset\ldots$
such that the sequence $\{\rho_p(K_n)\}_{n\geq 1}$ of the $p$-ranks of the ideal class groups
of $K_n$ grows linearly with the degree $[K_n:K]$.

\section{Applications in geometry and topology}
\label{sec:geom}

In this section we will discuss some applications of the Golod-Shafarevich theory 
to the study of  hyperbolic 3-manifolds or rather their 
fundamental groups. By a hyperbolic 3-manifold we shall always mean a finite volume orientable 
hyperbolic 3-manifold without boundary.
The fundamental groups of hyperbolic 3-manifolds are precisely
the torsion-free lattices in $PSL_2(\dbC)\cong SO(3,1)$, with cocompact lattices
corresponding to compact 3-manifolds. Arbitrary lattices in $PSL_2(\dbC)$
(which are always virtually torsion-free) correspond to hyperbolic 3-orbifolds.
 
If $X$ is a compact (orientable) 3-manifold and $\Gamma=\pi_1(X)$
its fundamental group, then by a result of Epstein~\cite{Ep}, $\Gamma$ has
a presentation $( X, R)$ with $|X|\geq |R|$. Thus $def(\Gamma)\geq 0$,
and moreover the same is true if $\Gamma$ is replaced by a finite index
subgroup (since a finite cover of a compact 3-manifold is itself a compact
3-manifold). Note that by Theorem~\ref{GS_strong}(b), a group $\Gamma$ with 
non-negative deficiency is Golod-Shafarevich with respect to a prime $p$
whenever $d(\Gamma_{\phat})\geq 5$.

Lubotzky~\cite{Lu2} proved that if $\Gamma$ is a finitely generated group
which is linear in characteristic different from $2$ or $3$ and not virtually solvable, then
for any prime $p$ the set $\{d(\Delta_{\phat}): \Delta$ is a finite index subgroup of 
$\Gamma\}$ is unbounded. If $X$ is hyperbolic, then $\Gamma=\pi_1(X)$ is linear and not virtually solvable 
(being a lattice in $PSL_2(\dbC)$). Thus, the discussion in the previous paragraph implies the following:

\begin{Proposition}[\cite{Lu1}]
\label{hyp3mGS}
Let $X$ be a hyperbolic $3$-manifold and $\Gamma=\pi_1(X)$. Then for every prime $p$,
$\Gamma$ has a finite index subgroup which is Golod-Shafarevich with respect to $p$.
\end{Proposition}

Proposition~\ref{hyp3mGS} was established in 1983 Lubotzky's paper \cite{Lu1}
as a tool for solving a major open problem, known at the time as Serre's conjecture.
The conjecture (now a theorem) asserts that arithmetic lattices in $SL_2(\dbC)$
do not have the congruence subgroup property. The proof of this conjecture
is a combination of three results:

\begin{itemize}
\item[(a)] Proposition~\ref{hyp3mGS}.
\item[(b)] If $\Gamma$ is an arithmetic group with the congruence subgroup property,
then for any prime $p$, the pro-$p$ completion $\Gamma_{\phat}$ is $p$-adic analytic.
\item[(c)] Golod-Shafarevich pro-$p$ groups are not $p$-adic analytic.
\end{itemize}

Lubotzky established part (c) using Lazard's theorem~\cite{Laz} which asserts that
a pro-$p$ group is $p$-adic analytic if and only if the coefficients of the
Hilbert series $Hilb_{\Fp[[G]],d}(t)$ grow polynomially, where $d$ is the standard
degree function. By Corollary~\ref{GS:divergent},
this cannot happen in a Golod-Shafarevich group. Now one can give simple alternative
proofs of (c) thanks to many new characterizations of $p$-adic analytic obtained after \cite{Lu1}.
For instance, a pro-$p$ group is 
$p$-adic analytic if and only if the set $\{d(U): U\mbox{ is an open subgroup of }G\}$ is bounded
(see, e.g., \cite{LuMn} or \cite[\S~3,7]{DDMS}). 
This also prevents $G$ from being Golod-Shafarevich, e.g., by Theorem~\ref{subgpgrowth_Jaikin}.

Lubotzky's work provided the first (and very non-trivial) geometric application of Golod-Shafarevich
groups and gave hope that even deeper problems about 3-manifolds could be tackled in the same way.
Indeed, suppose one wants to prove that hyperbolic 3-manifold groups always have certain property (P)
and (P) is inherited by finite index subgroups and overgroups. In view of Proposition~\ref{hyp3mGS},
to prove such a result it is sufficient to show that every Golod-Shafarevich group has property (P).

One of the main open problems about 3-manifolds is the virtually positive Betti number (VPBN) conjecture
due to Thurston and Waldhausen:

\begin{Conjecture}[VPBN Conjecture] 
\label{vpbn}
Let $M$ be a hyperbolic 3-manifold. Then $M$ has a finite cover with positive
first Betti number. Equivalently, $\pi_1(M)$ does not have property (FAb), that is,
$\pi_1(M)$ has a finite index subgroup with infinite abelianization.
\end{Conjecture}

This conjecture clearly cannot be settled just using Proposition~\ref{hyp3mGS} since there exist
torsion Golod-Shafarevich groups. However, Golod-Shafarevich theory seemed to be a promising tool
for attacking a weaker conjecture of Lubotzky and Sarnak.

\begin{Conjecture}[Lubotzky-Sarnak] 
\label{LSC}
Let $M$ be a hyperbolic 3-manifold. Then $\pi_1(M)$ does not have property $(\tau)$.
\end{Conjecture}

For the definition and basic properties of Kazhdan's property $(T)$ and its weaker (finitary)
version property $(\tau)$ we refer the reader to the books \cite{BHV} and \cite{LuZ}.

A finitely generated group with property $(\tau)$ must have (FAb), so Conjecture~\ref{vpbn}
would imply Lubotzky-Sarnak Conjecture. The latter was originally posed not because of its
intrinsic value, but with the hope that it may be easier to settle than VPBN conjecture,
while its solution may shed some light on VPBN conjecture.

It seemed quite feasible that Lubotzky-Sarnak conjecture might be solved using Golod-Shafarevich approach,
that is, it may be true that Golod-Shafarevich groups never have property $(\tau)$. The latter,
however, turned out to be false, as explicit examples of Golod-Shafarevich groups with property~$(\tau)$
(actually, with property~$(T)$) were constructed in \cite{Er}. 

These examples still leave a possibility that Lubotzky-Sarnak conjecture (or even VPBN conjecture)
could be solved by group-theoretic methods since it is easy to identify group-theoretic properties
which hold for hyperbolic 3-manifold groups and which clearly fail for all known examples of
Golod-Shafarevich groups with property $(\tau)$. Unfortunately, at present there seems to be no
group-theoretic conjecture which would imply Lubotzky-Sarnak conjecture and which could be attacked
with currently known methods.

Nevertheless, Golod-Shafarevich techniques did yield new important results about
3-manifold groups. Perhaps the most interesting of those are two results of Lackenby
dealing with subgroup growth.

\subsection{Subgroup growth of 3-manifold groups}

In \cite{La1} and \cite{La2}, Lackenby obtained strong lower bounds on the subgroup growth
of hyperbolic 3-manifold groups. The first result asserts that for any hyperbolic 3-manifold
group, the subgroup growth function is bounded below by an almost exponential function on an 
infinite subset of $\dbN$.

\begin{Theorem}\rm (\cite{La2})\it 
\label{thm:hyp_esg}
Let $\Gamma$ be the fundamental group of a hyperbolic 3-manifold,
and let $a_n(\Gamma)$ be the number of subgroups of index $n$ in $\Gamma$.
Then $a_n(\Gamma)\geq 2^{n/(\sqrt{\log n}\,\cdot\,\log(\log n))}$ for infinitely many $n$.  
\end{Theorem}

This result follows by a direct (though not completely straightforward) computation
from Corollary~\ref{Cor_Lacksg} and Lemma~\ref{sg_basicbound} for $p=2$ (see \cite[\S~6,Claim~2]{La2} for details)
applied to the pro-$p$ completion of $\Gamma$. (We note that by \cite{Lu2}, the assumption 
$d_p(\Gamma)=d(\Gamma_{\phat})\geq 4$ can always be achieved replacing $\Gamma$ by a finite index subgroup).
The proof of Corollary~\ref{Cor_Lacksg} (which is an algebraic
result) in \cite{La2} uses topological techniques, but the alternative proof given in this paper
is purely algebraic and based on the finitary Golod-Shafarevich inequality.

The second result of Lackenby asserts that for a large class of hyperbolic 3-manifolds
the subgroup growth is at least exponential:

\begin{Theorem}[\cite{La1}] 
\label{thm:orb_esg}
Let $M$ be a hyperbolic 3-manifold which is commensurable
with an orbifold $O$ with non-empty singular locus. Let $p$ be any prime such that $\pi_1(O)$
has an element of order $p$. Then $\pi_1(M)$ has an LRG $p$-chain and hence
has at least exponential subgroup growth. 
\end{Theorem}

Unlike Theorem~\ref{thm:hyp_esg}, it does not seem possible to give an entirely algebraic
proof of Theorem~\ref{thm:orb_esg} (although a substantial part of the argument
in \cite{La1} is group-theoretic). For this reason we do not discuss the proof of
this result in this paper and refer the reader to a very clear exposition in \cite{La1}. 
However, we do remark that there are many similarities between the proof of
Theorem~\ref{thm:orb_esg} and that of Proposition~\ref{LRG1} (in fact, the latter was inspired 
by the former).
\skv
Another interesting application of Golod-Shafarevich inequality to 3-manifold groups
(specifically, to the structure of their rational lower central series)
was obtained by Freedman, Hain and Teichner~\cite{FHT}.

\section{Golod-Shafarevich groups and Kazhdan's property $(T)$}
\label{sec:T}

\subsection{Golod-Shafarevich groups with property $(T)$}
In the previous section we discussed why the question of the existence of Golod-Shafarevich
groups with property $(\tau)$ was important in topology, or rather, why the lack of such groups
would have been very useful. This question, however, is quite natural from a purely
group-theoretic point of view as well, and when the question was open, one could present
natural heuristic arguments for both non-existence and existence of such groups. On the one hand,
as we already saw, every Golod-Shafarevich group has a lot of quotients (including finite quotients), 
seemingly too many for such a group to have property $(\tau)$. On the other hand, Golod-Shafarevich
groups behave similarly to hyperbolic groups in many ways, and there exist hyperbolic groups
with property $(T)$ (hence also property $(\tau)$). A posteriori, it seems that the latter heuristics was 
the ``right one'', at least it predicted the right answer, although the actual examples
of Golod-Shafarevich groups with property $(T)$ are completely different from the examples
of hyperbolic groups with $(T)$.

The first examples of Golod-Shafarevich groups with property $(T)$ were constructed
in \cite{Er} as positive parts of certain Kac-Moody groups over finite fields. Property~$(T)$ 
for such groups was established earlier by Dymara and Januszkiewicz~\cite{DJ}, while the
Golod-Shafarevich condition was verified using certain optimization of the Tits presentation
of such groups. We shall not discuss this construction since much simpler to describe examples 
of Golod-Shafarevich groups with $(T)$ were given in \cite{EJ1}.

\begin{Theorem}\cite{EJ1} 
\label{KazhT}
Let $p$ be a prime and $d\geq 2$ an integer, and consider
the group $$G_{p,d}=\la x_1,\ldots,x_d \mid x_i^p=1, [x_i,x_j,x_j]=1\mbox{ for }1\leq i\neq j\leq 9\ra.$$
Then
\begin{itemize}
\item[(i)] The group $G_{p,d}$ is Golod-Shafarevich with respect to $p$ whenever $p\geq 3$ and $d\geq 9$ or $p=2$ and $d\geq 12$.
\item[(ii)] The group $G_{p,d}$ has property $(T)$ whenever $p>(d-1)^2$.
\end{itemize}
In particular, for any $p\geq 67$, there exists a Golod-Shafarevich group (with respect to $p$)
with property $(T)$.
\end{Theorem}

Part (i) is established by direct verification: indeed, if $( X, R)$ is the presentation of
$G_{p,d}$ given above, then $1-H_X(\tau)+H_R(\tau)=1-d\tau+d(d-1)\tau^3+d\tau^p$, which is negative
for $\tau=2/d$ under the required conditions on $p$ and $d$.

Part (ii) is proved using a general criterion for property~(T) from \cite{EJ1} (see Theorem~\ref{DJcrit} below).

\begin{Definition}\rm Let $H$ and $K$ be subgroups of the same group. The {\it orthogonality constant}
$orth(H,K)$ is defined to be the smallest $\eps\geq 0$ with the following property: if $V$ is a unitary
representation of the group $\la H,K\ra$ without nonzero invariant vectors, $v\in V$ is $H$-invariant
and $w\in V$ is $K$-invariant, then $|\la v,w\ra|\leq\eps \|v\|\|w\|$. 
\end{Definition}

\begin{Theorem}\rm (\cite[Theorem~1.2]{EJ1})\it 
\label{DJcrit}
Suppose that a group
$G$ is generated by $n$ finite subgroups $H_1,\ldots, H_n$, and for each $1\leq i\neq j\leq n$
we have $orth(H_i,H_j)<\frac{1}{n-1}$. Then $G$ has property $(T)$. 
\end{Theorem}

The wonderful thing about this criterion is that the orthogonality constant $orth(H,K)$ is completely
determined by the representation theory of the subgroup $\la H,K \ra$; in fact, it suffices
to consider only irreducible representations. If $G=G_{p,d}$ is a group from Theorem~\ref{KazhT},
we let $H_i=\la x_i\ra$ for $1\leq i\leq d$. For any $i\neq j$, the group $\la H_i, H_j\ra$ is isomorphic 
to the Heisenberg group over $\Fp$ which has very simple representation theory, and one easily shows that
$orth(H_i,H_j)=1/\sqrt{p}$. Therefore, by Theorem~\ref{DJcrit}, $G_{p,d}$ has $(T)$ whenever 
$p>(d-1)^2$.

\begin{Remark} The ``Kac-Moody examples'' with property $(T)$ from \cite{Er}
are quotients of the groups $G_{p,d}$. These groups are  Golod-Shafarevich 
under stronger assumptions on $p$ and $d$ than the ones given in Theorem~\ref{KazhT},
and it takes some work to verify the Golod-Shafarevich condition for these groups. 
\end{Remark}

\subsection{Applications}

In terms of potential applications to 3-manifolds, the existence of Golod-Shafarevich groups 
with property $(T)$ was a ``negative result''. However, it turned out to be a very
useful tool for constructing examples of groups with exotic finiteness properties. 
For instance, it immediately implies the existence of residually finite torsion non-amenable groups.

\begin{Theorem}\cite{Er} 
\label{thm:nonamenrf}
There exist residually finite torsion non-amenable groups.
\end{Theorem}
\begin{proof} Let $G$ be a Golod-Shafarevich group with property $(T)$. By
Theorem~\ref{GS_Wilson}, $G$ has a torsion quotient $G'$ which is also Golod-Shafarevich.
Hence the image of $G'$ in its pro-$p$ completion, call it $G''$, is infinite. Then $G''$ 
is a torsion residually finite group which has property $(T)$ (being a quotient of $G$).
Since an infinite group with $(T)$ is non-amenable, we are done.  
\end{proof}
\begin{Remark} Recall that another construction of residually finite torsion non-amenable groups
due to Schlage-Puchta and Osin was described in \S~\ref{sec:powerdef}.
\end{Remark}

Golod-Shafarevich groups with $(T)$ also provide a very simple approach to
constructing infinite residually finite groups which have $(T)$ and some additional property $(P)$
via the following observation.

\begin{Observation} 
\label{obs:PT}
Let $(P)$ be a group-theoretric property such that every
Golod-Shafarevich group has an infinite residually finite quotient with $(P)$.
Then there exists an infinite residually finite group which has $(P)$ and $(T)$.
\end{Observation}

Recall that several properties $(P)$ satisfying the hypothesis of Observation~\ref{obs:PT}
were stated in \S~\ref{sec:quot}. Applying Observation~\ref{obs:PT} to those properties,
we obtain the following results:

\begin{Proposition}\rm (\cite[Theorem~1.3]{EJ3}) \it
\label{LERF}
There exists an infinite LERF group with $(T)$.
\end{Proposition}

\begin{Proposition}\cite{Er2} 
\label{FCrad}
There exists a residually finite group with $(T)$
whose FC-radical (the set of elements with finite conjugacy class) is not virtually abelian.
\end{Proposition}

Proposition~\ref{LERF} answers a question of Long and Reid~\cite{LR} which arose 
in connection with the study of property LERF for $3$-manifold groups while 
Proposition~\ref{FCrad}  settled a question of Popa and Vaes~\cite{PV} coming from 
measurable group theory.

\subsection{Kazhdan quotients of Golod-Shafarevich groups}

In this subsection we discuss the proof of the following theorem:

\begin{Theorem}\rm (\cite[Theorems~1.1, 4.6]{EJ2})\it 
\label{Kazhquot}
Every generalized Golod-Shafarevich group
has an infinite quotient with Kazhdan's property $(T)$.
\end{Theorem}

While the fact that Golod-Shafarevich groups with $(T)$ exist was somewhat surprising,
once it was established, it was natural to expect that the assertion of Theorem~\ref{Kazhquot}
is true, and this was explicitly conjectured by Lubotzky. The conjecture was partially motivated 
by the theory of hyperbolic groups where the analogous result was known to be true: every (non-elementary) hyperbolic group has an infinite quotient with property $(T)$, which follows directly
from two deep theorems:
\begin{itemize}
\item[(a)] There exists a hyperbolic group with property $(T)$.
\item[(b)] Any two hyperbolic groups have a common infinite quotient.
\end{itemize}
In fact, this analogy suggests a naive approach to Theorem~\ref{Kazhquot}:  Theorem~\ref{Kazhquot} would follow from Theorem~\ref{KazhT}, at least for $p\geq 67$, if one could show that any two GGS groups (with respect to the same $p$) have a common infinite quotient. The latter is of course too
much to expect, even if we consider GS groups instead of GGS groups (we do not know explicit counterexamples 
at this point, but there is little doubt that such counterexamples exist). Nevertheless, one could still try 
to show that for any GGS group $G$ there is another GGS group $H$ with $(T)$ such that $G$ and $H$ have a common
infinite quotient -- if true, this would still imply Theorem~\ref{Kazhquot}. 

In order to implement this approach, one needs to possess a large supply of GGS
group with $(T)$. The class of groups described in Theorem~\ref{KazhT} is way too small 
for this to work, but using essentially the same method one can construct more groups with
this property.

\begin{Theorem}\rm (see \cite[Theorem~4.2]{EJ2})\it
\label{KMS}
Let $p$ be a prime, $d>0$ an integer and $n_1,\ldots, n_d$ positive integers.
Consider the group $G$ given by the presentation $(X_{KMS}, R_{KMS})$
where $X_{KMS}=\{x_{i,k}: 1\leq i\leq d, 1\leq k\leq n_d\}$ and
$R_{KMS}=\{[x_{i,k},x_{j,l},x_{j,m}] \mbox{ for } i\neq j\}\cup\{[x_{i,k},x_{i,l}]\}\cup\{x_{i,k}^p\}.$  
If $d\geq 9$ and $p>(d-1)^2$, then $G$ is GGS and has property $(T)$. 
\end{Theorem}
\begin{Remark} The groups described in this theorem are called Kac-Moody-Steinberg
groups in \cite{EJ1} since they map onto suitable Kac-Moody groups over $\Fp$
as well as certain Steinberg groups. This explains the notations $X_{KMS}$ and $R_{KMS}$
for the sets of generators and relators.
\end{Remark}

This class of groups is still insufficient to make the naive approach work, but a more
convoluted scheme based on the same idea does work. We shall now outline the argument.

First we reduce the problem to the following:
\begin{Theorem}\rm 
\label{Kazhquot_fi}
Every generalized Golod-Shafarevich group
has a finite index subgroup which has an infinite quotient with Kazhdan's property $(T)$.
\end{Theorem}

The reduction is possible due to the following general statement:
\begin{Proposition} 
\label{JZ}
Let $(P)$ be a group-theoretic property, which is preserved by
quotients, finite direct products, finite index subgroups and finite index overgroups.
Let $G$ be a group, and suppose that some finite index subgroup of $G$ has an infinite quotient
with $(P)$. Then $G$ itself has an infinite quotient with $(P)$.
\end{Proposition}

Proposition~\ref{JZ} was proved by Jaikin-Zapirain in the case $(P)=(T)$ 
(see~\cite[Prop.~4.5]{EJ2}), but as observed in \cite[Prop.~3.5]{BT}, the same argument
applies to any property $(P)$ as above.

\begin{proof}[Proof of Theorem~\ref{Kazhquot_fi}(sketch)]
We shall restrict ourselves to the case $p\geq 67$; the proof in the case $p<67$
is similar, but more technical. 
Let $\Gamma$ be a generalized Golod-Shafarevich abstract group; without loss
of generality we can assume that $\Gamma$ is residually-$p$.
Let $G=\Gamma_{\phat}$ be the pro-$p$ completion of $\Gamma$ and $\overline W$ a valuation on $G$ such 
that $def_{\overline W}(G)>0$. The proof of Theorem~\ref{Kazhquot_fi} consists of four main steps.

{\it Step 1:}  Given $M\in\dbR$, find an open subgroup $H$ of $G$ such that
$def_{\overline W}(H)>M$. Then we can find a weighted presentation
$(X,R,W)$ of $H$ (where $W$ induces $\overline W$) such that
$def_{W}(X,R)>M$. 

{\it Step 2:}  Given real numbers $w>1$ and $\eps>0$, show that
there is a real number $f(w,\eps)$ such that if in Step~1 we take $M>f(w,\eps)$,
then there is another weighted presentation $(X',R',W')$ of $H$, with $X'\subset F_{\phat}(X)$,
such that $W'(X')=w$, $W'(R')<\eps$, $W'(x)<\eps$ for all $x\in X'$
and $W'(h)\leq W(h)$ for all $h\in F_{\phat}(X')$.

{\it Step 3:} Show that if $w$ and $\eps$ in Step~2 are suitably chosen,
then there is a group $\Lambda$ with $(T)$ from the family described in Theorem~\ref{KMS}
such that 
\begin{itemize}
\item[(i)] the canonical set of generators $X_{KMS}=\{x_{i,j}\}$ of $\Lambda$ 
has the same cardinality as $X'$ from Step~2, so there is a bijection (thought of as identification)
$\sigma: X_{KMS}\to X'$.
\item[(ii)] If $R_{KMS}$ is the canonical set of relators of $\Lambda$,
we can choose $\sigma: X_{KMS}\to X'$ in such a way that $def_{W'}(X',R'\cup R_{KMS})>0$,
so the pro-$p$ group $Q=\la X'\mid R'\cup R_{KMS}\ra$ is GGS.
\end{itemize}
{\it Step 4:} Let $\Delta$ be the image of $\Gamma\cap H$ in $Q$, and let $\Delta'$ be the
subgroup of $Q$ abstractly generated by $X'$. Note that $\Delta'$ is a quotient of $\Lambda$.
By Lemma~\ref{cuttails} (Tails Lemma), we can find a GGS 
quotient $Q'$ of $Q$ in which the images of $\Delta$ and $\Delta'$ coincide; call their common 
image $\Omega$.
\vskip .1cm
We claim that $\Omega$ satisfies the conclusion of Theorem~\ref{Kazhquot_fi}. Indeed,
by construction, $\Omega$ is a quotient of $\Gamma\cap H$ (which is a finite index subgroup of $\Gamma$)
and has $(T)$ being a quotient of $\Lambda$. Finally, $\Omega$ is infinite being a dense
subgroup of the GGS pro-$p$ group $Q'$.

We now comment briefly on the proof of each step. Step~4 has already been fully explained.
Recall that $def_{\overline W}(H)\geq def_{\overline W}(G)\cdot [G:H]_{\overline W}$
for any open subgroup $H$ by Theorem~\ref{thm:weightedSchreier}(b)
and the ${\overline W}$-index $[G:H]_{\overline W}$ can be made arbitrarily large
by Corollary~\ref{Worder}. This justifies Step~1.

The key tool in Step~2 is the notion of contraction of weight functions.

\begin{Definition}\rm Let $F$ be a free pro-$p$ group, $W$ a weight function on $F$
and $c\geq 1$ a real number. Choose a $W$-free generating set $X$ of $F$, and let
$W'$ be the unique weight function on $F$ with respect to $X$ such that $W'(x)=W(x)/c$
for all $x\in X$. We will say that the function $W'$ is obtained from $W$ by the
{$c$-contraction}. (It is easy to see that $W'$ does not depend on the choice of $X$).
\end{Definition}

In order to understand better what a contraction does we go back to
weight functions on power series algebras. By definition, the initial weight
function $W$ is given by $W(f)=w(f-1)$ where $w$ is a weight function on 
$\Fp[[F]]$ with respect to $U=\{x-1: x\in X\}$. Then the contracted weight function
$W'$ can be defined by $W'(f)=w'(f-1)$ where $w'$ is the unique weight function
on $\Fp[[F]]$ with respect to $U$ such that $w'(u)=w(u)/c$ for all $u\in U$.
Recall that $\Fp[[F]]\cong \Fp\lla U\rra$. It is clear that for any
$a\in \Fp[[F]]$ such that the degree of $a$ as a power series in $U$ is at least $k$
we have $w'(a)\leq w(a)/c^k$. In particular this implies that

\begin{itemize}
\item[(i)] $W'(f)\leq W(f)/c$ for all $f\in F$;
\item[(ii)] $W'(f)\leq W(f)/c^2$ for all $f\in \Phi(F)=[F,F]F^p$.
\end{itemize}

Let us now go back to the setting of Step~2. Assume first that all elements
of $R$ lie in $\Phi(F)$. If we obtain $W'$ from $W$ by the $c$-contraction
for $c=W(X)/w$, then $W'(X)=w$ and $W'(R)\leq W(R)/c^2$ by (ii).
Since $W(R)<W(X)$ and $W(X)>M$, we get $W'(R)\leq W(R)w^2/W(X)^2< w^2/M$
and $W'(x)<w/M$ for all $x\in X$.
Thus in this case we can simply set $X'=X$, $R'=R$ and $f(w,\eps)=\max\{w^2/\eps, w/\eps\}$.

In general, the situation is more complex. Note that starting with the presentation $(X,R)$, 
we can eliminate some of the relators together with the corresponding generators
(using the procedure described in Lemma~\ref{prop_basic}(ii)), so that in the new presentation
all relators lie in the Frattini subgroup; unfortunately, during this
operation the weighted deficiency may increase. In order to resolve this problem,
one needs to apply a contraction, followed by elimination of some of the relators,
followed for the second contraction. For the details we refer the reader
to \cite[Theorem~3.15]{EJ2}.

Finally, we turn to Step~3. Here the precise form of relators in $R_{KMS}$ plays an important role. 
Let $X_i=\{x_{i,j}\}_{j=1}^{n_i}$ for $1\leq i\leq 9$, so that $X_{KMS}=\sqcup X_i$.
The key property is that the presentation $( X_{KMS}, R_{KMS})$ is very symmetric, and therefore 
$def_V(X_{KMS}, R_{KMS})>0$ for many different weight functions $V$. A direct computation
(see the proof of Theorem~4.3 in \cite{EJ2}) shows that
$def_V(X_{KMS}, R_{KMS})>1/50$ whenever $V(X_{KMS})=\sum_{i=1}^9 V(X_i)=3/2$ and all subsets 
$X_1,\ldots, X_9$ have approximately equal $V$-weights; more precisely, it is enough to assume that 
$|V(X_i)-1/6|<1/100$ for $1\leq i\leq 9$. 

Now take $w=3/2$ and $\eps=1/100$ in Step~2. Then we can divide
the generators from $X'$ into $9$ subsets such that the total $W'$-weight in each subset
differs from $1/6=(3/2)/9$ by less than $1/100$. Letting $X_i$ be the
$i^{\rm th}$ subset, we obtain an identification $\sigma$ between $X'$ and $X_{KMS}$.
Then $def_{W'}(X',R'\cup R_{KMS})\geq def_{W'}(X_{KMS},R_{KMS})-W'(R')>1/50-\eps>0$,
so conditions (i) and (ii) from Step~3 are satisfied.
\end{proof}

\section{Residually finite monsters}

The following famous theorem was proved by Ol'shanskii in 1980:

\begin{Theorem}[Ol'shanskii, \cite{Ol}]
\label{thm:Olsh}
For every sufficiently large
prime $p$ there exists an  infinite group $\Gamma$
in which every proper subgroup is cyclic of order $p$.
\end{Theorem}

Groups satisfying the above condition are called Tarski monsters, named after Alfred Tarski
who first posed the question of their existence. Tarski monsters satisfy
a number of extremely unusual properties. However, they are not residually finite
(as they do not have any proper subgroups of finite index), and
it is a common phenomenon in combinatorial group theory that residually finite finitely generated groups are 
much better behaved than arbitrary finitely generated groups. Thus
it is interesting to find out how close a residually finite group can be to a Tarski monster.
In particular, the following natural question was asked by several different people.

\begin{Problem}
\label{prob:Tar_rf}
Let $p$ be a prime. Does there exist an infinite finitely generated residually finite
$p$-torsion group in which every subgroup is either finite or of finite index? 
\end{Problem}

This problem remains completely open, except for $p=2$ when non-existence of such groups
was known since 1970s and in fact can be proved by a very elementary argument 
(see \cite[\S~8.1]{EJ3} and references therein).
However, in \cite{EJ3}, Golod-Shafarevich techniques were used to prove the existence
of residually finite groups which satisfy the condition in 
the above problem for all finitely generated subgroups:

\begin{Theorem}
\label{thm:Tar_rf}
For every prime $p$ there exists an infinite finitely generated residually finite $p$-torsion group
in which every {\bf finitely generated} subgroup is either finite or of finite index. Moreover,
every (abstract) generalized Golod-Shafarevich group (with respect to $p$)
has a quotient with this property.  
\end{Theorem}
\subsection{Sketch of the proof of Theorem~\ref{thm:Tar_rf}}
The basic idea behind constructing such groups is very simple. Let $\Gamma$ be a generalized Golod-Shafarevich group.
Without loss of generality, we can assume right away that $\Gamma$ is $p$-torsion and residually-$p$,
so we can identify $\Gamma$ with a subgroup of $G=\Gamma_{\phat}$.
There are only countably many finitely generated subgroups of $\Gamma$, so we can enumerate them:
$\Lambda_1,\Lambda_2,\ldots$. At the first step we construct an infinite quotient $G_1$ of $G$ such that
if $\pi_1:G\to G_1$ is the natural projection, then $\pi_1(\Lambda_1)$ is either finite or has 
finite index in $\pi_1(\Gamma)$; note that the latter condition will be preserved
if we replace $G_1$ by another quotient. Next we construct an infinite quotient $\Gamma_2$ of $\Gamma_1$
such that if $\pi_2:G\to G_2$ is the natural projection, then $\pi_2(\Lambda_2)$ is either finite or of 
finite index in $\pi_2(\Gamma)$. We proceed in this way indefinitely.
Let $G_{\infty}=\varinjlim G_i$; in other words, if $G_i=G/N_i$ (so that the chain 
$\{N_i\}$ is ascending), we let $N_{\infty}=\overline{\cup N_i}$, the closure of $\cup N_i$,
and $G_{\infty}=G/N_{\infty}$. Since each $G_i$ is infinite, $G_{\infty}$ must also be infinite
(otherwise $N_{\infty}$ is of finite index in $G$, hence it is a finitely generated pro-$p$ group, 
which easily implies that $N_{\infty}=N_i$ for some $i$). Let $\Gamma_{\infty}$ be the image of $\Gamma$
in $G_{\infty}$. By construction, $\Gamma_{\infty}$ is $p$-torsion, and each of its finitely generated subgroups
is  finite or of finite index. Finally, $\Gamma_{\infty}$ is residually finite being a subgroup of $G_{\infty}$
and infinite being dense in $G_{\infty}$, so it satisfies the required properties. 

So, we just have to make sure that a sequence $\{G_i\}$ as above can indeed be constructed.
Things would have been really nice if at each step we could make $G_i$ a GGS group.
We do not know how to achieve this, and in order to resolve the problem we have to extend
the class of GGS groups even further.

\begin{Definition}\rm A pro-$p$ group $G$ will be called a {\it pseudo-GGS group} if
there exist an open normal subgroup $H$ of $G$ and a finite valuation $W$ on $H$
such that
\begin{itemize}
\item[(i)] $def_W(H)>0$ (so, in particular, $H$ is a GGS group).
\item[(ii)] The function $W$ is $G$-invariant, that is, $W(h^g)=W(h)$ for all $h\in H$ and $g\in G$.
\end{itemize}
\end{Definition}
\begin{Remark} Pseudo-GGS groups are called {\it groups of positive virtual weighted deficiency in \cite{EJ3}}.
\end{Remark}

We will need a simple lemma which generalizes  Lemma~\ref{GGS_quotbase}:

\begin{Lemma} 
\label{pseudo1}
Let $G$ be a pseudo-GGS group, and let $H$ and $W$ satisfy conditions (i) and (ii) above.
Let $S$ be a subset of $H$ and let $G'=G/\la S\ra^G$.
If $W(S)<def_W(H)/[G:H]$, then $G'$ is also a pseudo-GGS group.
\end{Lemma}
\begin{proof} Let $T$ be a transversal of $H$ in $G$ and let $H'$ be the image of $H$ in $G'$. Then $H'\cong H/\la S'\ra^{H}$ where $S'=\{s^t: s\in S, t\in T\}$. Since $W$ is $G$-invariant,
$W(S')\leq W(S)|T|=W(S)[G:H]$, so $def_W(H')>0$ by Lemma~\ref{GGS_quotbase}.
Thus, $G'$ is also a pseudo-GGS group with $H'$ satisfying 
conditions (i) and (ii) above.
\end{proof}

The following result is a key step in the proof of Theorem~\ref{thm:Tar_rf}:

\begin{Theorem}
\label{thm:TarKey}
Let $G$ be a pseudo-GGS pro-$p$ group, $\Gamma$ a finitely generated dense subgroup of $G$
and $\Lambda$ a finitely generated subgroup of $\Gamma$. Then there exists
an epimorphism $\pi:G\to Q$ such that
\begin{itemize}
\item[(i)] $Q$ is a pseudo-GGS pro-$p$ group;
\item[(ii)] $\pi(\Lambda)$ is either finite or has finite index in $\pi(\Gamma)$.
\end{itemize}
\end{Theorem}

Theorem~\ref{thm:TarKey} ensures that we can make each step in the above iterated algorithm,
and therefore we have now reduced Theorem~\ref{thm:Tar_rf} to Theorem~\ref{thm:TarKey}.

\begin{proof}[Sketch of the proof of Theorem~\ref{thm:TarKey}]
Let $H$ be an open normal subgroup of $G$ and $W$ a valuation on $H$ from the definition
of a pseudo-GGS group. By the Tails Lemma, we can assume that $\Gamma\cap H$ is abstractly generated by $X$.
Also, replacing $\Lambda$ by its finite index subgroup, we can assume that $\Lambda\subset H$.

Let $L$ be the closure of $\Lambda$. Which of the two alternatives in the conclusion of 
Theorem~\ref{thm:TarKey} will occur depends on whether the $W$-index $[H:L]_W$ is infinite or
finite.

{\it Case 1:} $[H:L]_W<\infty$. In this case, by multiplicativity of $W$-index (Proposition~\ref{index_basicrpop})
and Continuity Lemma (Proposition~\ref{lemma:continuity}), 
for any given $\eps>0$ we can find an open subgroup $U$ of $H$ containing $L$ such that
$[U:L]_W<1+\eps$. This easily implies that there exists
a subset $X_{\eps}$ of $U$ such that $W(X_{\eps})<\eps$ and $U$ is generated by $L$ and $X_{\eps}$.
The latter condition implies that if we let $Q=G/\la X_{\eps}\ra^G$ and let $\pi:G\to Q$ be the natural projection, 
then $\pi(L)=\pi(U)$, so $\pi(L)$ must be of finite index in $Q$. On the other hand,
by Lemma~\ref{pseudo1}, if we take $\eps<def_W(H)/[G:H]$, then $Q=\pi(G)$ is a pseudo-GGS group, as desired.

Note that if $G$ was a GGS group, then $Q=\pi(G)$ would also be a GGS group, so if Case~1 always
occurred, we would not need to consider pseudo-GGS groups at all. It is Case~2 where such generalization
is needed.

{\it Case 2:} $[H:L]_W=\infty$. In this case we start with an important subcase:

\vskip .12cm
\noindent
{\it Subcase:} $rk_W(L)<1$. In this subcase we construct the desired quotient $Q$
exactly as in the proof of Theorem~\ref{GS_second}.
Note that the assumption $[H:L]_W=\infty$ is not explicitly used in the proof; however,
it is already implied by the assumption $rk_W(L)<1$.

If $rk_W(L)\geq 1$, the first thing we can try is to replace $W$ by the valuation $W'$ obtained
from $W$ by $c$-contraction for some $c>1$ (recall that $c$-contractions were defined in Step~2
of the proof of Theorem~\ref{Kazhquot_fi}). More precisely, we choose a weight function $\widetilde W$
which induces $W$, let $\widetilde W'$ be the $c$-contraction of $\widetilde W$ and then induce
the valuation $W'$ from $\widetilde W'$. One can show that if $\widetilde W$ is suitably chosen, 
the valuation $W'$ will still be $G$-invariant (see \cite[Prop. 4.13]{EJ3}). If we take $c>rk_W(L)$, then 
clearly $rk_{W'}(L)<1$; the problem is that the deficiency $def_{W'}(H)$ may become negative; more precisely, we can only
guarantee that $def_{W'}(H)>0$ if $def_W(H)>rk_W(L)$.

To overcome this problem we proceed as follows. Using the assumption $[H:L]_W=\infty$, 
it is not hard to show that for any descending chain $\{U_i\}$ of open subgroups of $H$ with $\cap U_i=\{1\}$,
the quantity $\frac{def_W(U_i)}{rk_W(L\cap U_i)}$ goes to infinity. This follows from Theorem~\ref{thm:weightedSchreier},
Continuity Lemma and multiplicativity of $W$-index. In particular, we can find 
$U\subseteq H$ which is open and normal in $G$ for which $rk_W(L\cap U)<def_W(U)$. Thus, if we
let $W'$ be the valuation on $U$ (not on $H$) obtained from $W$ by the $c$-contraction, where
$rk_W(L\cap U)<c<def_W(U)$, then $rk_{W'}(L\cap U)<1$ and $def_{W'}(U)>0$. Now we can finish
the proof as in the above subcase with $W$ replaced by $W'$ and $H$ replaced by $U$.
\end{proof}

\section{Open questions}
\label{sec:questions}

In this section we pose several open problems about Golod-Shafarevich
groups. All these questions make sense for generalized Golod-Shafarevich
groups as well, but with the exception of Problem~\ref{Problem4},
it does not seem that answering them  for GGS groups would be easier or 
harder or more interesting than for GS groups. For each problem we provide 
brief motivation and discuss related works and conjectures. Our list has some overlap with the
list of problems in a paper of Button~\cite{Bu}.

\begin{Problem1}
\label{Problem1}
Let $G$ be a finitely presented  Golod-Shafarevich abstract group. Does $G$
contain a non-abelian free subgroup?  
\end{Problem1}

Recall that Golod-Shafarevich pro-$p$ groups contain non-abelian free
pro-$p$ groups, even if not finitely presented. In the abstract case
there exist Golod-Shafarevich torsion groups, so an additional assumption
about the group is needed to ensure the existence of a non-abelian free subgroup.
We conjecture that the  answer to Problem~\ref{Problem1} is positive, although
we are unaware of any promising approach to it at the moment.

\begin{Problem1}
\label{Problem2}
Let $G$ be a Golod-Shafarevich abstract group with a balanced presentation
(a presentation with the same number of generators and relators).
Is $G$ necessarily large?  
\end{Problem1}

The main motivation for this problem comes from $3$-manifold topology.
Lackenby posed a stronger form of the virtual positive Betti number
conjecture asserting that if $G$ is the fundamental group
of a hyperbolic $3$-manifold, then $G$ must be large.
As explained in \S~\ref{sec:geom}, such $G$ must have a finite index subgroup 
which is Golod-Shafarevich and has a balanced presentation, so a positive answer
to Problem~\ref{Problem2} would settle Lackenby's conjecture.
In fact, to settle the latter it is enough to answer Problem~2 in the positive 
under the stronger assumption that every finite index subgroup of $G$ has
a balanced presentation. Unfortunately, even in this form the problem remains
wide open and there are no strong indications that the answer should
be positive.

\begin{Problem1}
\label{Problem3}
Let $G$ be a Golod-Shafarevich abstract group with a balanced presentation.
Is it true that $G$ does not have (FAb)?  
\end{Problem1}
Recall that $G$ is said to have (FAb) if every finite index subgroup of $G$
has finite abelianization, so a positive answer to Problem~\ref{Problem2} 
would, of course, imply the same for Problem~\ref{Problem3}. 
Settling Problem~\ref{Problem3} in the affirmative would still be
an amazing result -- even under the extra hypothesis that every finite index subgroup of $G$ has balanced presentation, 
it would imply the virtual positive Betti number conjecture. 

We remark that the analogue
of Problem~\ref{Problem3} for pro-$p$ groups has negative answer --
as explained in \S~\ref{sec:nt}, the Galois group $G_{\dbQ,p,S}$
has a balanced (pro-$p$) presentation and is Golod-Shafarevich, provided
$|S|\geq 5$ and all primes in $S$ are congruent to $1$ mod $p$,
but also has (FAb) by class field theory. Note though that finite
index subgroups of the groups $G_{\dbQ,p,S}$ do not necessarily have
balanced presentations.

Finally, note that Problem~\ref{Problem3} (and hence also Problem~\ref{Problem2})
would have negative answer if we only assumed that $G$ is finitely presented
(not assuming the existence of a balanced presentation). Indeed, the groups
described in Theorem~\ref{KazhT} are finitely presented Golod-Shafarevich groups
which have property $(T)$ and therefore (FAb) as well.

\begin{Problem1}  
\label{Problem4}
Let $G$ be a GGS pro-$p$ group and $W$ a valuation on $G$
such that $def_W(G)>0$. Does $G$ always have a closed subgroup $H$ of finite
$W$-index such that $H$ can be mapped onto a non-abelian free pro-$p$ group? 
\end{Problem1}

Problem~\ref{Problem4} should be considered as a fancy pro-$p$ analogue of
Baumslag-Pride theorem, as we now explain.

Baumslag-Pride theorem~\cite{BP} asserts that if $G$ is an abstract
group of deficiency at least two (that is, $G$ has a presentation with
two more generators than relators), then $G$ is large. Several people
independently asked if Baumslag-Pride theorem remains true for pro-$p$
groups, that is, if a pro-$p$ group of deficiency at least two has an open
subgroup mapping onto a non-abelian free pro-$p$ group. It is clear that the proof 
of Baumslag-Pride theorem in the abstract case cannot possibly be adapted to pro-$p$ groups.  
The reason is that if $G$ is an abstract group with $def(G)\geq 2$, 
the index of a finite index subgroup $H$ of $G$, which is guaranteed to map onto a 
non-abelian free group, depends on the word length of relators of $G$, and in the 
pro-$p$ case relators may be words of infinite length. 
In fact, most experts believe that the analogue of Baumslag-Pride theorem 
for pro-$p$ groups should be false, although no counterexamples
(or even potential counterexamples) have been constructed.

Problem~\ref{Problem4} is a ``weighted substitute'' for Baumslag-Pride theorem 
for pro-$p$ groups: we consider a larger
class of groups replacing the condition $def(G)\geq 2$ by its weighted 
analogue $def_W(G)>0$, but also relax the assumption on the subgroup $H$,
only requiring finite $W$-index.

We remark that a positive answer to Problem~\ref{Problem4} would yield
a new solution to Zelmanov's theorem about the existence of non-abelian
free pro-$p$ subgroups in Golod-Shafarevich pro-$p$ groups.

\begin{Problem1} 
\label{Problem5}
Let $G$ be a Golod-Shafarevich pro-$p$ group. Is $G$ SQ-universal, that is, does every countably based pro-$p$ group embed into some (continuous) quotient of $G$?
\end{Problem1}

Recall that an abstract group is called SQ-universal if any finitely
generated group (and hence any countable group) embeds into some quotient
of $G$. If $G$ is an abstract (resp. pro-$p$) group which maps
onto a non-abelian free (resp. free pro-$p$) group, then $G$ is obviously
SQ-universal. A result of Hall and Neumann~\cite{Ne} shows that in the abstract case
SQ-universality extends to overgroups of finite index 
(we expect that the same is true for pro-$p$ groups), and therefore abstract
groups of deficiency at least two are SQ-universal by Baumslag-Pride theorem.

While the validity of Baumslag-Pride theorem for pro-$p$ groups is highly
questionable, it is reasonable to conjecture that pro-$p$ groups of
deficiency at least two are still SQ-universal. It is less likely that 
SQ-universality holds for all Golod-Shafarevich groups, 
but we do not see any obvious indications of why this should be false. 
We note that the existence of torsion Golod-Shafarevich abstract groups
means that Problem~\ref{Problem5} would have negative answer in the
category of abstract groups.

\begin{Problem1} 
\label{Problem6}
Find a Golod-Shafarevich group of subexponential subgroup growth.
\end{Problem1}

There is almost no doubt that such groups exist. In fact, we expect that 
Golod-Shafarevich groups with property $(T)$ described in Theorem~\ref{KazhT} 
have subexponential subgroup growth. In any case, it would be interesting to compute (or at least estimate) 
subgroup growth for these groups. In the unlikely case that their
subgroup growth is (at least) exponential, these groups would provide
the first examples of Kazhdan groups with (at least) exponential subgroup
growth.

\begin{Problem1} 
\label{Problem8}
Find an interesting intermediate condition between being virtually Golod-Shafarevich and having positive power $p$-deficiency.
\end{Problem1}

This problem has already been discussed at the end of \S~\ref{sec:powerdef}.

\begin{Problem1} 
\label{Problem10}
Establish new results about Golod-Shafarevich groups in characteristic zero.
\end{Problem1}

Let $\Omega$ be the class of (abstract) groups which are Golod-Shafarevich in
characteristic zero (see \S~\ref{GSzero} for the definition). Recall that every group in $\Omega$
is also Golod-Shafarevich with respect to $p$ for every prime $p$, and it
seems that all known results about groups in $\Omega$ follow from that fact. 
One obvious consequence is that given a group $G$ in $\Omega$,
for every $n\in\dbN$ and every prime $p$ there exists a finite index subgroup
$H=H(n,p)$ of $G$ s.t. $d_p(H)=d(H/[H,H]H^p)\geq n$. It is natural
to ask whether one can find such $H(n,p)$ which is independent of $p$.
Equivalently, is it true that for every $n\in\dbN$, there exists a subgroup
$H=H(n)$ of $G$ s.t. $d(H^{ab})=d(H/[H,H])\geq n$; in other words, does
$G$ have infinite virtual first Betti number? 

The latter question is particularly interesting for free-by-cyclic groups
$F\rtimes \dbZ$ (with $F$ free non-abelian). As mentioned at the end of \S~\ref{GSzero},
a group $G$ of this form is GS in characteristic zero whenever its first Betti
number is at least two (this is equivalent to saying that $G$ maps onto $\dbZ^2$).

\begin{Problem1}
\label{Problem12} Find a ``direct'' proof of non-amenability of Golod-Shafarevich groups.
\end{Problem1}
Recall that in \cite{EJ2}, non-amenability of GS groups
follows from the fact that they possess infinite quotients with property $(T)$ 
which, in turn, depends on the existence of a very concrete family of groups with property 
$(T)$ (described in Theorem~\ref{KMS}) which happen to be GGS with respect to many different weight functions.
While the fact that an infinite group with property $(T)$ is non-amenable is not a deep one,
it does not seem that the groups from Theorem~\ref{KMS} provide the ``real reason'' for
non-amenability of GS groups. 

Finding a proof of non-amenability of GS groups which does not use property $(T)$
is also of interest because it may shed some light on the following question of Vershik~\cite{Ve}
which is still open. 

\begin{Question}
\label{q:Vershik}
Let $G$ be a finitely generated group, let $p$ be a prime, 
let $M$ be the augmentation ideal of the group algebra $\Fp[G]$, and assume that 
the graded algebra $gr{\Fp[G]}=\oplus_{n=0}^{\infty}M^n/M^{n+1}$ has exponential growth. Does
it follow that $G$ is non-amenable?
\end{Question}
Corollary~\ref{GS:divergent}  implies that GS groups satisfy the above hypothesis, 
so a positive answer to Question~\ref{q:Vershik} would
provide a new proof of non-amenability of GS groups. 

\vskip .3cm
Our last problem deals with the Galois groups $G_{K,p,S}$ defined in \S~\ref{sec:nt}.
As explained in \S~\ref{sec:nt}.3, many such groups have an 
LRG (linear rank growth) chain, and it is natural to ask whether every chain
is an LRG chain in those groups. Assuming that $K,p$ and $S$ are such that
$G_{K,p,S}$ is infinite, the following conditions are easily seen to be equivalent:
\begin{itemize}
\item[(a)] $G_{K,p,S}$ has positive rank gradient.
\item[(b)] Any (strictly) descending chain of open normal subgroups of $G_{K,p,S}$
is an LRG chain.
\item[(c)] Let $K=K_0\subset K_1\subset \ldots$ be a (strictly) ascending chain
of finite Galois $p$-extensions of $K$ unramified outside of $S$. Then the sequence
$\{\rho_p(K_n)\}$ of $p$-ranks of the ideal class groups of $K_n$ grows linearly in $[K_n:K]$.
\end{itemize}

\begin{Problem1} 
\label{Problem9}
Assume that the group $G_{K,p,S}$ is infinite and hypotheses of Theorem~\ref{thm:shaf2} hold.
Determine whether the equivalent conditions (a),(b) and (c) above hold or fail
(depending on the triple $(K,p,S)$).
\end{Problem1}

We are not aware  of a single example where the answer to
this question is known. We conjecture that conditions (a),(b),(c) always
fail, that is, the group $G_{K,p,S}$ always has zero rank gradient (under the above
restrictions). This conjecture is based on the various known analogies between 
the groups $G_{K,p,S}$ and hyperbolic $3$-manifold groups (see, e.g., \cite{Rez} and \cite{Mo}).
In particular, similarly to the groups $G_{K,p,S}$,  many hyperbolic $3$-manifold 
groups have LRG  $p$-chains by Theorem~\ref{thm:orb_esg}. At the same time,
very deep recent work of Wise on quasi-convex
hierarchies combined with a theorem of Lackenby~\cite[Theorem~1.18]{La4} implies
that for every hyperbolic $3$-manifold group $G$ and every prime $p$, 
the $p$-gradient of $G$ (equal to the rank gradient of $G_{\phat}$) is zero.


\begin{thebibliography}{AAA}

\bibitem[An1]{An1} D. J. Anick,
\it Noncommutative graded algebras and their Hilbert series. 
\rm J. Algebra 78 (1982), no. 1, 120–-140.

\bibitem[An2]{An2} D. J. Anick,
\it Generic algebras and CW complexes. 
\rm Algebraic topology and algebraic $K$-theory (Princeton, N.J., 1983), 247–-321, 
Ann. of Math. Stud., 113, Princeton Univ. Press, Princeton, NJ, 1987.


\bibitem[AJN]{AJN} M. Abert, A. Jaikin-Zapirain and N. Nikolov,
\it The rank gradient from a combinatorial viewpoint,
\rm Groups Geom. Dyn. 5 (2011), no. 2, 213--230.


\bibitem[BL]{BL} Y. Barnea, M. Larsen,
\it A non-abelian free pro-$p$ group is not linear over a local field. 
\rm J. Algebra 214 (1999), no. 1, 338–-341. 


\bibitem[BaGr]{BaGr} L. Bartholdi and R. Grigorchuk,
\textit{Lie methods in growth of groups and groups of finite width.}
\rm Computational and geometric aspects of modern algebra (Edinburgh, 1998), 1--27,
London Math. Soc. Lecture Note Ser., 275, Cambridge Univ. Press, Cambridge, 2000.

\bibitem[BP]{BP} B. Baumslag and S. Pride,
\it Groups with two more generators than relators. 
\rm J. London Math. Soc. (2) 17 (1978), no. 3, 425–-426.

\bibitem[BHV]{BHV} B. Bekka, P. de la Harpe and A. Valette,
\it Kazhdan's property $(T)$.
\rm New Math. Monogr. 11, Cambridge Univ. Press, Cambridge, 2008

\bibitem[Bu]{Bu} J. Button,
\it Largeness of LERF and 1-relator groups. 
\rm Groups Geom. Dyn. 4 (2010), no. 4, 709–-738. 


\bibitem[BuTh]{BT} J. Button and A. Thillaisundaram,
\it Applications of $p$-deficiency and $p$-largeness,
\rm Internat. J. Algebra Comput. 21 (2011), no. 4, 547–-574.

\bibitem[DDMS]{DDMS} J. D. Dixon, M. P. F. du Sautoy, A. Mann and D. Segal,
\it Analytic pro-$p$ groups.
\rm Second edition. Cambridge Studies in Advanced Mathematics, 61. Cambridge University
Press, Cambridge, 1999.

\bibitem[DJ]{DJ} J. Dymara and T. Januszkiewicz,
\it Cohomology of buildings and their automorphism groups.
\rm Invent. Math.  150  (2002),  no. 3, 579--627.

\bibitem[Ep]{Ep} D.B.A. Epstein,
\it Finite presentations of groups and $3$-manifolds. 
\rm Quart. J. Math. Oxford Ser. (2) 12 (1961), 205–-212.

\bibitem[EJ1]{EJ1} M. Ershov and A. Jaikin-Zapirain,
\it Property (T) for noncommutative universal lattices. 
\rm Invent. Math. 179 (2010), no. 2, 303–-347.

\bibitem[EJ2]{EJ2} M. Ershov,
\it Kazhdan quotients of Golod-Shafarevich groups,
\rm with appendices by A. Jaikin-Zapirain,
Proc. Lond. Math. Soc. 102 (2011), no. 4, 599-636.

\bibitem[EJ3]{EJ3} M. Ershov and A. Jaikin-Zapirain,
\it Groups of positive weighted deficiency,
\rm to appear in J. Reine Angew. Math., arXiv:1007.1489.

\bibitem[Er1]{Er} M. Ershov,
\it Golod-Shafarevich groups with property $(T)$ and Kac-Moody groups,
\rm Duke Math. J. 145 (2008), no. 2, 309--339.

\bibitem[Er2]{Er2} M. Ershov,
\it Kazhdan groups whose $FC$-radical is not virtually abelian,
\rm preprint (2010).

\bibitem[FHT]{FHT} M. Freedman, R. Hain and P. Teichner, 
\it Betti number estimates for nilpotent groups. 
\rm Fields Medallists' lectures, 413–-434, World Sci. Ser. 20th 
Century Math., 5, World Sci. Publ., River Edge, NJ, 1997.

\bibitem[GS]{GS} E. Golod and I. Shafarevich,
\it On the class field tower. (Russian)
\rm Izv. Akad. Nauk SSSR Ser. Mat.  28 (1964), 261--272.

\bibitem[Go1]{Go} E. Golod,
\it On nil algebras and finitely approximable groups. (Russian)
\rm Izv. Akad. Nauk SSSR Ser. Mat.  28 (1964), 273--276.

\bibitem[Go2]{Go2} E. Golod,
\it Some problems of Burnside type. (Russian) 
\rm 1968 Proc. Internat. Congr. Math. (Moscow, 1966) pp. 284-–289.
Izdat. "Mir'', Moscow. 

\bibitem[Gr]{Gr} R. I. Grigorchuk,
\it On Burnside's problem on periodic groups. (Russian) 
\rm Funktsional. Anal. i Prilozhen. 14 (1980), no. 1, 53–-54.

\bibitem[Hab]{Hab} K. Haberland,
\it Galois cohomology of algebraic number fields,
\rm VEB Deutscher Verlag der h: Wissenschaften, Berlin, 1978.


\bibitem[Haj]{Haj} F. Hajir,
\it On the growth of $p$-class groups in $p$-class field towers. 
\rm J. Algebra 188 (1997), no. 1, 256–-271.


\bibitem[Ha]{Ha} P. de la Harpe,
\it Uniform growth in groups of exponential growth. 
\rm Proceedings of the Conference on Geometric and Combinatorial Group Theory, Part II (Haifa, 2000). 
Geom. Dedicata 95 (2002), 1–-17. 

\bibitem[IySh1]{IySh1} N. Iyudu and S. Shkarin,
\it The Golod-Shafarevich inequality for Hilbert series of quadratic algebras 
and the Anick conjecture.
\rm Proc. Roy. Soc. Edinburgh Sect. A 141 (2011), no. 3, 609–-629.  

\bibitem[IySh2]{IySh} N. Iyudu and S. Shkarin,
\it Finite dimensional semigroup quadratic algebras with minimal number of relations,
\rm preprint (2011), arXiv:1104.2029, accepted by Monatshefte f\"ur Mathematik.

\bibitem[KaMe]{KaMe} M. I. Kargapolov and Yu. I. Merzlyakov,
\it Osnovy teorii grupp. (Russian) [Fundamentals of group theory] Fourth edition. 
\rm Fizmatlit ``Nauka'', Moscow, 1996. 288 pp.

\bibitem[Ka]{Ka} M. Kassabov,
\it On pro-unipotent groups satisfying the Golod-Shafarevich condition. 
\rm Proc. Amer. Math. Soc. 131 (2003), no. 2, 329–-336.

\bibitem[Ko1]{Ko} H. Koch,
\textit {Galois theory of $p$-extensions.}
\rm Springer Monographs in Mathematics. Springer-Verlag, Berlin, 2002. xiv+190 pp.

\bibitem[Ko2]{Ko2} H. Koch,
\textit {On $p$-extensions with restricted ramification.}
Appendix 1 in \cite{Hab}.

\bibitem[Kos]{Kos} A. I. Kostrikin,
\it On defining groups by means of generators and defining relations. (Russian) 
\rm Izv. Akad. Nauk SSSR Ser. Mat. 29 (1965), 1119–-1122.

\bibitem[La1]{La1} M. Lackenby,
\it Covering spaces of 3-orbifolds. 
\rm Duke Math. J. 136 (2007), no. 1, 181–-203.

\bibitem[La2]{La2} M. Lackenby,
\it New lower bounds on subgroup growth and homology growth.
\rm Proc. Lond. Math. Soc. (3) 98 (2009), no. 2, 271–-297.

\bibitem[La3]{La3} M. Lackenby,
\it Large groups, property $(\tau)$ and the homology growth of subgroups.
\rm Math. Proc. Cambridge Philos. Soc. 146 (2009), no. 3, 625–-648.

\bibitem[La4]{La4} M. Lackenby,
\it Detecting large groups.
\rm J. Algebra 324 (2010), no. 10, 2636–-2657.

\bibitem[Laz]{Laz} M. Lazard,
\it Groupes analytiques $p$-adiques. (French) 
\rm Inst. Hautes \' Etudes Sci. Publ. Math. 26 (1965), 389–-603.

\bibitem[LR]{LR} D. D. Long and A. Reid,
\it Subgroup separability and virtual retractions of groups,
\rm Topology 47 (2008), 137--159.

\bibitem[Lu1]{Lu1} A. Lubotzky,
\it Group presentation, $p$-adic analytic groups and lattices in ${\rm SL}\sb{2}(C)$.
\rm Ann. of Math. (2)  118  (1983),  no. 1, 115--130.


\bibitem[Lu2]{Lu2} A. Lubotzky,
\it On finite index subgroups of linear groups.
\rm Bull. London Math. Soc. 19 (1987) 325--328.

\bibitem[LuMa1]{LuMa1} A. Lubotzky and A. R. Magid,
\it Cohomology of unipotent and prounipotent groups. 
\rm J. Algebra 74 (1982), no. 1, 76–-95.

\bibitem[LuMa2]{LuMa2} A. Lubotzky and A. R. Magid,
\it Free prounipotent groups. 
\rm J. Algebra 80 (1983), no. 2, 323–-349. 

\bibitem[LuMa3]{LuMa3} A. Lubotzky and A. R. Magid,
\it Cohomology, Poincar\'e series, and group algebras of unipotent groups. 
\rm Amer. J. Math. 107 (1985), no. 3, 531–-553. 

\bibitem[LuMn]{LuMn} A. Lubotzky and A. Mann,
\it Powerful $p$-groups. II. $p$-adic analytic groups. 
\rm J. Algebra 105 (1987), no. 2, 506–-515.

\bibitem[LuZ]{LuZ} A. Lubotzky and A. Zuk,
\it On property ($\tau$),
\rm unpublished book, available at http://www.ma.huji.ac.il/$\sim$alexlub/

\bibitem[Ma]{Ma} W. Magnus,
\it Beziehungen zwischen Gruppen und Idealen in einem speziellen Ring. (German)
\rm Math. Ann. 111 (1935), no. 1, 259–-280.

\bibitem[Mo]{Mo} M. Morishita,
\it On certain analogies between knots and primes. 
\rm J. Reine Angew. Math. 550 (2002), 141–-167.

\bibitem[MyaOs]{MyaOs} A. Myasnikov and D. Osin,
\it Algorithmically finite groups.
\rm J. Pure and Appl. Alg. 215 (2011), no. 11, 2789--2796.

\bibitem[NSW]{NSW}  J. Neukirch, A. Schmidt and K. Wingberg, 
\it Cohomology of number fields. Second edition. 
\rm Grundlehren der Mathematischen Wissenschaften 
[Fundamental Principles of Mathematical Sciences], 323. Springer-Verlag, Berlin, 2008. xvi+825 pp


\bibitem[Ne]{Ne} W. Neumann,
\it The $SQ$-universality of some finitely presented groups. 
\rm Collection of articles dedicated to the memory of Hanna Neumann, I. 
J. Austral. Math. Soc. 16 (1973), 1–-6.


\bibitem[NA]{NA} P. S. Novikov and S. I. Adyan,
\it Infinite periodic groups. I, II, III. (Russian) 
\rm Izv. Akad. Nauk SSSR Ser. Mat. 32 (1968), 212–-244, 251--524, 709--731.


\bibitem[Ol1]{Ol} A. Yu. Ol'shanskii,
\it An infinite group with subgroups of prime orders,
\rm Izvestia Akad. Nauk SSSR, Ser. Matem., 44 (1980), N 2, 309--321 (in Russian);
English translation: Math. USSR-Izv. 16 (1981), no. 2, 279--289.

\bibitem[Ol2]{Ol3} A. Yu. Ol'shanskii,
\it On the question of the existence of an invariant mean on a group. (Russian)
\rm Uspekhi Mat. Nauk 35 (1980), no. 4(214), 199–-200.


\bibitem[Ol3]{Ol2} A. Yu. Ol'shanskii,
\it Geometry of defining relations in groups. 
\rm Translated from the 1989 Russian original by Yu. A. Bakhturin. 
Mathematics and its Applications (Soviet Series), 70. 
Kluwer Academic Publishers Group, Dordrecht, 1991. xxvi+505 pp.


\bibitem[Os1]{Os1} D. Osin,
\it Small cancellations over relatively hyperbolic groups and embedding theorems. 
\rm Ann. of Math. (2) 172 (2010), no. 1, 1–-39.

\bibitem[Os2]{Os} D. Osin,
\it Rank gradient and torsion groups. 
\rm Bull. Lond. Math. Soc. 43 (2011), no. 1, 10–-16.


\bibitem[Pi]{Pi} D. I. Piontkovskii,
\it Hilbert series and relations in algebras. (Russian) 
\rm Izv. Ross. Akad. Nauk Ser. Mat. 64 (2000), no. 6, 205--219; translation in Izv. Math. 64 (2000), no. 6, 1297–-1311.


\bibitem[PP]{PP} A. Polishchuk and L. Positselski,
\it Quadratic algebras. 
\rm University Lecture Series, 37. American Mathematical Society, Providence, RI, 2005. xii+159 pp. 

\bibitem[PV]{PV} S. Popa and S. Vaes,
\it On the fundamental group of ${\rm II}_1$ factors and equivalence relations arising from group actions. 
\rm Quanta of maths, 519--541, Clay Math. Proc., 11, Amer. Math. Soc., Providence, RI, 2010.

\bibitem[Rez]{Rez} A. Reznikov,
\it Three-manifolds class field theory (homology of coverings for a nonvirtually $b_1$-positive manifold). 
\rm Selecta Math. (N.S.) 3 (1997), no. 3, 361–-399.

\bibitem[Ro]{Ro} P. Roquette,
\it On class field towers. 
\rm 1967 Algebraic Number Theory (Proc. Instructional Conf., Brighton, 1965) pp. 231–-249. 


\bibitem[Sch1]{Sch1} A. Schmidt,
\it Circular sets of prime numbers and $p$-extensions of the rationals.
\rm J. Reine Angew. Math. 596 (2006), 115–-130.

\bibitem[Sch2]{Sch2} A. Schmidt,
\it Rings of integers of type $K(\pi,1)$.
\rm Doc. Math. 12 (2007), 441–-471.  

\bibitem[Sch3]{Sch3} A. Schmidt,
\it  \" Uber pro-$p$-fundamentalgruppen markierter arithmetischer kurven. (German) 
[On pro-$p$ fundamental groups of marked arithmetic curves] 
\rm J. Reine Angew. Math. 640 (2010), 203–-235,
English translation available at arXiv:0806.1863


\bibitem[Sh]{Sh} I. Shafarevich,
\it Extensions with prescribed ramification points. (Russian) 
\rm Inst. Hautes \'Etudes Sci. Publ. Math. No. 18 (1963), 71–-95.

\bibitem[SP]{SP} J.-C. Schlage-Puchta,
\it A $p$-group with positive rank gradient.
\rm J. Group Th. 15 (2012), no. 2, 261--270.

\bibitem[Uf]{Uf} V. Ufnarovskii, 
\it Combinatorial and asymptotic methods in algebra (Russian), 
\rm Current problems in mathematics. Fundamental directions 57, 5–-177, 
Itogi Nauki i Tekhniki, Akad. Nauk SSSR, Moscow, 1990

\bibitem[Ve]{Ve} A. Vershik,
\it Amenability and approximation of infinite groups. 
\rm Selected translations. Selecta Math. Soviet. 2 (1982), no. 4, 311–-330.

\bibitem[Vi]{Vi} E. B. Vinberg,
\it On the theorem concerning the infinite-dimensionality of an associative algebra. (Russian) 
\rm Izv. Akad. Nauk SSSR Ser. Mat. 29 (1965), 209–-214.

\bibitem[Vo]{Vo} T. Voden,
\it Subalgebras of Golod-Shafarevich algebras. 
\rm Internat. J. Algebra Comput. 19 (2009), no. 3, 423–-442.


\bibitem[Wi1]{Wi1} J. Wilson,
\it Finite presentations of pro-$p$ groups and discrete  groups.
\rm Invent. Math. 105 (1991), no. 1, 177--183.


\bibitem[Wi2]{Wi} J. Wilson,
\it Profinite groups.
\rm  London Mathematical Society Monographs. New Series, 19.
The Clarendon Press, Oxford University Press, New York, 1998.



\bibitem[Wis1]{Wis} J. Wisliceny,
\it Zur Darstellung von Pro-$p$-Gruppen und Lieschen Algebren durch Erzeugende und Relationen. (German) 
[On the presentation of pro-$p$-groups and Lie algebras by generators and relations] 
\rm Math. Nachr. 102 (1981), 57–-78

\bibitem[Wis2]{Wis2} J. Wisliceny,
\it Konstruktion nilpotenter assoziativer Algebren mit wenig Relationen. (German) 
[Construction of nilpotent associative algebras with few relations] 
\rm Math. Nachr. 147 (1990), 75–-82


\bibitem[Ze1]{Ze1} E. Zelmanov,
\textit {On groups satisfying the Golod-Shafarevich condition.}
\rm New horizons in pro-$p$ groups,  223--232, Progr. Math., 184, Birkh\"auser Boston, Boston, MA, 2000

\bibitem[Ze2]{Ze2} E. Zelmanov,
\it Some open problems in the theory of infinite dimensional algebras. 
\rm J. Korean Math. Soc. 44 (2007), no. 5, 1185–-1195.

\end{thebibliography}
\end{document}